\documentclass[oneside, 11pt]{article}
\usepackage{setspace}
\usepackage{titlesec}
\usepackage[utf8]{inputenc}
\usepackage{tabularx}	
\usepackage{etoolbox}
\usepackage{mathtools}
\usepackage{amsmath}
\usepackage{amssymb} 
\usepackage{amsfonts}
\usepackage{mathrsfs} 
\usepackage{amsthm,pifont}
\usepackage{graphicx}
\usepackage{bbm}
\usepackage[colorlinks=true, pdfstartview=FitV, linkcolor=blue,
citecolor=blue, urlcolor=blue, pagebackref=false]{hyperref}
\usepackage{cleveref}
\usepackage[nottoc,numbib]{tocbibind}
\usepackage{xcolor}
\usepackage{tikz}
\usepackage{setspace}
\usepackage{cite}
\usepackage{braket}
\usepackage{leftidx,tensor}
\usepackage{dsfont}
\usepackage{calc}
\usepackage{caption}
\usepackage{subcaption}
\usepackage{color}
\usetikzlibrary{decorations.markings}
\usetikzlibrary{shapes.geometric}
\usetikzlibrary{patterns}
\tikzstyle{vertex}=[circle, draw, inner sep=0pt, minimum size=6pt]

\usepackage[utf8]{inputenc}
\usepackage[english]{babel}
\usepackage{enumitem}

\allowdisplaybreaks

\newtheorem{theorem}{Theorem}[section]
\makeatletter
\patchcmd{\ttlh@hang}{\parindent\z@}{\parindent\z@\leavevmode}{}{}
\patchcmd{\ttlh@hang}{\noindent}{}{}{}
\makeatother
\titleformat*{\section}{\large\bfseries}
\titleformat*{\subsection}{\small\bfseries}
\titleformat*{\subsubsection}{\small\bfseries}
\titleformat*{\paragraph}{\small\bfseries}
\titleformat*{\subparagraph}{\small\bfseries}

\newcommand{\N}{\mathbb{N}}
\newcommand{\R}{\mathbb{R}}
\newcommand{\Z}{\mathbb{Z}}
\newcommand{\E}{\mathbb{E}}
\renewcommand{\H}{\mathbb{H}}
\newcommand{\p}{\mathbb{P}}

\newcommand{\eps}{\varepsilon}

\newcommand{\cA}{\mathcal{A}}

\newcommand{\cD}{\mathcal{D}}

\newcommand{\cF}{\mathcal{F}}

\newcommand{\cJ}{\mathcal{J}}

\newtheorem{lemma}[theorem]{Lemma}
\newtheorem{remark}[theorem]{Remark}
\newtheorem{proposition}[theorem]{Proposition}

\newtheorem{claim}[theorem]{Claim}

\newtheorem{observation}[theorem]{Observation}
\theoremstyle{remark}
\newtheorem{definition}[theorem]{Definition}

\usepackage{geometry}
\begin{document}

	\title{A discontinuous percolation phase transition on the hierarchical lattice}
	
	\author{Johannes B\"aumler\thanks{Department of Mathematics, University of California, Los Angeles; \url{jbaeumler@math.ucla.edu}}
    \and
    Tom Hutchcroft\thanks{The Division of Physics, Mathematics and Astronomy, California Institute of Technology; \\  \url{t.hutchcroft@caltech.edu}}
    }

	\maketitle
	
	\begin{center}
		\parbox{13cm}{ \textbf{Abstract.} 
        For long-range percolation on $\Z$ with translation-invariant edge kernel $J$, it is a classical theorem of Aizenman and Newman (1986) that the phase transition is discontinuous when $J(x,y)$ is of order $|x-y|^{-2}$ and that there is no phase transition at all when $J(x,y)=o(|x-y|^{-2})$. We prove a strengthened version of this theorem for the hierarchical lattice, where the relevant threshold is at $|x-y|^{-2d} \log\log |x-y|$ rather than $|x-y|^{-2}$: There is a continuous phase transition for kernels of larger order, a discontinuous phase transition for kernels of exactly this order, and no phase transition at all for kernels of smaller order. As such, $|x-y|^{-2d} \log\log |x-y|$ is essentially the \emph{only} kernel that produces a discontinuous phase transition.  We also prove a hierarchical analogue of the ``$M^2\beta=1$'' conjecture of Imbrie and Newman (1988), which gives an exact formula for the density of the infinite cluster at the point of discontinuous phase transition and remains open in the Euclidean setting.
	}
	\end{center}
	
	\vspace{0.1cm}

\setstretch{1.1}

\section{Introduction}\label{sec:introduction}

Let $J:\Z\times \Z \to \left[0,\infty \right)$ be a symmetric, translation-invariant kernel, meaning that $J(x,y)=J(y,x)=J(0,y-x)$ for all $x,y\in \Z$, and suppose that $J$ is integrable in the sense that $\sum_{x\in \Z} J(0,x) < \infty$. Consider long-range percolation on $\Z$, where for any pair of distinct vertices $x,y\in \Z$ the edge $\{x,y\}$ is open with probability $1-\exp(\lambda J(x,y))$, independent of all other edges, where $\lambda \geq 0$ is a parameter. We denote the resulting measure by $\p_{\lambda}=\p_{\lambda,J}$. 
When the function $J$ is of the form $J(x
,y)=\|x-y\|^{-1-\alpha}$ with $\alpha>0$, the associated long-range percolation model has a non-trivial phase transition in the sense that $\lambda_c=\inf\{\lambda \geq 0:$ an infinite cluster exists $\p_{\lambda}$-a.s.$\}$ satisfies $0<\lambda_c<\infty$ if and only if $\alpha \in (0,1]$ \cite{newman1986one,duminil2020long}.
The boundary case $\alpha=1$ has the special feature that the model has a \emph{discontinuous phase transition} \cite{MR868738,duminil2020long}, meaning that there exists an infinite cluster almost surely at the critical point $\lambda=\lambda_c$, whereas for $\alpha<1$ the phase transition is \emph{continuous}  \cite{MR1896880,hutchcroft2020power}, meaning that there do not exist any infinite clusters at the critical point $\lambda=\lambda_c$ almost surely. (Moreover, the critical behaviour of the model with $d=1$ and $\alpha\in (0,1)$ is now very well understood following the works \cite{hutchcroft2022sharp,baumler2022isoperimetric,LRPpaper1,LRPpaper2,LRPpaper3} as surveyed in \cite{hutchcroft2025dimension}.) In addition to possessing a discontinuous phase transition, the kernel $J(x,y)=|x-y|^{-1-\alpha}$ on $\Z$ with $\alpha=1$ is also distinguished by its approximate \emph{self-similarity} (the Poisson process of edges on $\R \times \R \setminus \{(x,x):x\in \R\}$ with intensity $\lambda |x-y|^{-2}$ being exactly invariant under rescaling), which also leads to the infinite \emph{supercritical} cluster having a rich, fractal-like geometry as studied in e.g.\ \cite{ding2023uniqueness,baumler2023distances,ding2025uniqueness}.

The discontinuity of the phase transition for the self-similar one-dimensional long-range percolation (i.e., the model with $d=\alpha=1$) is proven by establishing \cite[Proposition 4.1]{newman1986one} more generally for any translation-invariant kernel $J:\Z\times \Z\to [0,\infty)$ that
\begin{equation}
\label{eq:M2beta_gap}
\p_{\lambda,J}(0\leftrightarrow\infty)^2 \left(\limsup_{|x-y|\to \infty} \lambda |x-y|^{2}J(x,y)\right) \in \{0\} \cup [1,\infty],
\end{equation}
yielding an explicit lower bound of $(\limsup_{|x-y|\to \infty} \lambda_c|x-y|^2 J(x,y))^{-1/2}$ on the density of the infinite cluster at the point of the discontinuous phase transition. It was conjectured by Imbrie and Newman \cite{imbrie1988intermediate} that this lower bound is in fact an \emph{equality}, so that
\begin{equation}
\p_{\lambda_c,J}(0\leftrightarrow\infty) = \left(\lim_{|x-y|\to \infty} \lambda_c |x-y|^2 J(x,y)\right)^{-1/2},
\label{eq:ImbrieNewman}
\end{equation}
under the assumption that the relevant limit is well-defined and in $(0,\infty)$. That is, the Imbrie--Newman conjecture yields a simple exact formula for the density of the infinite percolation cluster for self-similar long-range percolation on $\Z$ at the point of discontinuous phase transition as a function of the asymptotic intensity of long edges. This conjecture remains open to this day, nearly four decades after it was posed.

\medskip

\noindent \textbf{A discontinuous phase transition in hierarchical percolation.}
The goal of this paper is to explore analogous phenomena in \emph{hierarchical percolation}, for which we will show that a much more comprehensive analysis is possible.
Broadly speaking, hierarchical models 
are toy models of statistical mechanics introduced by Dyson \cite{MR436850} and Baker~\cite{baker1972ising} that have often served as a testing ground to develop techniques and build insights that are later adapted to Euclidean models \cite{bleher2010critical,abdesselam2013rigorous,MR3969983}, as has recently been done for long-range percolation in the works \cite{hutchcroft2022sharp,hutchcroft2022critical,LRPpaper1,LRPpaper2,LRPpaper3,baumler2022isoperimetric,hutchcrofthierarchical}. Hierarchical percolation is defined exactly as we defined long-range percolation on $\Z$ except that the Euclidean metric on $\Z$ is replaced by a \emph{hierarchical ultrametric} in which we decompose the lattice hierarchically into $L$-adic boxes and (by abuse of notation) set $\|x-y\|$ (which is not really a norm or a function of $x-y$)  to be the side length of the smallest box in the decomposition containing both $x$ and $y$. (See \Cref{subsec:definitions} for other equivalent definitions and further relevant background on the model.)

In hierarchical percolation with kernel $J(x,y)=\|x-y\|^{-d-\alpha}$, it is known that the phase transition is continuous when $\alpha<d$ \cite{MR2955049} (in which case the critical behaviour is also very thoroughly understood \cite{hutchcroft2022critical,hutchcrofthierarchical}) and that there is no phase transition at all when $\alpha \geq d$ \cite{MR2955049,MR3035740}. 
We show that one can recover a discontinuous phase transition similar to the one observed on $\Z$ by instead using a kernel of order
\[
J(x,y) \asymp \frac{\log\log \|x-y\|}{\|x-y\|^{2d}}.
\]
Moreover, we show that (subject to mild regularity assumptions) this is the \emph{only} form of the kernel in which a discontinuous phase transition takes place: More concretely, we show that if the kernel is divergently smaller than $\|x-y\|^{-2d}\log\log \|x-y\|$ then there is no phase transition, while if it is divergently larger than $\|x-y\|^{-2d}\log\log \|x-y\|$ then the phase transition is continuous. We conjecture that the analogous statement is also true on $\Z$ with the critical kernel being $|x-y|^{-2}$; the only part of this conjecture that remains open following the results of \cite{MR868738,duminil2020long} is to show that the phase transition is continuous under the minimal assumption that $J(x,y) = \omega(|x-y|^{-2})$.

Finally, we are also able to prove a hierarchical analogue of the  Imbrie--Newman conjecture \eqref{eq:ImbrieNewman},  yielding an exact formula for the density of the infinite cluster at the point of discontinuous phase transition; this is the most technical part of the paper and relies on a much more delicate analysis than the proofs of (dis)continuity. 

 \medskip
 
 Precise statements of our theorems are given in \Cref{subsec:theorems} following the detailed definition of the hierarchical model in \Cref{subsec:definitions}.

\subsection{Definition of the model}
\label{subsec:definitions}

We now formally define the class of models we will consider.
It will be convenient to work with a different (but equivalent) definition of the hierarchical lattice than that given above; we refer to 
\cite[Section 2.1]{hutchcroft2022sharp} for a detailed explanation of how the two definitions lead to isometric metric spaces.
For a dimension $d\in \N$ and side-length $L\in \{2,3,4,\ldots\}$, the \textbf{hierarchical lattice} $\H_L^d$ is defined to be the abelian group $\bigoplus_{i=1}^\infty \left(\Z / L\Z \right)^d$ equipped with the ultrametric given by $\|x-y\| \coloneqq L^{\max \{i:x_i \neq y_i\}}$ for all distinct $x,y\in \H_L^d$. (This metric is not a norm, but we use this notation to emphasize its analogy with the metrics on $\Z^d$ induced by norms on $\R^d$.) The ultrametric balls of radius $L^n$ in this space are referred to as $n$-\textbf{blocks}, with the $n$-block containing the origin denoted by $\Lambda_n$ and the $n$-block containing the point $x$ denoted by $\Lambda_n(x)$. As a metric space, $\H_L^d$
can also be constructed recursively by taking $\Lambda_0=\{0\}$ and, for each $n\geq 0$, taking $\Lambda_{n+1}$ to be the union of $L^d$ disjoint copies of $\Lambda_n$ with distances defined by $\|x-y\| = L^{n+1}$ for each pair $x,y\in \Lambda_{n+1}$ belonging to distinct copies of $\Lambda_{n}$. 

We now define long-range percolation on the hierarchical lattice. It will be convenient to use a slightly different (and more general) formulation of the model than in the opening paragraphs of the paper, in which we vary the kernel $J$ itself rather than varying a constant prefactor $\lambda$.
We think of the hierarchical lattice as a weighted graph with vertex set given by the group $V=\bigoplus_{i=1}^\infty \left(\Z / L\Z \right)^d$ and edge set $E=\{\{x,y\}:x,y\in V, x\neq y\}$. Given an edge $e=\{x,y\}$, we also write $|e|=\|x-y\|$ for its length. Let $J:E\to \left[0,\infty\right)$ be a kernel. We define percolation on $\H_L^d$ with kernel $J$ to be the random graph with vertex set $\H_L^d$ in which each potential edge $e$ is open with probability $1-\exp\left(-J(e)\right)$, independent of all other edges. The law of the resulting random graph is denoted $\p_J$. We are mostly interested in the case where kernels are \textbf{isometry-invariant}, meaning that $J(e)=J(f)$ for all $e,f\in E$ with $|e|=|f|$, \textbf{integrable}, meaning that $\sum_{e : 0 \in e} J(e) < \infty$, and \textbf{regular}, meaning that there exists a constant $C$ such that 
\begin{equation*}
	\sum_{x \notin \Lambda_n} J(0,x) \leq C \sum_{x \in \Lambda_{n+1}\setminus\Lambda_{n}} J(0,x)
\end{equation*}
for all large enough $n \in \N$, where we write $J(x,y)=\mathbbm{1}(x\neq y)J(\{x,y\})$ for each pair of vertices $x,y\in \mathbb{H}^d_L$. The integrability condition guarantees that the long-range percolation configuration is locally finite almost surely. 

We will also  consider parametrized families of kernels $\cJ:\left[0,\infty\right) \times E \to \left[0,\infty\right)$. Again, for $\lambda \in \left[0,\infty\right)$, we say that an edge $e$ is open with probability $1-\exp\left(-\cJ(\lambda,e)\right)$, independent of all other edges. We write $\p_{\cJ(\lambda)}$ for the resulting measure. If it is clear which family of kernels is considered, we also write $\p_{\lambda}$ for $\p_{\cJ(\lambda)}$. 
In this paper, we only consider the setting in which the kernels $\cJ(\lambda,\cdot)$ are isometry-invariant, integrable, and regular for all $\lambda \in \left(0,\infty\right)$.
In addition, we will also require that the following further properties hold regarding the dependence of $\mathcal{J}(\lambda)$ on the parameter $\lambda$: We say that a family of kernels $\cJ:\left[0,\infty\right) \times E \to \left[0,\infty\right)$ is \textbf{monotone} if $\cJ(\lambda,e)$ is (weakly) increasing in $\lambda$ for all $e\in E$, \textbf{continuous} if for each $\lambda^\star>0$ the functions
\begin{equation}\label{proper:conti}
    \lambda \mapsto \sup_{e\in E} \frac{\cJ(\lambda,e)}{\cJ(\lambda^\star,e)} \ \text{ and } \ \lambda \mapsto \sup_{e\in E} \frac{\cJ(\lambda^\star,e)}{\cJ(\lambda,e)}
\end{equation}
are finite-valued, continuous functions of $\lambda \in (0,\infty)$, and \textbf{non-percolative at $0$} if $\sum_{e : 0\in e} \cJ(\lambda,e)<1$ for all sufficiently small $\lambda \geq 0$.
The continuity condition guarantees that the kernels $\cJ(\lambda,\cdot)$ and $\cJ(\lambda+\eps,\cdot)$ behave similarly for $|\eps|$ small, while the monotonicity condition guarantees that $\p_{\cJ(\lambda)}(e \text{ is open})$ is non-decreasing in $\lambda$. In particular, the monotonicity assumption guarantees that we can couple the measures $(\p_{\cJ(\lambda)})_{\lambda \geq 0}$ using the standard Harris coupling so that we only add edges to the graph when increasing $\lambda$. Finally, the ``non-percolative at $0$'' condition ensures,  by comparison with a subcritical branching process, that for all sufficiently small $\lambda$, there are only finite open clusters under the measure $\p_\lambda$ almost surely.  
We say that a family of kernels $\cJ$ is \textbf{proper} if the family of kernels is monotone, continuous, non-percolative at $0$, and if for every $\lambda > 0$ the kernel $\cJ(\lambda,\cdot)$ is isometry-invariant, integrable, and regular.

Let us discuss two natural examples of families of kernels that satisfy these assumptions.  First, when $\cJ$ is of the form $\cJ(\lambda,e)= \lambda J(e)$ for some kernel $J$, then it is proper whenever $J$ is isometry-invariant, integrable, and regular. Increasing $\lambda$ for this kernel increases the probability $\p_\lambda(e \text{ open})$ for all edges $e$ with $J(e)>0$. Given an isometry-invariant, integrable, regular kernel $J$, another example of a family of kernels that satisfies our assumptions is given by
\begin{equation*}
	\cJ(\lambda,e) = \begin{cases}
	J(e) & \text{ if } |e|>N \\
	\lambda & \text{ if } |e|\leq N
	\end{cases}
\end{equation*}
where $N\in \N$ is a fixed cut-off value large enough that the ``non-percolative at $0$'' condition  is satisfied. For this family of kernels, an increase in $\lambda$ increases the probability $\p_\lambda(e \text{ open})$ only for edges $e$ with length at most $N$.

\subsection{Statement of main theorems}
\label{subsec:theorems}

We now state our main theorems. We begin with our classification of hierarchical kernels producing continuous and discontinuous phase transitions, and then state our hierarchical analogue of the Imbrie-Newman conjecture. Finally, we state a theorem giving a sharp analysis of which kernels can produce phase transitions at all and explain how this result underpins all our other results.

\medskip

\noindent \textbf{Discontinuous phase transitions occur only for $J\approx \|x-y\|^{-2d}\log\log\|x-y\|$.}
Our first main result concerns the nature of the percolation phase transition when it is non-trivial. Suppose $\mathcal{J}$ is a proper family of kernels.
As the existence of an infinite open cluster is a monotone property, we can define the critical value
\begin{align*}
	\lambda_c = \lambda_c(\cJ) = \inf \left\{ \lambda \geq 0 : \p_{\cJ(\lambda)} \left( |K_0| = \infty \right)> 0 \right\},
\end{align*}
where $K_0$ denotes the cluster of the origin. The critical value $\lambda_c$ is positive by the assumption that $\cJ$ is non-percolative at $0$.
By translation-invariance, the almost sure existence of an infinite open cluster is equivalent to the origin having a positive probability of being in an infinite cluster. Note also that the infinite cluster is a.s.\ unique when it exists for any isometry-invariant kernel $J$, since $(\H_L^d,J)$ is an amenable transitive weighted graph.
For families of kernels $\cJ$ for which a phase transition exists (i.e., for which $\lambda_c(\cJ) \in (0,\infty)$), we are interested in the question of whether an infinite cluster exists for $\lambda = \lambda_c$. We prove that, subject to mild regularity assumptions, this is determined by whether $\cJ$ is of larger, smaller, or equal order to the kernel $\|x-y\|^{-2d}\log\log\|x-y\|$.

Before stating this theorem, let us first introduce some relevant asymptotic notation.
Given a function $f:\N\to \R_{\geq 0}$ and a family of kernels $\cJ$, we write $\cJ\approx f$ if
\begin{align*}
0<\liminf_{|e|\to \infty} \frac{\cJ(\lambda,e)}{f(|e|)} \leq \limsup_{|e|\to \infty} \frac{\cJ(\lambda,e)}{f(|e|)} < \infty
\end{align*}
for every $\lambda > 0$. (If the family is continuous, this estimate holds for every $\lambda > 0$ if and only if it holds for $\lambda=1$.) We also write $\cJ \ll f$ if $\limsup_{|e|\to \infty} \cJ(\lambda,e)/f(|e|)=0$ for every $\lambda>0$ and 
$\cJ \gg f$ if $\liminf_{|e|\to \infty} \cJ(\lambda,e)/f(|e|)=\infty$ for every $\lambda>0$.

\begin{theorem}\label{theo:(dis)continuity}
	 Let $\cJ : \left[0,\infty\right) \times E \to \left[0,\infty \right)$ be a proper family of kernels on the hierarchical lattice $\mathbb{H}^d_L$.
     \begin{enumerate}
     \item \emph{(Slow decay $\Rightarrow$ Continuity.)} If $\cJ \gg n^{-2d}\log\log n$ and $\lambda_c<\infty$ then the phase transition is \emph{continuous} in the sense that  $\p_{\lambda_c} ( |K_0| = \infty ) = 0$.
        \item  \emph{(Fast decay $\Rightarrow$ Triviality.)} If $\cJ \ll n^{-2d}\log\log n$ then there is \emph{no phase transition} in the sense that $\lambda_c=\infty$.
  \item \emph{(Critical decay $\Rightarrow$ Discontinuity.)} If $\cJ\approx n^{-2d}\log\log n$ and $\lambda_c<\infty$ then the phase transition is \emph{discontinuous} in the sense that $\p_{\lambda_c} ( |K_0| = \infty ) > 0$.
  \end{enumerate}
\end{theorem}

We will return to the question of when $\lambda_c<\infty$ below.

As stated earlier, we conjecture that an analogous theorem holds on $\Z$ with $\|x-y\|^{-2d}\log\log\|x-y\|$ replaced by $|x-y|^{-2}$. Following \cite{MR868738,duminil2020long}, the only part of this conjecture that remains open is to show that the phase transition is continuous under the minimal assumption that $\cJ \gg n^{-2}$ (it appears that all existing methods  \cite{MR1896880,hutchcroft2020power,hutchcroft2022sharp} require an additional logarithmic factor in the kernel).

The part of \Cref{theo:(dis)continuity} concerning continuity was previously established for kernels of the form $\cJ(\lambda,e) = \lambda|e|^{-d-\alpha}$ for $\alpha \in (0,d)$ by Koval, Meester, and Trapman \cite{MR2955049}; our proof follows an optimized version of their strategy.
More concretely, our proof of continuity follows the ``supercritical strategy'' \cite[Section 1.1]{hutchcroft2020power}: We show that the set $\left\{\lambda : \p_{\lambda} \left( |K_0| = \infty \right) > 0\right\}$ is open by proving that if an infinite cluster exists for some $\lambda>0$ then an infinite cluster still exists at $\lambda-\eps$ for small enough $\eps>0$. To prove this, we show that a certain coarse-grained version of the model must be \emph{highly-supercritical} whenever the original model has an infinite cluster. (Alternative proofs of continuity for the kernels $\lambda |e|^{-d-\alpha}$ with $\alpha\in (0,d)$ using the \emph{subcritical strategy} are given in \cite{hutchcroft2020power,hutchcrofthierarchical}. These proofs give stronger quantitative conclusions when they apply, but do not seem to generalize to the class of kernels we consider here.) Compared to \cite{MR2955049}, the main technical improvement is our sharper analysis of which kernels admit a phase transition at all, which is an important input to this coarse-graining argument.

\medskip

\noindent \textbf{The hierarchical analogue of the Imbrie--Newman conjecture.}
Analogously to the Euclidean arguments discussed around \eqref{eq:M2beta_gap}, the discontinuity part of \Cref{theo:(dis)continuity} is proven by showing that
\begin{equation}
\p_{\lambda} ( 0\leftrightarrow \infty)^2 \left(\limsup_{e\to \infty} \cJ(\lambda,e) \frac{L^{-d}|e|^{2d}}{\log\log |e|}\right) \in \{0\} \cup [1,\infty],
\end{equation}
see Proposition \ref{propo:discont} below. The function $\theta(\lambda) \coloneqq \p_{\lambda} (0\leftrightarrow\infty)$ is right continuous since it is the infimum of the continuous increasing functions $\p_\lambda( |K_0|\geq n)$.
Writing $\beta^*(\lambda) = \limsup_{e\to \infty}L^{-d} \cJ(\lambda,e) |e|^{2d} (\log\log |e|)^{-1}$, this implies  that $\theta(\lambda_c) \geq \beta^*(\lambda_c)^{-1/2}>0$ when $\beta^*(\lambda_c)<\infty$.
Our next main theorem states that this is in fact an \emph{equality} under the assumption that the limit $\beta(\lambda_c) \coloneqq \limsup_{e\to \infty}L^{-d} \cJ(\lambda_c,e) |e|^{2d} \left(\log\log |e|\right)^{-1}$ is well-defined; this is the natural hierarchical analogue of the aforementioned conjecture of Imbrie and Newman \eqref{eq:ImbrieNewman}.

\begin{theorem}\label{theo:ImbrieNewman}
	Let $\cJ : \left[0,\infty\right) \times E \to \left[0,\infty \right)$ be a proper family of kernels with $\lambda_c(\cJ)<\infty$. If the limit $\beta = \lim_{|e|\to \infty} L^{-d}\cJ(\lambda_c,e) |e|^{2d} \left(\log\log|e|\right)^{-1}$
	is well-defined and finite, then
        $\p_{\lambda_c} \left( 0\leftrightarrow  \infty \right)^2 \beta = 1$.
\end{theorem}

Parts $1$ and $2$ of \Cref{theo:(dis)continuity} correspond formally to the $\beta=\infty$ and $\beta=0$ cases of this theorem, respectively.
\Cref{theo:ImbrieNewman} is significantly more difficult to prove than \Cref{theo:(dis)continuity} as we outline in more detail momentarily.

\medskip

\noindent \textbf{Sharp analysis of the (non)triviality of the phase transition.}
Underlying both \Cref{theo:(dis)continuity,theo:ImbrieNewman} is a sharp analysis of which translation-invariant kernels can produce infinite clusters.
For an isometry-invariant kernel, the weight of an edge depends only on the distance of its endpoints, so we may also consider the kernel as a function $J:\{L, L^2, L^3,\ldots\} \to [0,\infty)$ defined by $J(|e|)=J(e)$. 
Given a function $f:\N\to [0,\infty)$ and a kernel $J$, we say that $J\sim f$ if $\lim_{|e|\to \infty} J(e)/f(|e|) = 1$.
Roughly speaking, we prove that the transition between models that can or cannot have an infinite open cluster occurs precisely when $J \sim L^d n^{-2d} \log\log n$. For \Cref{theo:(dis)continuity}, it is sufficient to prove that this transition occurs on the \emph{order} $n^{-2d} \log\log n$, which is relatively straightforward, while \Cref{theo:ImbrieNewman} requires the exact determination of the leading constant, which is significantly more delicate.

\begin{theorem}\label{theo:existence}
	For all $d,L\in \N$ and $a>1$ there exists an isometry-invariant kernel $J$ on $\mathbb{H}^d_L$ satisfying 
	   \vspace{-0.5em} \begin{align*}
		J \sim \frac{a L^d \log\log n}{n^{2d}}
	\end{align*}
	such that $\p_{J} \left( |K_0| = \infty \right)> 0$. 
	On the other hand, if $J$ is an isometry-invariant kernel on $\H^d_L$ such that
	\begin{align*}
	\limsup_{n\to \infty}
	 \frac{L^{-d} \left( L^n \right)^{2d}}{\log\log \left( L^n \right)} J\!\left(L^n\right) \leq 1
	\end{align*}
	then $\p_{J} \left( |K_0| = \infty \right) = 0$. 
\end{theorem}

This result complements and sharpens the results of Dawson and Gorostiza \cite{MR3035740} who showed that  $\p_{J} \left( |K_0| = \infty \right) = 0$ for every kernel $J$ with $J \sim a L^{d} n^{-2d} \log\log n$ with $a < 1$; our main contribution is to prove that the phase transition can be non-trivial for any value $a>1$. Previously, the existence of an infinite open cluster under $\lambda J$ for sufficiently large $\lambda$ was established for kernels $J$ satisfying $J\sim  n^{-2d} (\log n)^\alpha$ with $\alpha>0$
\cite{MR3035740,MR3769822}
and for $J\sim n^{-d-\alpha}$ with $\alpha \in (0,d)$ \cite{MR2955049} (see also \cite{hutchcrofthierarchical,hutchcroft2020power}).  

\begin{remark}
The expected cluster size for the hierarchical model with kernel $\lambda \|x-y\|^{-2d}$ was shown in \cite{easo2024double} to have order $\exp(\exp(\Theta(\lambda)))$ as $\lambda \to \infty$, a fact that is presumably related to the appearance of the $\log\log\|x-y\|$ terms in our results. The sharper analysis of
\Cref{theo:ImbrieNewman,theo:existence} should be related to improving the results of \cite{easo2024double} to compute a precise constant $C$ such that $\log\log \E_\lambda |K| \sim C \lambda$ as $\lambda\to \infty$ for the hierarchical model with kernel $\lambda\|x-y\|^{-2d}$.
\end{remark}

\medskip

For the proof of Theorem \ref{theo:ImbrieNewman}, we will require not just Theorem \ref{theo:existence}, but the following slightly stronger theorem concerning a mixed site-bond model. Such models arise naturally in the context of coarse-grained renormalization arguments. Here we write $\p_{q,J}$ for the random graph in which each vertex is included independently at random with probability $q\in [0,1]$ and, independently, each edge between open vertices is included independently at random with probability $1-\exp(-J(e))$.

 \begin{theorem}\label{theo:existence mixed}
	For all $L\in \N, d \in \N$, and $a>1$ there exists $q\in (0,1)$ and $N_0, N_1 \in \N$  such that site-bond percolation on the hierarchical lattice $\H_L^{d}$ with the kernel $J$ defined by 
	\begin{align*}
	J(e) = \begin{cases}
		 \frac{a L^d \log\log(|e|)}{|e|^{2d}} & \text{ for } |e| \geq N_1\\
		 N_0 & \text{ for } |e| < N_1
		 \end{cases}
	\end{align*}
	satisfies $\p_{q,J} \left( |K_0| = \infty \right)> 0$.
\end{theorem}

\noindent \textbf{Organization.} 
We will first prove a weaker version of Theorem \ref{theo:existence mixed} in Section \ref{sec:large a}. Using this weaker version, Theorem \ref{theo:existence large a}, we prove the continuity of the phase transition for families of kernels with slow decay (Theorem \ref{theo:(dis)continuity}, part 1) in Section \ref{sec:continuity}. After this, we will prove the triviality and discontinuity of the phase transition, i.e., Theorem \ref{theo:(dis)continuity} parts 2 and 3, for kernels of fast and critical decay, respectively, in Section \ref{sec:discontin}. 
We will prove Theorem \ref{theo:existence mixed} in Section \ref{sec:pt for large a}, and conclude the proof of Theorem \ref{theo:ImbrieNewman} in Section \ref{sec:imbrienewman}.

Note that it suffices to prove all the aforementioned results for dimension $d=1$ thanks to the following observation:

\begin{observation}
	Suppose that the results of Theorems \ref{theo:(dis)continuity}, \ref{theo:ImbrieNewman}, \ref{theo:existence}, and \ref{theo:existence mixed} hold for $d=1$ and all $L\in \N$. Then these results hold for all $L,d \in \N$.
\end{observation}

This follows from the fact that the space $\H_L^d$ can be identified with the space $\H_{L^d}^1$ by a map $\phi: \H_L^d \to \H_{L^d}^{1}$ such that $\|\phi(x)-\phi(y)\| = \|x-y\|^{d}$ for all $x,y \in \H_{L}^{d}$. In particular, the model on $\H_L^d$ with kernel 
\[
J(x,y)=\frac{a L^d \log\log \|x-y\|}{\|x-y\|^{2d}} \vee 0
\]
is equivalent to the model on $\H_{L^d}^1$ with kernel
\[
\tilde J(u,v)=\frac{a L^d \log\log (\|u-v\|^{1/d})}{\|u-v\|^{2}} \vee 0 = (1+o(1)) \frac{a L^d \log\log \|u-v\|}{\|u-v\|^{2}}
\]
as $\|u-v\|\to \infty$.
We can therefore restrict to the case $d=1$ in the rest of this paper, which will lighten notation. We will also just write $\H_L$ instead of $\H_L^1$.

\section[Existence of a phase transition for large a]{Existence of a phase transition for large $a$}\label{sec:large a}

In this section, we prove a weaker version of Theorem \ref{theo:existence mixed}.

\begin{theorem}\label{theo:existence large a}
    For all $L\in \N$ and $a>200 L$ there exist $q\in (0,1)$ and an isometry-invariant kernel $J$ on $\mathbb{H}_L^1$ satisfying 
	\begin{align*}
		J \sim \frac{a L \log\log(n)}{n^{2}}
	\end{align*}
	such that $\p_{q,J} \left( |K_0| = \infty \right)> 0$.
\end{theorem}

The difference to Theorem \ref{theo:existence mixed} is that Theorem \ref{theo:existence large a} only guarantees the existence of a phase transition when the constant $a$ defined above is at least $200 L$, whereas Theorem \ref{theo:existence} guarantees the existence of a phase transition for all constants $a>1$. We prove this weaker version now, as it is sufficient 
for the proofs of \Cref{theo:(dis)continuity} and its proof is much easier than the proof of Theorem \ref{theo:existence mixed}

\medskip

Before starting the proof, we first introduce some notation. Given $A \subseteq \mathbb{H}^1_L$ and $x,y \in A$, we write $x \overset{A}{\longleftrightarrow} y$ if there exists an open path connecting $x$ to $y$ inside $A$. We write $K_x(A) = \{y \in A : x \overset{A}{\longleftrightarrow} y \}$ for the open cluster containing $x$ in $A$ and write $|K_{\max}(A)| = \max_{x \in A} |K_x(A)|$ for the size of the largest cluster inside $A$. We also define $K_{\max}(A)$ as the vertex set of the largest cluster inside $A$; in case this cluster is not unique, we use an arbitrary deterministic tie-breaking rule depending only on the configuration inside of $A$. 

\medskip

We now begin working towards the proof of Theorem \ref{theo:existence large a}. Fix $L\in \N$ and $a>200 L$
and let $M > L^2$ be a sufficiently large integer such that $1-\sum_{m=M}^{\infty} \frac{1}{m^2} \geq 0.9$ and $M^{2} \geq 2^{14}$. For each $m \geq M$, we define
    \begin{equation*}
        \theta_{m} = 1-\sum_{k=M}^{m} \frac{1}{k^2} \geq 0.9 .
    \end{equation*}
    Let $r_m$ be the unique integer such that $m^{2} \leq L^{r_m} < L m^{2}$,
    and define $R_{m}$ by
        $R_{m} = M + \sum_{i=M}^{m} r_{i}$.

\begin{definition}
    We say that a block of the form $\Lambda_{R_m}(u)$ is  \emph{good} if $\left|K_{\max}\left( \Lambda_{R_m}(u) \right)\right| \geq \theta_{m} \left| \Lambda_{R_m}(u) \right|$, and \emph{bad} otherwise.
\end{definition}

The proof of Theorem \ref{theo:existence large a} is based on the idea that clusters of density $\theta_m$ inside blocks of the form $\Lambda_{R_{m}}(u)$ with $u\in \Lambda_{R_{m+1}}(0)$ will, with high probability, merge and create a cluster of density $\theta_{m+1}$ inside the block $\Lambda_{R_{m+1}}(0)$. 
The main input for the proof of Theorem \ref{theo:existence large a} is the following renormalization inequality.

\begin{lemma}\label{lem:renorm ineq}
    Let $J$ be a kernel on $\H_L$ such that $J(L^{n}) \geq a L^{-2n+1} \log n$ for all $n > R_{m}$. Then 
    \begin{equation*}
         \p_{q,J} \left( \Lambda_{R_{m+1}}(0) \text{ \emph{is bad}} \right) \leq L^2 (m+1)^4 \p_{q,J} \left( \Lambda_{R_{m}}(0) \text{ \emph{is bad}} \right)^2 + (m+1)^2 L^3 m^{-\frac{a}{10 L}} .
    \end{equation*}
\end{lemma}

Let us first see how this lemma implies Theorem \ref{theo:existence large a}.

\begin{proof}[Proof of Theorem \ref{theo:existence large a} given Lemma \ref{lem:renorm ineq}]
    Recall that the value of $a> 200 L$ was fixed. Let $J$ be the kernel defined by 
    \begin{equation*}
        J(L^n) = \begin{cases}
            \frac{a L \log(n)}{L^{2n}} & \text{ if } n > M
            \\
            N & \text{ if } n \leq M
        \end{cases},
    \end{equation*}
    which satisfies $J \sim aL n^{-2} \log\log n$.
    For each $m \geq M$, define
    $b_m \coloneqq \p_{q,J} \left( \Lambda_{R_{m}} (0) \text{ bad} \right)$.
    Let $q \in (0,1)$ and $N \in [0,\infty)$ be large enough so that $b_M \leq M^{-7}$. We will now show that $b_m \leq m^{-7}$ by induction over all $m \geq M$. The base case $m=M$ holds by assumption. For the induction step, we use Lemma \ref{lem:renorm ineq} and the assumptions $m \geq M > L^2, a > 200 L$ to see that
    \begin{align*}
        b_{m+1} & \leq L^2 (m+1)^4 b_{m}^2 + (m+1)^2 L^3 m^{-\frac{a}{10L}}
        \leq 2^4 m^5 b_{m}^2 + 2 m^5 m^{-\frac{a}{10L}}
        \\
        &
        \leq 2^4 m^5 \left( \frac{1}{m^7} \right)^2  + 2 m^5 m^{-20}
        \leq 2^5 m^{-9} \leq  2^{14} (m+1)^{-9} \leq (m+1)^{-7},
    \end{align*}
    where we used the assumption $(m+1)^2 \geq M^2 \geq 2^{14}$ for the last inequality. 

    In particular, the previous argument implies that $1-b_{m} = \p_{q,J}\left( \Lambda_{R_m}(0) \text{ is good} \right) \geq 0.5$ for all $m \geq M$. The isometry-invariance of the kernel $J$ thus implies that
    \begin{multline*}
        \p_{q,J} \left( |K_0| \geq 0.9 |\Lambda_{R_{m}}| \right) 
        \geq 
        \p_{q,J} \left( \Lambda_{R_m}(0) \text{ good}, 0 \in K_{\max}\left(\Lambda_{R_m}(0) \right) \right)
        \\
        \geq
        0.9 \ \p_{q,J} \left( \Lambda_{R_m}(0) \text{ good} \right) = 0.9 (1-b_m) \geq 0.45,
    \end{multline*}
    and since $|\Lambda_{R_{m}}(0)|$ diverges as $m \to \infty$, we get that
    \begin{equation*}
        \p_{q,J} \left( |K_0| = \infty \right) = \lim_{m \to \infty} \p_{q,J} \left( |K_0| \geq 0.9 |\Lambda_{R_{m}}(0)| \right) \geq 0.45 > 0 
    \end{equation*}
    as required.
\end{proof}

We now proceed to the proof of Lemma \ref{lem:renorm ineq}.

\begin{proof}[Proof of Lemma \ref{lem:renorm ineq}]
    For each $j \in \{0,1,\ldots,r_{m+1}\}$, let $u_1^{j},\ldots, u_{L^{r_{m+1}-j}}^{j} \in \H_L$ be such that
    \begin{equation*}
        \bigcup_{i=1}^{L^{r_{m+1}-j}} \Lambda_{R_{m}+j}(u_i^j) = \Lambda_{R_{m+1}}(0)
    \end{equation*}
    and write $U_j \coloneqq \{ u_1^{j},\ldots, u_{L^{r_{m+1}-j}}^{j} \}$.
    
    For $j \in \{1,\ldots,r_{m+1}\}$ and $u \in U_{j}$, let $v_1,\ldots,v_L \in \Lambda_{R_{m}+j}(u)$ be such that $\Lambda_{R_{m}+j}(u) = \bigcup_{i=1}^{L} \Lambda_{R_{m}+j-1}(v_i)$, so  that the blocks $\Lambda_{R_{m}+j-1}(v_i)$ and $ \Lambda_{R_{m}+j-1}(v_{i^\prime})$ are disjoint for all distinct $i,i^\prime \in \{1,\ldots,L\}$. We say that the block $\Lambda_{R_{m}+j}(u)$ is \emph{nice} if 
        $ K_{\max}\left( \Lambda_{R_{m}+j-1}(v_i) \right) \sim K_{\max}\left( \Lambda_{R_{m}+j-1}(v_{i^\prime}) \right)$ (meaning that there is an open edge connecting the two sets) for all distinct $i,i^\prime \in \{1,\ldots,L\} $ for which $ |K_{\max}\left( \Lambda_{R_{m}+j-1}(v_i) \right)| , |K_{\max}\left( \Lambda_{R_{m}+j-1}(v_{i^\prime}) \right)| \geq 0.4 L^{R_{m}+j-1}$.
    We see that the definition of $\Lambda_{R_{m}+j}(u)$ being nice does not depend on the exact choice of $v_1,\ldots,v_{L}$.

    For distinct $v_{i}, v_{i^\prime}$ as above and each pair of vertices $x,y$ with $x \in K_{\max}\left( \Lambda_{R_{m}+j-1}(v_{i}) \right)$, $ y \in K_{\max}\left( \Lambda_{R_{m}+j-1}(v_{i^\prime}) \right)$ one has $\|x-y\| = L^{R_m+j}$. Thus, for any two sets $A\subseteq \Lambda_{R_{m}+j-1}(v_{i}), B \subseteq \Lambda_{R_{m}+j-1}(v_{i^\prime})$ one has
    \begin{equation*}
        \p_{q,J} \left( A \nsim B \big| v \text{ open for all } v \in A\cup B \right) = \prod_{x \in A} \prod_{y \in B} e^{-J(\{x,y\})} = \exp \left( - J\!\left( L^{R_{m}+j} \right) |A| |B| \right) .
    \end{equation*}
    Writing $\mathcal{B}_{i,i^\prime}$ for the event that $\left| K_{\max}\left( \Lambda_{R_{m}+j-1}(v_{i^\prime}) \right)\right|,  \left| K_{\max}\left( \Lambda_{R_{m}+j-1}(v_{i}) \right) \right| \geq 0.4 L^{R_{m}+j-1}$ and noting that this event is independent of the status of the edges between the two blocks, we see that
    \begin{align*}
        &\p_{q,J} \left( K_{\max}\left( \Lambda_{R_{m}+j-1}(v_{i}) \right) \nsim K_{\max}\left( \Lambda_{R_{m}+j-1}(v_{i^\prime}) \right) \big| \mathcal{B}_{i,i^\prime} \right) 
        \\&\hspace{1.3cm}\leq
        \exp \left( - J\!\left( L^{R_{m}+j} \right) \left( 0.4 L^{R_{m}+j-1} \right)^2 \right)
        \leq
        \exp \left( - \frac{a L \log \left( R_{m}+j \right)}{\left(L^{R_{m}+j}\right)^2} \left( 0.4 L^{R_{m}+j-1} \right)^2 \right)
        \\&\hspace{1.3cm} =
        \exp \left( - \frac{0.16 a   \log \left( R_{m}+j \right)}{L} \right)
        \leq 
        m^{-\frac{ a}{10 L}} ,
    \end{align*}
    where we used that $R_{m}+j \geq R_{m} \geq m$ in the last inequality. A union bound over all distinct $i,i^\prime \in \{1,\ldots,L\}$ therefore shows that
    \begin{align*}
        & \p_{q,J} \left( \Lambda_{R_{m}+j}(u) \text{ not nice} \right)
        \\
        &
        \hspace{1cm}\leq \sum_{\substack{i,i^\prime \in \{1,\ldots,L\}: \\ i \neq i^\prime }} \p_{q,J} \left( K_{\max}\left( \Lambda_{R_{m}+j-1}(v_{i}) \right) \nsim K_{\max}\left( \Lambda_{R_{m}+j-1}(v_{i^\prime}) \right) \big| \mathcal{B}_{i,i^\prime} \right) \p_{q,J} \left( \mathcal{B}_{i,i^\prime} \right)
        \\
        &
        \hspace{1cm}\leq \sum_{\substack{i,i^\prime \in \{1,\ldots,L\}: \\ i \neq i^\prime }} \p_{q,J} \left( K_{\max}\left( \Lambda_{R_{m}+j-1}(v_{i}) \right) \nsim K_{\max}\left( \Lambda_{R_{m}+j-1}(v_{i^\prime}) \right) \big| \mathcal{B}_{i,i^\prime} \right)
        \\
        &
        \hspace{1cm}\leq L^2 m^{-\frac{a}{10L}}.
    \end{align*}
    A union bound over all $u \in U_{j}$ and $j \in \{1,\ldots,r_{m+1}\}$ then   shows that
    \begin{align}\label{not nice}
        & \notag \p_{q,J} \left( \text{there exists } j \in \{1,\ldots,r_{m+1}\} \text{ and } u \in U_j \text{ such that } \Lambda_{R_{m}+j}(u) \text{ is not nice} \right)
        \\
        & \hspace{2cm}
        \leq \notag
        \sum_{j = 1}^{r_{m+1}} \sum_{u \in U_j}
        \p_{q,J} \left( \Lambda_{R_{m}+j}(u) \text{ is not nice} \right)
        \\
        & \hspace{2cm}
        \leq
        \sum_{j = 1}^{r_{m+1}} L^{r_{m+1}-j} L^2 m^{-\frac{a}{10 L}} 
        \leq 
        L^{r_{m+1}} L^2 m^{-\frac{a}{10L}}
        \leq
        L (m+1)^2 L^2 m^{-\frac{ a}{10 L}} .
    \end{align}
    We write $\mathscr{N}_m$ for the event whose probability is estimated in \eqref{not nice}.

    To proceed, 
    we divide the set $U_0 = U_0^G \cup U_0^B$ into its \emph{good part} $U_0^G$ and its \emph{bad part} $U_0^B$, which are defined by
    \begin{equation*}
        U_0^G \coloneqq \left\{ u \in U_0 : \Lambda_{R_{m}}(u) \text{ good} \right\} \quad \text{ and } \quad U_0^B \coloneqq \left\{ u \in U_0 : \Lambda_{R_{m}}(u) \text{ bad} \right\} = U_0 \setminus U_0^G ,
    \end{equation*}
    respectively.
    We will use the following deterministic claim, whose proof is deferred to the end of the section:

    \begin{claim}\label{claim:niceness}
        If $|U_0^B| \leq 1$ and all blocks of the form $\Lambda_{R_{m}+j}(u)$ for $u \in U_j, j \in \{1,\ldots,r_{m+1}\}$ are nice, then $\left| K_{\max}\left( \Lambda_{R_{m+1}}(0) \right) \right| \geq \theta_{m+1} \left| \Lambda_{R_{m+1}}(0) \right|$.
    \end{claim}
    
   This claim implies in particular that
    \begin{multline}
 \p_{q,J} \left( \Lambda_{R_{m+1}}(0) \text{ is bad} \right)
        =
        \p_{q,J} \left( \left| K_{\max}\left( \Lambda_{R_{m+1}}(0) \right) \right| < \theta_{m+1} L^{R_{m+1}} \right) 
        \\
        \leq \p_{q,J} \left( \{|U_0^B| \geq 2\} \cup \mathscr{N}_m\right)
          \label{eq:2 bad or not nice} 
        \leq \p_{q,J} \left( |U_0^B| \geq 2 \right) + \p_{q,J} \left(\mathscr{N}_m \right).
    \end{multline}
    To bound the first summand in \eqref{eq:2 bad or not nice}, note that the events $\left\{u \in U_0^B\right\}$ are independent for different $u \in U_0$, and they have the same probability for all $u \in U_0$, so that
    \begin{align*}
        \p_{q,J} \left( |U_0^B| \geq 2 \right) & 
        \leq 
        \sum_{ \substack{u, u^\prime \in U_0 : \\ u \neq u^\prime} } \p_{q,J} \left( u,u^\prime \in U_0^B \right) 
        =
        \sum_{ \substack{u, u^\prime \in U_0 : \\ u \neq u^\prime} } \p_{q,J} \left( u \in U_0^B \right) \p_{q,J} \left( u^\prime \in U_0^B \right) \\
        &
        \leq |U_0|^2 \p_{q,J} \left( \Lambda_{R_{m}}(0) \text{ bad} \right)^2 
        = 
        \left( L^{r_{m+1}} \right)^2 \p_{q,J} \left( \Lambda_{R_{m}}(0) \text{ bad} \right)^2
        \\
        &
        \leq
        L^2 (m+1)^4 \p_{q,J} \left( \Lambda_{R_{m}}(0) \text{ bad} \right)^2.
    \end{align*}
    Combining this with the upper bound \eqref{not nice} and inserting it into \eqref{eq:2 bad or not nice}, we get that
    \begin{align*}
        \p_{q,J} \left( \Lambda_{R_{m+1}}(0) \text{ is bad} \right) \leq L^2 (m+1)^4 \p_{q,J} \left( \Lambda_{R_{m}}(0) \text{ bad} \right)^2 + (m+1)^2 L^3 m^{-\frac{ a}{10L}}
    \end{align*}
    as claimed.
\end{proof}

It remains to prove Claim \ref{claim:niceness}.

\begin{proof}[Proof of Claim \ref{claim:niceness}]
    We observe that if $|U_0^B| = 0$ and all blocks of the form $\Lambda_{R_{m}+j}(u)$ for $u \in U_j, j \in \{1,\ldots,r_{m+1}\}$ are nice, then $\left| K_{\max}\left( \Lambda_{R_{m+1}}(0) \right) \right| \geq \theta_m \left| \Lambda_{R_{m+1}}(0) \right| \geq \theta_{m+1} \left| \Lambda_{R_{m+1}}(0) \right|$. This readily follows by induction over $j=0,\ldots,r_{m+1}$, showing that if $u \in U_{j}$ is such that $U_0^B \cap \Lambda_{R_{m}+j}(u) = \emptyset$, then $\left| K_{\max} \left( \Lambda_{R_{m}+j}(u) \right) \right| \geq \theta_{m} \left| \Lambda_{R_{m}+j}(u) \right|$.
    
    Now suppose that $|U_0^B|= 1$, say with $U_0^B = \{a_0\}$, and that all blocks of the form $\Lambda_{R_{m}}(v)$ for $v \in U_j, j \in \{1,\ldots,r_{m+1}\}$ are nice. We now show via induction over $j \in \{1,2,\ldots,r_{m+1}\}$ that there exists $a_j \in U_j$ such that 
    \begin{align}\label{induc 1}
        \left| K_{\max} \left( \Lambda_{R_m + j} (a_j) \right) \right| \geq \left( 1- L^{-j} \right) \theta_m \left| \Lambda_{R_m + j} (0) \right| ,
    \end{align}
    and that
    \begin{align}\label{induc 2}
        \left| K_{\max} \left( \Lambda_{R_m + j} (v) \right) \right| \geq \theta_m \left| \Lambda_{R_m + j} (0) \right| 
    \end{align}
    for all $j\in \{1,\ldots,r_{m+1}\}, v \in U_j \setminus \{a_j\}$.
    (Note that $a_j$ can be chosen to be the unique element $w \in U_j$ for which $\Lambda_{R_m}(a_0) \subset \Lambda_{R_m + j} (w)$.)\\

    For the start of the induction at $j=1$, let $v_1, \ldots, v_{L-1} \in U_0$ be distinct and such that $\|a_0-v_i\| = L^{R_{m}+1}$ for all $i \in \{1,\ldots,L-1\}$. Further, let $a_1 \in U_1$ be the unique element in $U_1$ for which $a_0,v_1, \ldots, v_{L-1} \in \Lambda_{R_{m}+1}(a_1)$. As $\Lambda_{R_{m}+1}(a_1)$ was assumed to be nice, and the blocks $ \Lambda_{R_{m}}(v_1),\ldots, \Lambda_{R_{m}}(v_{L-1})$ all contain a cluster of density at least $0.4$, we get that
    \begin{align*}
        &\left|K_{\max} \left( \Lambda_{R_{m}+1}(a_1) \right) \right| \geq \sum_{i=1}^{L-1} \left| K_{\max}\left( \Lambda_{R_{m}}(v_i) \right) \right|
        \geq (L-1) \theta_{m} \left| \Lambda_{R_{m}}(0) \right|
        \\
        &
        \hspace{6cm}
        =
        \left( 1- \frac{1}{L} \right) \theta_{m} \left| \Lambda_{R_{m}+1}(0) \right|
        \geq 0.4 \left| \Lambda_{R_{m}+1}(0) \right| 
    \end{align*}
    where we used $\theta_m \geq 0.9$ for the last inequality. This shows \eqref{induc 1} for $j=1$. For all other elements $u \in U_1 \setminus \{a_1\}$, let $v_1,\ldots,v_L \in U_0$ be distinct elements of $U_1$ such that $\Lambda_{R_m}(v_i) \subseteq \Lambda_{R_m+1}(u)$, so that
    \begin{align*}
        & \left|K_{\max} \left( \Lambda_{R_{m}+1}(u) \right) \right| \geq \sum_{i=1}^{L} \left| K_{\max}\left( \Lambda_{R_{m}}(v_i) \right) \right|
        \geq L \theta_{m} \left| \Lambda_{R_{m}}(0) \right|
        =
         \theta_{m} \left| \Lambda_{R_{m}+1}(0) \right| ,
    \end{align*}
    which shows \eqref{induc 2} for $j=1$.
    For the induction step from $j$ to $j+1$, let $a_{j+1}$ be the unique element $w \in U_{j+1}$ for which $\Lambda_{R_m+j}(a_j) \subset \Lambda_{R_m + j+1}(w)$. Let $v_1,\ldots,v_{L-1} \in U_j \setminus \{a_j\}$ be distinct such that $\Lambda_{R_m+j}(v_i) \subset \Lambda_{R_m + j+1}(a_{j+1})$ for all $i\in \{1,\ldots,L-1\}$. As $\Lambda_{R_{m}+1}(a_{j+1})$ was assumed to be nice and all $\Lambda_{R_m+j}(a_j),\Lambda_{R_m+j}(v_1),\ldots,\Lambda_{R_m+j}(v_{L-1})$ all contain an open cluster of density at least $0.4$, we get that
    \begin{align*}
        &\left|K_{\max} \left( \Lambda_{R_{m}+j+1}(a_{j+1}) \right) \right| \geq \left| K_{\max}\left( \Lambda_{R_{m}+j}(a_j) \right) \right| + \sum_{i=1}^{L-1} \left| K_{\max}\left( \Lambda_{R_{m}+j}(v_i) \right) \right| 
        \\
        & \hspace{6cm}
        \geq \left( 1- \frac{1}{L^j} \right) \theta_{m} \left| \Lambda_{R_{m}+j}(0) \right| 
        +
        (L-1) \theta_{m} \left| \Lambda_{R_{m}+j}(0) \right| 
        \\
        &
        \hspace{6cm}
        =
        \left( 1- \frac{1}{L^{j+1}} \right) \theta_{m} \left| \Lambda_{R_{m}+1}(0) \right| ,
    \end{align*}
    which shows \eqref{induc 1}. For all other elements $u \in U_{j+1} \setminus \{a_{j+1}\}$, let $v_1,\ldots,v_L \in U_j$ be distinct elements of $U_j$ such that $\Lambda_{R_m+j}(v_i) \subseteq \Lambda_{R_m+j+1}(u)$, so that
    \begin{align*}
        & \left|K_{\max} \left( \Lambda_{R_{m}+1}(u) \right) \right| \geq \sum_{i=1}^{L} \left| K_{\max}\left( \Lambda_{R_{m}}(v_i) \right) \right|
        \geq L \theta_{m} \left| \Lambda_{R_{m}}(0) \right|
        =
         \theta_{m} \left| \Lambda_{R_{m}+1}(0) \right| ,
    \end{align*}
    which shows \eqref{induc 2} and thus finishes the proof of the induction step. We now apply the result of the induction, inequality \eqref{induc 2}, for $j=r_{m+1}$. Since $|U_{r_{m+1}}|=1$, the element $a_{r_{m+1}}$ satisfies $\Lambda_{R_m+r_{m+1}}(a_{r_{m+1}}) = \Lambda_{R_{m+1}}(0)$ and thus
    \begin{align*}
        & \left| K_{\max} \left( \Lambda_{R_{m+1}}(0) \right) \right|
        =
        \left| K_{\max} \left( \Lambda_{R_{m}+r_{m+1}}(a_{r_{m+1}}) \right) \right|
        \geq
        \left( 1- L^{-r_{m+1}} \right) \theta_m \left| \Lambda_{R_m + r_{m+1}} (0) \right|
        \\
        &
        \hspace{55mm}
        \geq
        \left( 1- \frac{1}{(m+1)^2} \right) \theta_m \left| \Lambda_{R_m + r_{m+1}} (0) \right|
        \\
        &
        \hspace{55mm}
        \geq
        \left( \theta_m - \frac{1}{(m+1)^2} \right) \left| \Lambda_{R_m + r_{m+1}} (0) \right|
        =
        \theta_{m+1}  \left| \Lambda_{R_{m+1}} (0) \right|
    \end{align*}
    as required.
\end{proof}

\section{The continuous regime}\label{sec:continuity}

In this section, we apply \Cref{theo:existence large a} to prove part $1$ of \Cref{theo:(dis)continuity}. We continue to use the notation $K_\mathrm{max}(\Lambda_n)$ as defined in Section \ref{sec:large a}.

\medskip

As a first step towards the continuity of the phase transition, we will prove that the largest open cluster inside $\Lambda_n$ is of order $|\Lambda_n|\p_{J} \left( |K_0| = \infty \right)$ whenever $\p_{J} \left( |K_0| = \infty \right) > 0$ and $n$ is large. We use a similar technique as in the proof of \cite[Theorem 1.5]{MR2955049}.

\begin{lemma}[Cluster sizes inside blocks]\label{lem:cluster size restricted}
	Let $J$ be an isometry-invariant, integrable, and regular kernel on $\H^1_L$. If the infinite cluster density $\theta \coloneqq \p_{J} \left( |K_0| = \infty \right)$ is positive then 
	\begin{equation*}
		\p_{J} \left( |K_{\max}(\Lambda_n)| \geq (\theta -\eps)|\Lambda_n| \right) \underset{n\to \infty}{\longrightarrow} 1 
	\end{equation*}
    for every $\eps>0$.
\end{lemma}

(Note that this would be a trivial consequence of the ergodic theorem if we were allowing connections to leave the block.)

\begin{proof}[Proof of \Cref{lem:cluster size restricted}]
	We define the {\sl critical size at scale $n$}, which we denote by $f(n)$, via
    \[\frac{1}{f(n)} = \sum_{x \notin \Lambda_n} J(\{0,x\}),\] so that \[\E \left[\left|\left\{x \notin \Lambda_n : \{0,x\} \text{ open}\right\}\right|\right] = \sum_{x\notin \Lambda_n} 1-\exp(-J(\{0,x\})) \leq \sum_{x\notin \Lambda_n} J(\{0,x\}) \leq f(n)^{-1}\] and $f(n)$ diverges to infinity as $n\to \infty$ since $J$ is integrable. We first prove that
	\begin{multline}\label{eq:critical size}
		\lim_{n\to \infty} \frac{1}{|\Lambda_n|}\E_J \left[\left|\left\{x \in \Lambda_n : |K_x(\Lambda_n)| \leq r f(n) ,\, x \leftrightarrow \infty \right\} \right| \right] \\= \lim_{n\to \infty} \p_J(|K_0(\Lambda_n)|\leq r f(n), 0 \leftrightarrow \infty)= 0
	\end{multline}
    for every $r>0$.
	Indeed, we have
	\begin{align}\label{eq:sum of two}
		&\notag \p_J \left( |K_0(\Lambda_n)| \leq r f(n),\, 0 \leftrightarrow \infty \right) \\
		& =
		\p_J \left( \sqrt{f(n)} < |K_0(\Lambda_n)| \leq r f(n),\, 0 \leftrightarrow \infty \right) + \p_J \left( |K_0(\Lambda_n)| \leq \sqrt{f(n)} ,\, 0 \leftrightarrow \infty \right).
	\end{align}
	To bound the second term,
    we note that if $0$ is connected to $\infty$ then there must be an open edge with one endpoint in $K_0(\Lambda_n)$ and the other in $\Lambda_n^c$, so that
	\begin{align}\label{eq:summand2}
		\notag \p_J \left( |K_0(\Lambda_n)| \leq \sqrt{f(n)} ,\, 0 \leftrightarrow \infty \right) &\leq \p_J \left( 0 \leftrightarrow \infty \;\Big|\; |K_0(\Lambda_n)| \leq \sqrt{f(n)}  \right) \\ &\leq \sqrt{f(n)} \sum_{x\notin \Lambda_n} \p_{J} \left(0\sim x\right) \leq \frac{\sqrt{f(n)} }{f(n) } = o(1)
	\end{align}
    as required, where we write $\{0\sim x\}$ for the event that the edge $\{0,x\}$ is open.
	To bound the first term in \eqref{eq:sum of two}, we note that for any set $A\subset \Lambda_n$ with $|A|\leq r f(n)$, there is a uniformly positive (in $n$) probability that there is no open edge from $A$ to $\Lambda_n^c$. Indeed, we have explicitly for every set $A\subset \Lambda_n$ with $|A|\leq r f(n)$ that
    \begin{equation*}
        \p_J \left( A \nsim \Lambda_n^c  \right) \geq \exp \left( - |A| \sum_{x\notin \Lambda_n} J(\{0,x\}) \right) \geq e^{-r},
    \end{equation*}
    where $\{A\nsim \Lambda_n^c\}$ denotes the event that there are no open edges between $A$ and $\Lambda_n^c$.
    As the edges inside $\Lambda_n$ are independent of the edges between $\Lambda_n$ and $\Lambda_n^c$, we also have the conditional bound
    \begin{align*}
        \p_J \left( A \sim \Lambda_n^c \;|\; K_0(\Lambda_n) = A \right) \leq 1 \leq e^{r} \p_J \left( A \nsim \Lambda_n^c \;|\; K_0(\Lambda_n) = A \right),
    \end{align*}
    which holds for all sets $A \subset \Lambda_n$ with $|A|\leq rf(n)$.
    Writing $\cA$ for all subsets $A \subset \Lambda_n$ with $\sqrt{f(n)} < |A| \leq r f(n)$, we see that
	\begin{align} \label{eq:summand1}
		& \notag \p_J \left( \sqrt{f(n)} < |K_0(\Lambda_n)| \leq r f(n),\, 0 \leftrightarrow \infty \right) \\
        &\hspace{5cm} \notag = \sum_{A \in \cA} \p_J \left( 0 \leftrightarrow \infty \;|\; K_0(\Lambda_n) = A \right) \p_J \left( K_0(\Lambda_n) = A \right)
        \\
        &\hspace{5cm} \notag
        \leq \sum_{A \in \cA} \p_J \left( A \sim \Lambda_n^c \;|\; K_0(\Lambda_n) = A \right) \p_J \left( K_0(\Lambda_n) = A \right)
        \\
        &\hspace{5cm} \notag
        \leq e^r \sum_{A \in \cA} \p_J \left( A \nsim \Lambda_n^c \;|\; K_0(\Lambda_n) = A \right) \p_J \left( K_0(\Lambda_n) = A \right)
        \\
        &\hspace{5cm} \notag
        = e^r \p_J \left( \sqrt{f(n)} < |K_0(\Lambda_n)| \leq r f(n) ,\, K_0(\Lambda_n) \nsim \Lambda_n^c \right)
        \\
        &\hspace{5cm} 
        \leq e^r \p_J \left( \sqrt{f(n)} < |K_0| \leq r f(n)  \right) = o(1),
	\end{align}
    where the fact that the final expression tends to zero as $n\to\infty$ follows from the fact that $f(n)$ diverges and thus $\p_J \left( \sqrt{f(n)} < |K_0| \leq r f(n)  \right)$ tends to zero regardless of the law of $|K_0|$.
	Combining the two inequalities \eqref{eq:summand2} and \eqref{eq:summand1} and inserting them into equation \eqref{eq:sum of two} we get that $\p_J \left( |K_0(\Lambda_n)| \leq r f(n), 0 \leftrightarrow \infty \right)=o(1)$ as required.
    
    We now apply \eqref{eq:critical size} to prove the lemma. It follows from the ergodic theorem (which applies to percolation on any amenable group \cite{lindenstrauss1999pointwise}) that 
		$|\{x \in \Lambda_n : x \leftrightarrow \infty \}|/|\Lambda_n| \to \theta$
	almost surely and hence in probability. Since 
$|\{x \in \Lambda_n : |K_x(\Lambda_n)| \leq r f(n) , x \leftrightarrow \infty \}| / |\Lambda_n| \to 0$
    in expectation and thus also in probability by \eqref{eq:critical size}, it follows that
	\begin{align*}
		\p_J \left( \left|\bigl\{x \in \Lambda_n : |K_x(\Lambda_n)| \geq rf(n)\bigr\}\right| \geq (\theta-\eps_1) |\Lambda_n| \right) \to 1
	\end{align*}
    for every $\eps_1>0$.
	Assuming this lower bound on the cluster size in $\Lambda_n$, we now construct a \emph{linear-size} cluster in $\Lambda_{n+1}$. Fix $\eps_1 \in (0,\theta/2)$. $\Lambda_{n+1}$ consists of $L$ independent copies of $\Lambda_n$; call these copies $\Lambda_n^1, \ldots, \Lambda_n^{L}$ and let $\cF$ be the sigma-algebra generated by the restriction of the percolation configuration to each of these copies.
    Let $D_i = |\{x \in \Lambda_n^i : |K_x(\Lambda_n^i)| \geq rf(n)\}| $ and consider the event that $D_i \geq (\theta-\eps_1) |\Lambda_n|$ for every $i\in \{1,\ldots,L\}$. 
    On this event, since $D_1 \geq \theta/2 |\Lambda_n|$, 
    we can split each cluster of size at least $r f(n)$ in $\Lambda^1_n$ into (not necessarily connected) sets of size between $\lfloor \sqrt{r} f(n) \rfloor$ and $2\lfloor \sqrt{r} f(n) \rfloor$ to obtain that 
    there exists 
    \begin{equation}\label{eq:N1_bound}\frac{|\Lambda_n|}{\lfloor\sqrt{r}f(n) \rfloor} \geq N_1 \geq \frac{\theta |\Lambda_n|}{4 \lfloor \sqrt{r} f(n) \rfloor} \end{equation}
    and a collection of disjoint (not necessarily connected\footnote{In fact, it is possible to take these sets to be connected after changing the relevant constants by a bounded factor using \cite[Lemma 2.4]{hutchcroft2020power}, but we will not need this.}) sets
     $K_{1,1},\ldots, K_{1,N_1} \subset \Lambda^1_n$ such that $|K_{1,j}| \in \left\{ \lfloor \sqrt{r} f(n) \rfloor, \ldots, 2  \lfloor \sqrt{r} f(n) \rfloor\right\}$  for every $1\leq j \leq N_1$
     and such that each $K_{1,j}$ is a subset of a single cluster in $\Lambda^1_n$; after conditioning on $\cF$ we fix one such choice of $K_{1,1},\ldots,K_{1,N_1}$ and refer to these distinguished sets as \textbf{metavertices}.
Repeating the same procedure for each of the copies
$\Lambda_n^i$, we obtain for each $1\leq i \leq L$ a collection of metavertices $\{K_{i,j}:1\leq j \leq N_i\}$,  chosen in a $\cF$-measurable manner, with each $N_i$ satisfying the same bounds \eqref{eq:N1_bound} as $N_1$.
     
    We now connect the metavertices. Considering the edges that get drawn between different metavertices when we introduce the edges of the percolation configuration between the different blocks $\Lambda_n^i$, we have a random $L$-partite graph in which the $i$-th $L$-partite component consists of $N_i$ metavertices, and two metavertices from different $L$-partite components are connected with probability at least
	\begin{equation}\label{eq:metaconnection}
		1-\exp \left( \lfloor\sqrt{r}f(n) \rfloor^2 J\left(L^{n+1}\right)  \right) 
		\geq 1-\exp \left( \frac{r f(n)^2 J\left(L^{n+1}\right)}{2}  \right),
	\end{equation} 
	where the last inequality holds for $r,n$ large enough.
	From the definition of $f(n)$ and the regularity of the kernel $J$, we know that there exists a constant $c>0$ such that for all large enough $n$
    \begin{equation*}
        f(n) J\left(L^{n+1}\right) = \frac{f(n)}{|\Lambda_{n+1}\setminus \Lambda_{n}|} \sum_{x \in \Lambda_{n+1}\setminus \Lambda_{n}} J\left(\{0,x\}\right)  \geq \frac{c f(n)}{|\Lambda_{n}|} \sum_{x \notin \Lambda_{n}} J\left(\{0,x\}\right) = \frac{c}{|\Lambda_n|}
    \end{equation*}
    and thus we get
	\begin{align*}
		\frac{r f(n)^2 J\left(L^{n+1}\right)}{2} \geq \frac{r f(n) c}{2|\Lambda_n|}\geq c \sqrt{r} \frac{\lfloor \sqrt{r} f(n) \rfloor}{2|\Lambda_n|}
		\geq    \sqrt{r} \frac{c\theta}{8 N_i},
	\end{align*}
    where we used \eqref{eq:N1_bound} in the final inequality.
	Inserting this into \eqref{eq:metaconnection} shows that two meta-vertices are connected with probability at least
	\begin{align*}
		1-\exp \left( \frac{r f(n)^2 J\left(L^{n+1}\right)}{2}  \right) \geq
		1-\exp \left( \sqrt{r} \frac{c\theta}{8 N_i}  \right).
	\end{align*}
	For an $L$-partite graph with this connection probability, one deduces that there exists a positive constant $r_0$ (determined by $L$ and the constant $c\theta/8$) such that if $r \geq r_0$  then there exists a linear size component with high probability, and the density of this linear size component tends to $1$ as $r$ diverges; see e.g.\ \cite{bollobas2007phase}. As such, for each $\eps_2>0$ there exists $r_1$ such that if $r\geq r_1$ then a proportion $1-\eps_2$ of the metavertices merge with high probability, as $n \to \infty$. If $D_i = \left| \left\{ x \in \Lambda_n^i : |K_x(\Lambda_n^i)| \geq r f(n) \right\} \right| \geq (\theta - \eps_1) |\Lambda_n|$ for all $i\in \{1,\ldots,L\}$, then a proportion $1-\eps_2$ of the metavertices needs to contain at least 
    \begin{align*}
        &(\theta-\eps_1) L |\Lambda_n| - \eps_2 \left( \sum_{i=1}^L N_i \right) 2 \lfloor \sqrt{r} f(n) \rfloor
        \\
        &
        \hspace{3cm} \overset{\eqref{eq:N1_bound}}{\geq} (\theta-\eps_1) L |\Lambda_n| - \eps_2 L \frac{|\Lambda_n|}{\lfloor \sqrt{r} f(n) \rfloor}  2 \lfloor \sqrt{r} f(n) \rfloor
        = (\theta - \eps_1 - 2 \eps_2 ) |\Lambda_{n+1}|
    \end{align*}
    many vertices in $\H_L$.
    As the metavertices are themselves path-connected within $\Lambda_n^i$, this shows that there exists a cluster of density at least $(\theta-\eps_1-2\eps_2)$ in $\Lambda_{n+1}$ with high probability. This finishes the proof, since $\eps_1, \eps_2 > 0$ were arbitrary.
\end{proof}

\begin{proof}[Proof of \Cref{theo:(dis)continuity}, part 1]
	Assume that $\theta = \p_{\lambda_c}(|K_0|=\infty) > 0$. Take $\eps>0$ small enough and $M$ large enough so that for the kernel $J$ defined by
	\begin{align*}
		J(e) = M \frac{\log\log(|e|)}{|e|^2} \vee 0
	\end{align*}
	one has $\p_{1-\eps, J} \left(0 \leftrightarrow \infty\right) > 0$.
	By Lemma \ref{lem:cluster size restricted} one has for $n$ large enough that 
    \[\p_{\lambda_c} \left(|K_{\max}(\Lambda_n)| \geq \frac{\theta}{2}|\Lambda_n|\right) \geq 1-\frac{\eps}{2}.\]
    As such, by continuity and the assumption $\cJ \gg \frac{\log\log n}{n^2}$, we may find $n$ and $\lambda< \lambda_c$ such that
	\begin{equation*}
		\p_{\lambda} \left(|K_{\max}(\Lambda_n)| \geq \frac{\theta}{4}|\Lambda_n|\right) \geq 1-\eps
	\end{equation*}
	and 
	\begin{equation*}
	1-\exp \left(- \frac{\theta^2}{16}|\Lambda_n|^2 \mathcal{J}\left(\lambda, L^N\right) \right) \geq 1- \exp \left(-\left( M \frac{\log\log\left( L^{N-n} \right)}{L^{2(N-n)} } \vee 0 \right)\right)
	\end{equation*}
    for all $N\geq n$.
We now consider a site-bond percolation model in which we contract each $n$-block $\Lambda_n(u)$ into a single vertex as follows. We say that an $n$-block $\Lambda_n(u)$ is \textbf{good} if $|K_{\max}(\Lambda_n(u))| \geq \frac{\theta}{4}|\Lambda_n|$, so that each $n$-block is good with probability at least $1-\eps$ under the measure $\p_\lambda$. For each good $n$-block, pick one of the largest clusters inside the block (in a manner depending only on the configuration inside the block). We say that two $n$-blocks are connected if there exists an edge between the largest clusters of these blocks. Thus, two good $n$-blocks $\Lambda_n(u)$ and $\Lambda_n(v)$ with $\|u-v\|=L^N$ for some $N>n$ are connected with probability at least
	\begin{equation*}
	1-\exp \left(- \frac{\theta^2}{16}|\Lambda_n|^2 \mathcal{J}\left(\lambda, L^N\right) \right) \geq 1- \exp \left(-\left( M \frac{\log\log\left( L^{N-n} \right)}{L^{2(N-n)} } \vee 0 \right)\right).
	\end{equation*}
	It follows that the random graph whose vertices are the good $n$-blocks and where two good $n$-blocks are connected by an edge if their largest clusters are, stochastically dominates percolation with the measure $\p_{1-\eps, J}$ on $\H_L$ and therefore that there exists an infinite open cluster under the measure $\p_\lambda$. This contradicts the assumption that $\lambda<\lambda_c$.
\end{proof}

\section{The discontinuous regime}\label{sec:discontin}

In this section, we use \Cref{theo:existence large a} to prove parts 2 and 3 of \Cref{theo:(dis)continuity}.
Our main tool in showing this is Proposition \ref{propo:discont}, which states that for all kernels $J$ on $\H_{L}^{d}$ one has \[\left(\limsup_{|e|\to \infty} J(e) \frac{L^{-d}|e|^2}{\log\log(|e|)} \right) \p_J \left( |K_0| = \infty \right)^2 \in \{0\} \cup \left[1,\infty\right].\] This result is the hierarchical version of a similar result for one-dimensional $|x-y|^{-2}$ long-range percolation by Aizenman and Newman \cite[Proposition 4.1]{MR868738}. This also implies that for all kernels $J$ with $\limsup_{|e|\to \infty} J(e) \frac{L^{-d}|e|^2}{\log\log(|e|)} \leq 1$ one has $\p_J \left( |K_0| = \infty \right)=0$, yielding the second claim of \Cref{theo:existence}. As usual, it suffices to prove the claim in the case $d=1$.

\begin{proposition}\label{propo:discont}
    Consider percolation on $\H_L$ with respect to some isometry-invariant kernel $J$. If we define $\theta=\p_J ( |K_0| = \infty )$ and 
    	\begin{align}\label{eq:definition of a2}
		\beta \coloneqq \limsup_{|e|\to \infty} J(e) \frac{L^{-1}|e|^{2}}{\log\log(|e|)} ,
	\end{align}
    then one has either $\theta=0$ or $\beta \theta^2 \geq 1$.
\end{proposition}

\begin{proof}[Proof of \Cref{propo:discont}]
	We will assume that $\beta < \infty$ and $\beta \theta^2 < 1$ and deduce that $\theta=0$. Defining $\theta^\star = \theta^\star(M) = \p_J\left(|K_0|>M\right)$ for a large constant $M$, we choose $\beta^\star > \beta$, $\eps > 0$, $\gamma \in (0,1)$, and $M\in \N$ such that 
	\begin{align}
		 \label{eq:M_beta*_assumptions} &\beta^\star \left(\theta^\star + \eps \right)^2 < 1,\qquad   L^{-M} < \frac{\eps}{2},\qquad \text{ and } \qquad J(L^n) \leq \frac{\beta^\star L \log(n)}{L^{2n}}
         \end{align} for all sufficiently large $n$, and 
         \begin{align}
		\label{eq:gamma assumption} & 	L^{-M}\sum_{j=0}^\infty \frac{(L-1)\beta^\star }{L^{1+j}} 
		 +
		 \sum_{j=M}^\infty  \frac{(L-1)\beta^\star }{L^{1+j}}  + (\theta^\star + \eps)^2  \sum_{j=0}^{M-1}  \frac{(L-1)\beta^\star}{L^{1+j} } 
		 < \gamma.
	\end{align}
This is possible since, as $M\to \infty$, $\theta^\star=\theta^\star(M)$ converges to $\theta$ and the sum in \eqref{eq:gamma assumption} converges to 
    \begin{equation*}
        (\theta + \eps)^2 \beta^\star  \sum_{j=0}^{\infty}  \frac{(L-1)}{L^{1+j} } = (\theta + \eps)^2 \beta^\star,
    \end{equation*}
 which is strictly less than $1$ for suitable choices of $\beta^\star > \beta$ and $\eps > 0$ since $\beta \theta^2 < 1$.

    Our first goal is to show that the event $A_N$, defined by
	\begin{equation*}
        A_N = \bigcap_{i=0}^{N} \left\{ 
		\Lambda_{N + i \lfloor \log N \rfloor^2} \nsim \Lambda_{N + (i+1) \lfloor \log N \rfloor^2-1}^c \right\},
	\end{equation*}
    holds with high probability when $N$ is large.
    For $i \in \{0,\ldots,N\}$ and for $N$ large enough, we can take a union bound over the annuli $\Lambda_j \setminus \Lambda_{j-1}$ with $j\geq N+(i+1)\lfloor \log N\rfloor^2$ to get that
	\begin{align*}
		& \p_J\left(\Lambda_{N + i \lfloor \log N \rfloor^2} \sim \Lambda_{N + (i+1) \lfloor \log N \rfloor^2-1}^c  \right) \\
        &\hspace{4cm}\leq  \sum_{x\in \Lambda_{N + i \lfloor \log N \rfloor^2}} \ 
        \sum_{j=N + (i+1) \lfloor \log N \rfloor^2}^\infty \ 
        \sum_{y \in \Lambda_j \setminus \Lambda_{j-1}}
        \p_J\left( x \sim  y \right) \\
        &\hspace{4cm}
		\leq \left|\Lambda_{N + i \lfloor \log N \rfloor^2}\right| 
		\sum_{j=N + (i+1) \lfloor \log N \rfloor^2}^{\infty} |\Lambda_j| \frac{\beta^\star L \log\log\left(L^j\right)}{L^{2j}}\\
		&\hspace{4cm}
		= \beta^\star L \left|\Lambda_{N + i \lfloor \log N \rfloor^2}\right| 
		\sum_{j=N + (i+1) \lfloor \log N \rfloor^2}^{\infty}  \frac{\log\log\left(L^j\right)}{L^{j}}
		\\
		 &\hspace{4cm}\leq C \left|\Lambda_{N + i \lfloor \log N \rfloor^2}\right| 
		 \frac{\log\log\left(L^{N + (i+1) \lfloor \log N \rfloor^2}\right)}{L^{\left(N + (i+1) \lfloor \log N \rfloor^2\right)}}
		 \leq C^\prime \frac{\log (N)}{L^{ \lfloor \log N \rfloor^2 }}
	\end{align*}
	for some constants $C,C^\prime < \infty$. Taking a further union bound over  $i\in \{0,\ldots, N \}$, we obtain that the crude bound
	\begin{align}\label{eq:A_N probab bound}
		\p_{J}(A_N) \geq 1- (N+1) C^\prime \frac{\log (N)}{L^{\lfloor \log N \rfloor^2 }} \geq 1 - \frac{1}{N^2}
	\end{align}
	holds for all sufficiently large $N$.

    Now, for each $N,i \in \N$ and $j \in \N\cup \{+\infty\}$, define 
	\begin{equation*}
        N_{i}= N+i \lfloor \log(N) \rfloor^2 \qquad \text{ and } \qquad
		\Delta_{i,j} = \Lambda_j \setminus \Lambda_i ,
	\end{equation*}
    where we define $\Lambda_{\infty}=V$.
	Let $\mathcal{C}_{i,j}$ be the collection of open clusters inside $\Delta_{i,j}$ (which can only use edges with both endpoints inside this set). Given disjoint sets of vertices $A$ and $B$, we write $\{A\Leftrightarrow B\}$ for the event that there exist at least two distinct open edges connecting $A$ to $B$.  
	Next, for $\lfloor \log(N) \rfloor^2>M$, we define the two events $ \mathbb{B}_N^i, \mathbb{C}_N^i$ by
	\begin{align*}
		\mathbb{B}_N^i & = \left\{ \Lambda_{N_{i} - M} \nsim \Delta_{N_{i},\infty} \right\} \cap \left\{ \Lambda_{N_{i}} \nsim \Delta_{N_{i}+M,\infty} \right\}\\
		&\hspace{8mm}
		 \cap \bigcap_{j=0}^{M-1} \left\{ A \nsim B \text{ for all } A \in \mathcal{C}_{N_{i} - M,N_{i}}, B \in \mathcal{C}_{N_{i}+j,N_{i}+j+1} \text{ with } 
		\left|A\right| > M, \left|B\right| > M
		 \right\},
		 \\
		 \mathbb{C}_N^i & = 
		 \left\{ A \nLeftrightarrow \Delta_{N_{i}-M,N_{i}}^c  \text{ for all } A \in \mathcal{C}_{N_{i} - M,N_{i}} \text{ with } 
		 \left|A\right| \leq M
		 \right\}\\
		 &\hspace{8mm}
		 \cap \bigcap_{j=0}^{M-1} \left\{ A \nLeftrightarrow \Delta_{N_{i}+j,N_{i}+j+1}^c \text{ for all } A \in \mathcal{C}_{N_{i} +j,N_{i} +j+1} \text{ with } 
		 \left|A\right| \leq M
		 \right\}.
	\end{align*}
    We will proceed with our analysis using two claims about these events whose proofs are deferred to later in this section:
    \begin{claim}\label{claim:intersec}
	If $\mathbb{B}_N^i \cap \mathbb{C}_N^i$ holds for some $N$ with $\lfloor \log(N)\rfloor^2 >M$ then there is no open path from $\Lambda_{N_{i}-M}$ to $\Lambda_{N_{i}+M}^c$.
\end{claim}
\begin{claim}\label{claim:probab}
	The event $\mathbb{B}_N^i \cap \mathbb{C}_N^i$ satisfies $\p_J \left( \mathbb{B}_N^i \cap \mathbb{C}_N^i \right) \geq 0.5 N^{-\gamma}$ for all sufficiently large $N$.
\end{claim}
(Recall that the constant $\gamma \in (0,1)$ was chosen at the beginning of the proof.) \Cref{claim:intersec} implies in particular that 
    \begin{multline*}
        \p_J \left( 0 \nleftrightarrow \infty \right) \geq \p \left( \exists j \in \{1,\ldots,N\} : \Lambda_{N_j-M} \nleftrightarrow \Lambda_{N_j+M}^c \right) \\ \geq \p \left( \exists j \in \{1,\ldots,N\} : \mathbb{B}_N^{j} \cap \mathbb{C}_N^{j} \text{ occurs} \right) .
    \end{multline*}
Meanwhile, \Cref{claim:probab} and the estimate \eqref{eq:A_N probab bound} imply that
    \begin{equation*}
        \p_J \left( \mathbb{B}_N^i \cap \mathbb{C}_N^i \mid A_N \right) = \frac{\p_J \left( \mathbb{B}_N^i \cap \mathbb{C}_N^i \right) - \p_J \left( \mathbb{B}_N^i \cap \mathbb{C}_N^i \cap A_N^c \right)}{\p_J \left( A_N \right)} \geq
        \frac{0.5 N^{-\gamma}-N^{-2}}{1}
        \geq
        \frac13 N^{-\gamma}
    \end{equation*}
    for sufficiently large $N$. Observe that, conditioned on $A_N$, the event $\{\mathbb{B}_N^i \cap \mathbb{C}_N^i\}$ depends only on the edges
    \begin{equation*}
        \left\{ e=\{x,y\} : x \in \Delta_{N_{i-1},N_{i+1}} , y \in \Delta_{N_{i-2},N_{i+2}} \right\} 
    \end{equation*}
    and that for different values of $i \in 4\N$, these sets of edges are disjoint.
    It follows that  the events $(\mathbb{B}_N^{4j} \cap \mathbb{C}_N^{4j})_{j\in \{1,\ldots,\lfloor N/4\rfloor \}}$ are independent conditional on $A_N$ and hence that
    \begin{multline*}
        \p_J \Biggl(\, \bigcap_{j=1}^{\lfloor N/4 \rfloor} \bigl( \mathbb{B}_N^{4j} \cap \mathbb{C}_N^{4j} \bigr)^c  \Biggr)
        \leq
        \p_J \Biggl(\, \bigcap_{j=1}^{\lfloor N/4 \rfloor} \bigl( \mathbb{B}_N^{4j} \cap \mathbb{C}_N^{4j} \bigr)^c \;\Bigg|\; A_N  \Biggr)
        +
        \p_J(A_N^c)\\
        =
        \prod_{j=1}^{\lfloor N/4 \rfloor} \p_J \left(  \left( \mathbb{B}_N^{4j} \cap \mathbb{C}_N^{4j} \right)^c \,\Big|\, A_N  \right)
        +
        \p_J(A_N^c)
        \leq
        \prod_{j=1}^{\lfloor N/4 \rfloor} \left( 1 - \frac{1}{3} N^{-\gamma} \right)
        +
        \p_J(A_N^c)
        =
        o(1),
    \end{multline*}
    where we used that $\gamma <1$ when concluding that the product in the final expression is $o(1)$. Thus, with high probability (in $N$), there exists a $j \in \{1,\ldots, N\}$ such that the event $\mathbb{B}_N^j \cap \mathbb{C}_N^j$ holds, meaning that
    \begin{align*}
        \p_J \left( 0 \nleftrightarrow \infty \right) \geq \lim_{N \to \infty} \p \left( \exists j \in \{1,\ldots,N\} : \mathbb{B}_N^{j} \cap \mathbb{C}_N^{j} \text{ occurs} \right) = 1 
    \end{align*}
    as $N\to \infty$, as required.
\end{proof}

It remains to prove the two claims that were deferred from the above proof.

\begin{proof}[Proof of \Cref{claim:intersec}]
Suppose that $\mathbb{B}_N^i \cap \mathbb{C}_N^i$ holds for some $\lfloor \log(N)\rfloor^2 > M$. We call the clusters of size at most $M$ the {\sl small clusters}, and the clusters of size strictly larger than $M$ the {\sl big clusters}.
	Note that every path from $\Lambda_{N_{i}-M}$ to $\Lambda_{N_{i}+M}^c$ needs to go from $\Lambda_{N_{i}-M}$, through both $\Delta_{N_{i}-M,N_{i}}$ and $\Delta_{N_{i},N_{i}+M}$, to $\Delta_{N_{i}+M,\infty}$, by the first two requirements in the definition of $\mathbb{B}_N^i$.
Meanwhile, there is no small cluster in $\Delta_{N_{i}-M,N_{i}}$ that is connected to both $\Lambda_{N_{i}-M}$ and $\Lambda_{N_{i}}^c$, by the first requirement of $\mathbb{C}_N^i$. So, each path that goes from $\Lambda_{N_{i}-M}$ to $\Delta_{N_{i}-M,N_{i}}$ to $\Delta_{N_{i},N_{i}+M}$ leaves $\Delta_{N_{i}-M,N_{i}}$ through a big cluster. Thus, by the third requirement in the definition of $\mathbb{B}_N^i$, each such path needs to take some edge $e$ to go to a small cluster $A$ in $\Delta_{N_{i} + j,N_{i}+j+1}$, for some $j\in \{0,\ldots,M-1\}$. However, by the last requirement in $\mathbb{C}_N^i$, this small cluster $A$ is connected only through the edge $e$ to $A^c$, and thus we can not continue the path. This shows that there is no open path from $\Lambda_{N_{i}-M}$ to $\Lambda_{N_{i}+M}^c$.
\end{proof}

\begin{proof}[Proof of \Cref{claim:probab}]
	We define the events $\mathbb{D}_N^i$ and $\mathbb{D}_N^{i,\star}$ by
	\begin{align*}
        &
        \mathbb{D}_N^{i,\star}  \coloneqq \Bigl\{ \left|\bigl\{x \in \Delta_{N_{i}-M,N_{i}} : \left|K_x\right| > M \bigr\} \right| \leq (\theta^\star +\eps ) \left|\Delta_{N_{i}-M,N_{i}}\right|\Bigr\}
        \\
        &
        \hspace{5mm}
		\cap \bigcap_{j=0}^{M-1} 
		\Bigl\{ \left|\bigl\{x \in \Delta_{N_{i}+j,N_{i}+j+1} : \left|K_x\right| > M \bigr\} \right| \leq (\theta^\star +\eps ) \left|\Delta_{N_{i}+j,N_{i}+j+1}\right|\Bigr\}, \text{ and }
        \\
        &
		\mathbb{D}_N^i  \coloneqq \Bigl\{ \left|\bigl\{x \in \Delta_{N_{i}-M,N_{i}} : \left|K_x\left(\Delta_{N_{i}-M,N_{i}}\right)\right| > M \bigr\} \right| \leq (\theta^\star +\eps ) \left|\Delta_{N_{i}-M,N_{i}}\right|\Bigr\}\\
        &
        \hspace{5mm}
		\cap \bigcap_{j=0}^{M-1} 
		\Bigl\{ \left|\bigl\{x \in \Delta_{N_{i}+j,N_{i}+j+1} : \left|K_x\left(\Delta_{N_{i} + j,N_{i} + j+1}\right)\right| > M \bigr\} \right| \leq (\theta^\star +\eps ) \left|\Delta_{N_{i}+j,N_{i}+j+1}\right|\Bigr\},
	\end{align*}
    respectively. The probability of the event $\mathbb{D}_N^{i,\star}$ converges to $1$ as $N\to \infty$ by the definition of $\theta^\star=\p_{J} \left( |K_0|>M\right)$ and the ergodic theorem (which applies to percolation on any amenable group \cite{lindenstrauss1999pointwise}). Since $\mathbb{D}_N^{i} \supseteq \mathbb{D}_N^{i,\star}$, the probability of the event $\mathbb{D}_N^{i}$ also converges to $1$ as $N\to \infty$.
	Let $\cF$ bet the sigma-Algebra generated by edges which have both endpoints in $\Delta_{N_{i}-M,N_{i}}$ or have both endpoints in $\Delta_{N_{i}+j,N_{i}+j+1}$ for some $j\in \{0,\ldots, M-1\}$, so that the event $\mathbb{D}_N^i$ is measurable with respect to $\cF$. The event $\mathbb{C}_N^i$ is neither increasing nor decreasing, so we can not directly apply the Harris-FKG inequality. However, conditioned on the sigma-Algebra $\cF$, both events $\mathbb{B}_N^i$ and $\mathbb{C}_N^i$ are decreasing, so that we can apply Harris-FKG to the conditional measure (which is still a product measure) to obtain that
	\begin{align*}
		\p_J \left( \mathbb{D}_N^i \cap \mathbb{B}_N^i \cap \mathbb{C}_N^i \right)
		& = \E_J \left[ \p_J \left( \mathbb{D}_N^i \cap \mathbb{B}_N^i \cap \mathbb{C}_N^i \mid \cF \right) \right]
		= \E_J \left[ \mathbbm{1}_{\mathbb{D}_N^i} \p_J \left(\mathbb{B}_N^i \cap \mathbb{C}_N^i \mid \cF \right) \right]\\
		&
		\geq \E_J \left[ \mathbbm{1}_{\mathbb{D}_N^i} \p_J \left(\mathbb{B}_N^i  \mid \cF \right)\p_J \left( \mathbb{C}_N^i \mid \cF \right) \right] .
	\end{align*}
Since $\p_J \left(\mathbb{D}_N^i\right) \to 1$ as $N\to \infty$, to prove that $\p_J \left( \mathbb{D}_N^i \cap \mathbb{B}_N^i \cap \mathbb{C}_N^i \right) \geq 0.5 N^{-\gamma}$ for all sufficiently large $N$ it suffices to prove that if $N$ is sufficiently large then
	\begin{align*}
		 \p_J \left( \mathbb{B}_N^i \mid \cF \right) (\omega)\geq  N^{-\gamma}, \qquad \text{and} \qquad \p_J \left(\mathbb{C}_N^i \mid \cF \right)(\omega)\geq 0.9
	\end{align*}
	for every $\omega \in \mathbb{D}_N^i$.

    For disjoint sets $A,B \subseteq \H_L$, we define the {\sl interaction} between $A$ and $B$ by
    \begin{equation*}
        J(A,B) \coloneqq \sum_{x \in A} \sum_{y\in B} J(\{x,y\}) .
    \end{equation*}
    Fix $N$ large and $\omega \in \mathbb{D}_N^i$. We first lower bound the conditional probability of $\mathbb{B}^i_N$. We upper bound the interaction between $  \Lambda_{N_{i} - M} $ and $ \Delta_{N_{i},\infty} $, between $ \Lambda_{N_{i}} $ and $ \Delta_{N_{i} + M,\infty}$, and between big clusters in $
	\mathcal{C}_{N_{i} - M,N_{i}}$ and big clusters in $ \mathcal{C}_{N_{i}+j,N_{i}+j+1}$ for each $j \in \{0,\ldots,M-1\}$; the event $\mathbb{B}^i_N$ holds exactly when all of the edges that we sum over in the definition of the different interactions are closed.
    The interaction between $\Lambda_{N_{i} - M} $ and $ \Delta_{N_{i},\infty} $ is upper bounded by
\begin{align}
J & \left( \Lambda_{N_{i} - M} , \Delta_{N_{i},\infty} \right) = \sum_{x \in \Lambda_{N_{i} - M}} \sum_{y \in \Delta_{N_{i},\infty} } J(\{x,y\})
\nonumber\\
&\hspace{5cm}\leq 
 L^{N_{i}-M}\sum_{j=0}^\infty \left|\Delta_{N_{i} + j,N_{i}+j+1} \right|  J \left(L^{N_{i}+j+1}\right)
 \nonumber\\
 &\hspace{5cm}\leq 
 L^{N_{i}-M}\sum_{j=0}^\infty \left(L^{N_{i}+j+1} - L^{N_{i}+j}\right) \frac{\beta^\star L \log \left(N_{i}+j+1\right)}{L^{2N_{i}+2j+2}} 
 \nonumber\\
 &\hspace{5cm}=
 L^{-M}\sum_{j=0}^\infty \frac{(L-1)\beta^\star \log \left(N_{i}+j+1\right)}{L^{1+j}}.\label{eq:bbB1}
\end{align}
Similarly, the interaction between $  \Lambda_{N_{i}}$ and $ \Delta_{N_{i}+M,\infty} $ is bounded by
\begin{align}
    & J \left( \Lambda_{N_{i}}, \Delta_{N_{i}+M,\infty} \right) = \sum_{x \in \Lambda_{N_{i}}} \sum_{ y \in \Delta_{N_{i}+M,\infty} } J( \{x,y\} )
   \nonumber \\&\hspace{4cm}\leq
    L^{N_{i}}\sum_{j=M}^\infty \left|\Delta_{N_{i} + j,N_{i}+ j+1} \right|  J \left(L^{N_{i}+j+1}\right)
    \nonumber\\&\hspace{4cm}\leq L^{N_{i}}\sum_{j=M}^\infty \left(L^{N_{i}+j+1} - L^{N_{i}+j}\right)  \frac{\beta^\star L \log \left(N_{i}+j+1\right)}{L^{2N_{i}+2j+2}}
    \nonumber \\&\hspace{4cm}= 
    \sum_{j=M}^\infty  \frac{(L-1)\beta^\star \log \left(N_{i}+j+1\right)}{L^{1+j} }.
    \label{eq:bbB2}
\end{align}
The two estimates \eqref{eq:bbB1} and \eqref{eq:bbB2} hold for every $\omega$, and do not require $\omega\in \mathbb{D}^i_N$.
Finally, for $\omega \in \mathbb{D}_N^i$, the interaction between big clusters in $
	\mathcal{C}_{N_{i} - M,N_{i}}$ and big clusters in $ \mathcal{C}_{N_{i}+j,N_{i}+j+1}$ for some $j \in \{0,\ldots,M-1\}$ is bounded by
    \begin{align}
    & \notag J \left( \left\{ x : x \in A, A \in \mathcal{C}_{N_i-M,N_i}, |A|> M \right\} , \bigcup_{j=0}^{M-1} \left\{ y : y \in B, B \in \mathcal{C}_{N_i+j,N_i+j+1}, |B|> M \right\} \right)
    \\
    & \notag = \sum_{j=0}^{M-1} J \left( \left\{ x : x \in A, A \in \mathcal{C}_{N_i-M,N_i}, |A|> M \right\} ,  \left\{ y : y \in B, B \in \mathcal{C}_{N_i+j,N_i+j+1}, |B|> M \right\} \right)
    \\
&\leq \sum_{j=0}^{M-1} (\theta^\star + \eps) |\Delta_{N_i-M,N_i}|  (\theta^\star + \eps) \left|\Delta_{N_{i}+j,N_{i}+j+1}\right| J \left(L^{N_{i}+j+1}\right)
\nonumber\\&\leq (\theta^\star + \eps)^2 L^{N_{i}} \sum_{j=0}^{M-1}  \left(L^{N_{i}+j+1} - L^{N_{i}+j}\right) \frac{\beta^\star L \log \left(N_{i}+j+1\right)}{L^{2N_{i}+2j+2}}
\nonumber\\&=(\theta^\star + \eps)^2  \sum_{j=0}^{M-1}  \frac{(L-1)\beta^\star \log \left(N_{i}+j+1\right)}{L^{1+j}}
\label{eq:bbB3}
    \end{align}
for every $\omega \in \mathbb{D}^i_N$, where we used that $\omega \in \mathbb{D}^i_N$ in the first inequality.

Let $F \subset \left\{ \{x,y\} : x,y \in V, x\neq y \right\}$ be the random (but $\cF$-measurable) set over which the sums in \eqref{eq:bbB1}, \eqref{eq:bbB2}, and \eqref{eq:bbB3} range, i.e.,
\begin{align*}
    &F = \left\{ \{x,y\} : x \in \Lambda_{N_{i}-M}, y \in \Delta_{N_{i},\infty} \right\}
    \cup \left\{ \{x,y\} : x \in \Lambda_{N_{i}}, y \in \Delta_{N_{i}+M,\infty} \right\}
    \\
    & \hspace{5cm}
    \cup \bigcup_{j=0}^{M-1} \ \bigcup_{\substack{ A \in \mathcal{C}_{N_i-M,N_i} : \\ |A|>M }} \ \bigcup_{\substack{ B \in \mathcal{C}_{N_i+j,N_i+j+1} : \\ |B|>M }} \left\{ \{x,y\} : x \in A, y \in B \right\}.
\end{align*}
Putting together the three estimates \eqref{eq:bbB1}, \eqref{eq:bbB2}, and \eqref{eq:bbB3} and using
the assumptions on $M, \beta^\star,\eps, \theta^\star$, and $\gamma$ imposed in \eqref{eq:M_beta*_assumptions} and \eqref{eq:gamma assumption}, we obtain that if $N$ is sufficiently large, then 
the sum of the three interactions bounded in \eqref{eq:bbB1}, \eqref{eq:bbB2}, and \eqref{eq:bbB3} is at most
\begin{equation*}
    \sum_{e \in F} J(e) \leq \gamma \log(N),
\end{equation*}
given $\cF$ when $\omega \in \mathbb{D}^i_N$.
Since, conditional on $\cF$, the events $\left\{e \text{ open}\right\}$ are independent for different edges $e \in F$, it follows that
\begin{align*}
    & \p_J(\mathbb{B}^i_N \mid \cF)(\omega)  = \prod_{e \in F} \left( 1 - \p_J( e \text{ open} \mid \cF)(\omega) \right)
    =
    \exp \left( - \sum_{e \in F} J(e) \right) \geq N^{-\gamma} 
\end{align*}
whenever $N$ is sufficiently large and $\omega \in \mathbb{D}^i_N$.

    It remains to bound the conditional probability of $\mathbb{C}^i_N$ for sufficiently large $N$, on the event that $\omega \in \mathbb{D}^i_N$.
    Let $C< \infty$ be a sufficiently large constant such that for all $k \in \N$ with $k\geq 2$ and $x \in \Delta_{k,k+1}$ one has \[\p_J \left(x \sim \Delta_{k,k+1}^c\right)\leq \frac{C \log(k)}{L^k}.\] This implies in particular that for all $0\leq j < M$ and
    $x \in \Delta_{N_i-M+j,N_i-M+j+1}$ one has
    \begin{equation*}
        \p_J\! \left( x \sim \Delta_{N_i-M,N_i}^c \right) 
        \leq
        \p_J\! \left( x \sim \Delta_{N_i-M+j,N_i+j+1}^c \right) \leq \frac{C \log \left( N_i-M+j \right)}{L^{N_i-M+j}} \leq \frac{C \log\left( N_i \right)}{L^{N_i-M}}.
    \end{equation*}
Writing $\tilde{\mathcal{C}}_{i,j}$
for the collection of open clusters inside $\Delta_{i,j}$ of size at most $M$ and using the BK-inequality, we deduce that
\begin{align*}
\p_J \left(\left(\mathbb{C}_N^i\right)^c \,\Big|\, \cF \right)\!(\omega) 
&\leq
		\p_J \left(
		 A \Leftrightarrow \Delta_{N_{i}-M,N_{i}}^c  \text{ for some } A \in  \tilde{\mathcal{C}}_{N_{i} - M,N_{i}} \mid	\cF \right)\!(\omega) 
         \\
         &\hspace{1.3cm}+
		\sum_{j=0}^{M-1} \p_J\left( A \Leftrightarrow \Delta_{N_{i}+j,N_{i}+j+1}^c  \text{ for some } A \in \tilde{\mathcal{C}}_{N_{i} +j,N_{i} +j+1}\right)\!(\omega)
        \\
        & \hspace{-1cm}
        \leq L^{N_i}  \left(\frac{ M C \log (N_i)}{L^{N_i-M}}\right)^2 + \sum_{j=0}^{M-1} L^{N_{i}+j+1} \left(\frac{ M C \log(N_i + j)}{L^{N_i + j}}\right)^2.
\end{align*}
This upper bound converges to zero as $N\to \infty$, concluding the proof.
\end{proof}

\begin{proof}[Proof of \Cref{theo:(dis)continuity}, part 3] It suffices to consider the case $d=1$.
	Let $\cJ$ be a proper family of kernels satisfying $\cJ \approx \frac{\log\log (n)}{n^2}$ and $\lambda_c(\cJ) \in (0,\infty)$, and let 
	\begin{equation*}
		\beta(\lambda) = \limsup_{|e|\to \infty} \cJ(\lambda,e) \frac{L^{-1}|e|^{2}}{\log\log(|e|)} .
	\end{equation*}
    As $\cJ \approx \frac{\log\log(n)}{n^2}$, we get that $\beta(\lambda)<\infty$ for all $\lambda\geq 0$. Further, since $\cJ$ is a proper family of kernels, assumption \eqref{proper:conti} implies that the function $\lambda \mapsto \beta(\lambda)$ is continuous on $(0,\infty)$. Applying Proposition \ref{propo:discont}, it follows that
	\begin{equation*}
		\theta(\lambda)^2 \beta(\lambda) = \p_\lambda \left(|K_0|=\infty\right)^2 \beta(\lambda) \in \left\{0\right\} \cup \left[1, \infty\right).
	\end{equation*}
	As $\theta(\lambda)=\p_\lambda \left(|K_0|=\infty\right)=0$ for small enough $\lambda$ and $\p_\lambda \left(|K_0|=\infty\right)>0$ for large enough $\lambda$ by assumption, this implies that $\lambda \mapsto \p_\lambda \left(|K_0|=\infty\right)$ is a discontinuous function in $\lambda$. As $\theta(\lambda) = \p_\lambda \left(|K_0|=\infty\right)$ is continuous from the right (being an infimum of the continuous increasing functions $\p_\lambda (|K_0|\geq m)$), it follows that $\theta(\lambda_c)^2 \beta(\lambda_c) \geq 1$. This shows that the phase transition is discontinuous as claimed.
\end{proof}

\section[Existence of a phase transition for a>1]{Existence of a phase transition for $a>1$}\label{sec:pt for large a}

In this section, we prove \Cref{theo:existence mixed}.
We begin by proving a number of elementary estimates concerning the probability that a certain renormalized version of the model is well-connected in \Cref{sec:finiteboxes} before using these estimates to implement our main renormalization argument in \Cref{sec:renormalization_step}.
Although the inequalities in Section \ref{sec:finiteboxes} hold true for general $a>0$, we only apply them with $a>1$.

\subsection{Connection probabilities for finitely many blocks}\label{sec:finiteboxes}

For $\ell,k \in \N$ with $\ell \leq k$, let $\Lambda_{\ell,k}$ denote the set of copies of the $\ell$-block $\Lambda_\ell$ that are contained in the $k$-block $\Lambda_k$. We write $\Lambda_{\ell,\infty}$ for the set of $\ell$-blocks in $\H_L$. For a block $\varpi \in \Lambda_{\ell,\infty}$, we also say that $\varpi$ is at level $\ell$. We will define a weighted graph structure on $\Lambda_{\ell,\infty}$ which corresponds in a certain sense to contracting each of the $\ell$-blocks in $\H_L$. 
Given $x,y \in \H_L$, we define the distance between the corresponding $\ell$-blocks $u=\Lambda_\ell(x)$ and $v=\Lambda_\ell(y)$ by $\|u-v\|_{\ell,\infty} = L^{-\ell}\|x-y\|$: this distance is well-defined independently of the choice of points $x \in \Lambda_k$ and $y\in \Lambda_k$ used to represent $u \in \Lambda_{\ell,k}$ and $v\in \Lambda_{\ell,k}$, respectively. For $u \in \Lambda_{\ell,\infty}$ we define $\Lambda_j(u) = \left\{v \in \Lambda_{\ell,\infty} : \|u-v\|_{\ell,\infty} \leq L^{j} \right\}$ as the ball of radius $L^j$ around $u$. For a set $A \subset \Lambda_{\ell,\infty}$, we write $\phi(A)$ for the set of vertices in $\mathbb{H}_L$ that correspond to the set $A$. Given $\ell\leq k$ and two disjoint subsets $A$ and $B$ of $\Lambda_{\ell,k}$, we introduce the notation
\[
J(A,B)\coloneqq \sum_{x\in \phi(A)} \sum_{y\in \phi(B)} J(\{x,y\}) 
\]
for the total weight of edges in the original model connecting the two sets of $\ell$-blocks $A$ and $B$. Since our percolation model is equivalent (up to identifying multiple parallel edges) to the model in which edges are included as a Poisson process with intensity $J$, we have that
\begin{equation}
\label{eq:Xi_large_deviation}
\p_J(A \nsim B) = \exp\left(-J(A,B) \right)
\end{equation}
for any disjoint sets of $\ell$-blocks $A$ and $B$.

Given the bond percolation configuration $\omega$ on $\mathbb{H}_L$, we write $\omega_{\ell,k}$ for the random graph with vertex set $\Lambda_{\ell,k}$ in which two distinct $\ell$-blocks are connected in $\omega_{\ell,k}$ if there is an edge of $\omega$ with one endpoint in each of the two blocks. Our first lemma shows that this graph is very well connected with high probability when $\ell$ is large and $k-\ell$ is a constant.
(Note that this lemma does not require any symmetry assumptions on the kernel $J$ defined as a function on unordered pairs of points on the hierarchical lattice.)

\begin{lemma}\label{lem:allbutone}
    Let $a,\delta>0$ and let $J$ be a kernel such that
    if $j$ is sufficiently large then $J(e) \geq (a+\delta) L\log(j)L^{-2j}$  for every edge $e$ with $|e|=L^j$. Then there exists $k_0 = k_0(\delta, L) < \infty$ such that for all $k \geq k_0$ there exists $N =N(k,J)$ and $C=C(k,J)$ such that
	\begin{align}\label{eq:allbutone}
		\p_{J} \left( \Lambda_{n,n+k}\setminus\{\Lambda_n(0)\} \textnormal{ is not connected off $\Lambda_n(0)$ in $\omega_{n,n+k}$} \right)
		\leq C n^{-\frac{a}{2}}.
	\end{align}
 for all $n \geq N$.
\end{lemma}

It is important here that we get the sharp power, with the factor $1/2$, up to an $o(1)$ error. The presence of the ``defect'' (i.e., the fact that we are not allowed to use the block $\Lambda_n(0)$) also makes the analysis significantly more complicated; The analogous statement without the ``defect" is treated in inequality \eqref{eq:inside connection 1/2} below. Further, the factor of $\frac{1}{2}$ in the exponent is essentially optimal, as shown in the following Remark.

\begin{remark}\label{rem:allbutone}
    Let $L=2$, $a>0$, and let $J$ be a kernel on $\H_2$ such that
    $J(e) \leq a 2\log(j)2^{-2j}$  for every edge $e$ with $|e|=2^j$. Then for all fixed $k \in \N$ one has
	\begin{align*}
		\p_{J} \left( \Lambda_{n,n+k}\setminus\{\Lambda_n(0)\} \textnormal{ is not connected off $\Lambda_n(0)$ in $\omega_{n,n+k}$} \right)
		\geq c n^{-\frac{a}{2}}.
	\end{align*}
    for some constant $c=c(k)>0$.
\end{remark}

\begin{proof}
    Define the sets $A = \Lambda_{n,n+k-1} \setminus \{\Lambda_n(0)\}$ and $B = \Lambda_{n,n+k}\setminus \left( A \cup \{\Lambda_n(0) \} \right)$. Then $|A|=2^{k-1}-1,|B|=2^{k-1}$, and 
    \begin{align*}
        J(A,B) = |A|2^{n} |B| 2^{n} \frac{a 2 \log(n+k)}{\left(2^{n+k}\right)^2} = \left(2^{k-1}-1\right)2^{n} 2^{k-1} 2^{n} \frac{a 2 \log(n+k)}{\left(2^{n+k}\right)^2} \leq \frac{a\log(n+k)}{2},
    \end{align*}
    so that
    \begin{align*}
        & \p_{J} \left( \Lambda_{n,n+k}\setminus\{\Lambda_n(0)\} \text{ is not connected off $\Lambda_n(0)$ in $\omega_{n,n+k}$} \right)
        \\
        &
        \hspace{17mm} \geq \p_{J} ( A \nsim B ) = \exp \left( -J(A,B) \right)  \geq \exp \left( - \frac{a\log(n+k)}{2} \right) = (n+k)^{-\frac{a}{2}} \geq c n^{-\frac{a}{2}}
    \end{align*}
    for some constant $c=c(k)>0$.
\end{proof}

We proceed with the proof of Lemma \ref{lem:allbutone}.
\begin{proof}[Proof of Lemma~\ref{lem:allbutone}]
    We stress that all constants in this proof may depend on $k$, making a rather crude analysis possible. 
	For a set $A\subset \Lambda_{n,n+k} \setminus \left\{ \Lambda_n(0) \right\}$, we define the set $\bar{A} \subset \Lambda_{n,n+k} \setminus \left\{ \Lambda_n(0) \right\} $ by $\bar{A} = \Lambda_{n,n+k}\setminus\left(\{\Lambda_n(0)\}\cup A\right)$. By a union bound, we have that
	\begin{align*}
		\p_{J} \left( \Lambda_{n,n+k} \setminus \left\{ \Lambda_n(0) \right\} \text{ is not connected} \right) 
		\leq \sum_{ \emptyset \subsetneq A \subsetneq \Lambda_{n,n+k} \setminus \left\{ \Lambda_n(0) \right\}} \p_{J} \left(A \nsim \bar{A}\right)
	\end{align*}
	and thus it suffices to prove that $\p_{J} \left(A \nsim \bar{A}\right) \leq C n^{-\frac{a}{2}}$ for a constant $C$ depending on $k$ and $a$ only. In order to prove this, it suffices by \eqref{eq:Xi_large_deviation} to prove that
	\begin{align*}
		J(A,\bar{A})  \geq \frac{a}{2} \log(n)
	\end{align*}
    for all sufficiently large $n$.  We first prove this estimate when $A$ belongs to one of three special classes of sets, which we treat separately, before showing that the general case can be reduced to these three cases.

    \medskip
	
	\noindent \textbf{Case 1:}
	First assume that there exists $u \in \Lambda_{n,n+k}$ and $\ell\leq k-1$ such that 
	\begin{align}\label{eq:case1}
	\Lambda_{\ell}(u) \cap A \neq \emptyset,\qquad \Lambda_{\ell}(u) \cap \bar{A} \neq \emptyset,\qquad \text{ and }\qquad \Lambda_n(0) \notin \Lambda_{\ell}(u).
	\end{align}
	In this case, there exists $\tilde{\ell} \leq \ell$ and $\tilde{u} \in \Lambda_{\ell}(u)$ such that   $\Lambda_{\tilde{\ell}} (\tilde{u})$ intersects both $A$ and $\bar{A}$ and if we write $\Lambda_{\tilde{\ell}} (\tilde{u}) = \bigcup_{i=1}^{L} \Lambda_{\tilde{\ell}-1} (\tilde{u}_i)$ for the covering of the $\tilde{\ell}$-block at $\tilde{u}$ by $(\tilde{\ell}-1)$-blocks then we can order the $\tilde{u}_i$ so that
	\begin{equation*}
	\Lambda_{\tilde{\ell}-1} (\tilde{u}_1), \ldots, \Lambda_{\tilde{\ell}-1} (\tilde{u}_m) \subset A, \qquad \text{ and } \qquad \Lambda_{\tilde{\ell}-1} (\tilde{u}_{m+1}), \ldots, \Lambda_{\tilde{\ell}-1} (\tilde{u}_L) \subset \bar{A}
	\end{equation*}
    for some $m \in \{1,\ldots, L-1\}$. Indeed, if $\tilde{\ell}\in \{1,\ldots,\ell\}$ is minimal such that there exists $u \in \Lambda_{n,n+k}$ satisfying \eqref{eq:case1}, then $\Lambda_{\tilde{\ell}}(u)$ will have this property. 
    In this situation, we have that
	\begin{align*}
	&J\!\left( A, \Bar{A} \right) \geq \sum_{i=1}^{m} \sum_{j=m+1}^{L} J\!\left( \Lambda_{\tilde{\ell}-1} (\tilde{u}_i) , \Lambda_{\tilde{\ell}-1} (\tilde{u}_j) \right)
    \\
    &
    \hspace{4cm}
    \geq \sum_{i=1}^{m} \sum_{j=m+1}^{L} \left(L^{n+\tilde{\ell}-1}\right)^2 \frac{(a+\delta)L \log  (n+\tilde{\ell}-1)}{L^{2(n+\tilde{\ell})}}
    \\
    &
    \hspace{4cm}
	\geq   \frac{a m(L-m) \log (n)}{L} \geq \frac{a }{2}\log(n) ,
	\end{align*}
	where the last inequality holds since $m(L-m) \geq L/2$.

\medskip
    
	\noindent \textbf{Case 2:}
	Next, assume that $A$ is of the form $A= \Lambda_{\ell}(\Lambda_n(0))\setminus \{\Lambda_n(0)\} \cup \bigcup_{i=2}^{m} \Lambda_{\ell} (u_i)$
    for some $\ell \in \{0,\ldots,k-1\}$,  $m \in \{1,\ldots, L-1\}$, and $u_2,\ldots,u_m \in \Lambda_{n,n+k}$ such that the set $\Lambda_{\ell}(\Lambda_n(0))\setminus \{\Lambda_n(0)\}$ and the blocks $\Lambda_{\ell} (u_2),  \ldots, \Lambda_{\ell} (u_m)$ are mutually disjoint
    and $A \subset \Lambda_{\ell+1}(\Lambda_n(0))$. In this situation we have $|A| = L^{\ell}-1 + (m-1) L^{\ell} = m L^{\ell}-1$ and $|\Bar{A}   \cap \Lambda_{\ell+1} (\Lambda_{n}(0))| = (L-m)L^{\ell}$. Further, for all $j \in \{\ell+1,\ldots,k-1\}$, we have $A \subseteq \Lambda_{\ell+1}(\Lambda_{n}(0) ) \subseteq \Lambda_{j}(\Lambda_{n}(0))$ and thus 
    \begin{equation*}
        \left| \Bar{A}   \cap \left( \Lambda_{j+1} (\Lambda_{n}(0)) \setminus \Lambda_{j} (\Lambda_{n}(0)) \right) \right| = \left|\Lambda_{j+1} (\Lambda_{n}(0)) \setminus \Lambda_{j} (\Lambda_{n}(0)) \right| = (L-1)L^j .
    \end{equation*}
    Note that the case $\ell = k-1$ is slightly special, in the sense that there is no such $j\in \{\ell+1,\ldots,k-1\}$ satisfying the above condition, since $\{\ell+1,\ldots,k-1\} = \emptyset$ in this case. In the following calculation, the case $\ell=k-1$ corresponds to empty sums, but nevertheless, one checks that each step is valid both when $\ell=k-1$ as well as when $\ell \in \{0,\ldots,k-1\}$. The interaction between $A$ and $\Bar{A}$ can be bounded from below by
	\begin{align*}
        & J(A,\Bar{A}) 
        =
        J\left( A, \Bar{A}   \cap \Lambda_{\ell+1} (\Lambda_{n}(0)) \right) + \sum_{j=\ell+1}^{k-1} J\left( A, \Bar{A}   \cap \left( \Lambda_{j+1} (\Lambda_{n}(0)) \setminus \Lambda_{j} (\Lambda_{n}(0)) \right) \right)
		\\
        &\hspace{1cm}
		\geq (m L^{\ell}-1)L^n \left((L-m)L^{\ell}\right) L^n \frac{(a+\delta) L \log(n+\ell)}{L^{2(n+\ell +1)}} 
        \\
        &
        \hspace{5cm} +
		\sum_{j=\ell + 1}^{k-1} \left(m L^{\ell}-1\right) L^n (L-1)L^{j+n} \frac{(a+\delta) L \log(n+j)}{L^{2(n+j+1)}}\\
		&\hspace{1cm}
		\geq \left(m -\frac{1}{L^{\ell}} \right) (L-m)  \frac{(a+\delta)}{L} \log(n) +
		\sum_{j=\ell + 1}^{k-1} \left(m -\frac{1}{L^{\ell}} \right) (L-1) \frac{(a+\delta) \log(n)}{L^{1+j-\ell}}\\
		&\hspace{1cm}
		= \left(m -\frac{1}{L^{\ell}} \right) (L-m)  \frac{(a+\delta)}{L} \log(n) +
		\frac{\left(m -\frac{1}{L^{\ell}} \right) (L-1) (a+\delta) \log(n)}{L}
		\sum_{j=1}^{k-1-\ell}  \frac{1}{L^{j}}\\
		&\hspace{1cm}
		= \left(m -\frac{1}{L^{\ell}} \right) (L-m)  \frac{(a+\delta)}{L} \log (n) +
		\frac{\left(m -\frac{1}{L^{\ell}} \right) (L-1) (a+\delta) \log(n)}{L}
		\frac{1-L^{1-(k-\ell)}}{L-1}\\
		&\hspace{1cm}
		= \left(m -\frac{1}{L^{\ell}} \right) (L-m)  \frac{(a+\delta)}{L} \log (n) +
		\frac{\left(m -\frac{1}{L^{\ell}} \right)  (a+\delta) \log(n)}{L}
		(1-L^{1-(k-\ell)})\\
		&\hspace{1cm}
		= \left(m -\frac{1}{L^{\ell}} \right) \left(L-m + 1-L^{1-(k-\ell)}\right)  \frac{(a+\delta)}{L} \log (n).
	\end{align*}
As such, it suffices to show that for $k$ large enough and $\ell \in \{0,\ldots,k-1\}$ one has 
	\begin{equation}\label{eq:technical inequality}
		\left(m -\frac{1}{L^{\ell}} \right) \left(L-m + 1-L^{1-(k-\ell)}\right)  \frac{a+\delta}{L} \geq \frac{a}{2}.
	\end{equation}
	This is clear for $\ell$ large enough, as
 \begin{equation*}
     \left(m -\frac{1}{L^{\ell}} \right) \left(L-m + 1-L^{1-(k-\ell)}\right)  \frac{a+\delta}{L}
     \geq
     \left(m -\frac{1}{L^{\ell}} \right) \left(L-m \right)  \frac{a+\delta}{L}
     \geq (L-1)\frac{a}{L} \geq \frac{a}{2},
 \end{equation*}
 where the second-to-last inequality holds for $\ell$ large enough, since $m(L-m) \geq L-1$.
 Next, let us consider small $\ell$, say $\ell \leq \frac{k}{2}$. In this situation, we have
 \begin{equation*}
     \left(m -\frac{1}{L^{\ell}} \right) \left(L-m + 1-L^{1-(k-\ell)}\right) \geq \left(m -\frac{1}{L^{\ell}} \right) \left(L-m + 1-L^{1-k/2}\right)
 \end{equation*}
 and thus, in order to show \eqref{eq:technical inequality} for $k$ large enough, it suffices to show that
	\begin{equation}\label{eq:inequality half}
		\left(m -\frac{1}{L^{\ell}} \right) \left(L-m + 1\right)  \geq \frac{L}{2}.
	\end{equation}
    In order to show \eqref{eq:inequality half}, it suffices to consider $\ell \in \{0,1\}$. For $\ell=0$, we need to have $m\geq 2, L \geq 3$ by construction. One checks that $(m-1)(L-(m-1)) \geq L/2$ in this case. Let us consider $\ell=1$ now. Here
	\begin{equation*}
		\left(m -\frac{1}{L} \right) \left(L-m + 1\right) \geq \left(m - \frac{1}{2} \right) \left(L- m + 1\right)  \geq  \frac{L}{2}
	\end{equation*}
	where the last inequality follows by distinguishing the cases $m=1$ and $m\geq 2$.
	This shows \eqref{eq:inequality half} (and thus \eqref{eq:technical inequality}), which finishes the proof of Case 2.\\

    \medskip
    
	\noindent \textbf{Case 3:}
	Next, assume that $A$ is of the form $A= \bigcup_{i=1}^{m} \Lambda_{\ell} (u_i)$
    for some $\ell \in \{0,\ldots,k-1\}$,  $m \in \{1,\ldots, L-1\}$, and $u_1,\ldots,u_m \in \Lambda_{n,n+k}$ such that the blocks $\Lambda_{\ell} (u_1),  \ldots, \Lambda_{\ell} (u_m)$ are mutually disjoint, $A \subset \Lambda_{\ell+1}(\Lambda_n(0))$, but $\Lambda_n(0) \notin A$. In this situation we have $|A| = m L^{\ell}$ and $|\Bar{A} \cap \Lambda_{\ell+1} (\Lambda_{n}(0))| = (L-m)L^{\ell}-1$. Further, for all $j \in \{\ell+1,\ldots,k-1\}$, we have $A \subseteq \Lambda_{\ell+1}(\Lambda_{n}(0) ) \subseteq \Lambda_{j}(\Lambda_{n}(0))$ and thus 
    \begin{equation*}
        \left| \Bar{A}   \cap \left( \Lambda_{j+1} (\Lambda_{n}(0)) \setminus \Lambda_{j} (\Lambda_{n}(0)) \right) \right| = \left|\Lambda_{j+1} (\Lambda_{n}(0)) \setminus \Lambda_{j} (\Lambda_{n}(0)) \right| = (L-1)L^j .
    \end{equation*}
    Note that the case $\ell = k-1$ is slightly special, in the sense that there is no such $j\in \{\ell+1,\ldots,k-1\}$ satisfying the above condition, since $\{\ell+1,\ldots,k-1\} = \emptyset$ in this case. In the following calculation, the case $\ell=k-1$ corresponds to empty sums, but nevertheless, one checks that each step is valid both when $\ell=k-1$ as well as when $\ell \in \{0,\ldots,k-1\}$. The interaction between $A$ and $\Bar{A}$ can be bounded from below by
	\begin{align*}
        & J(A,\Bar{A}) 
        =
        J\left( A, \Bar{A}   \cap \Lambda_{\ell+1} (\Lambda_{n}(0)) \right) + \sum_{j=\ell+1}^{k-1} J\left( A, \Bar{A}   \cap \left( \Lambda_{j+1} (\Lambda_{n}(0)) \setminus \Lambda_{j} (\Lambda_{n}(0)) \right) \right)
		\\
        &\hspace{1cm}
		\geq (m L^{\ell})L^n \left((L-m)L^{\ell}-1\right) L^n \frac{(a+\delta) L \log(n+\ell)}{L^{2(n+\ell +1)}} 
        \\
        &
        \hspace{5cm} +
		\sum_{j=\ell + 1}^{k-1} \left(m L^{\ell}\right)L^n (L-1)L^{j+n} \frac{(a+\delta) L \log(n+j)}{L^{2(n+j+1)}}\\
		&\hspace{1cm}
		\geq m \left(L-m -\frac{1}{L^{\ell}} \right)  \frac{(a+\delta)}{L} \log(n) +
		\sum_{j=\ell + 1}^{k-1} m (L-1) \frac{(a+\delta) \log(n)}{L^{1+j-\ell}}\\
		&\hspace{1cm}
		= m \left(L-m -\frac{1}{L^{\ell}} \right)  \frac{(a+\delta)}{L} \log(n) +
		\frac{m (L-1) (a+\delta) \log(n)}{L}
		\sum_{j=1}^{k-1-\ell}  \frac{1}{L^{j}}\\
		&\hspace{1cm}
		= m \left(L-m -\frac{1}{L^{\ell}} \right)  \frac{(a+\delta)}{L} \log (n) +
		\frac{ m (L-1) (a+\delta) \log(n)}{L}
		\frac{1-L^{1-(k-\ell)}}{L-1}\\
		&\hspace{1cm}
		= m \left(L-m -\frac{1}{L^{\ell}} \right)  \frac{(a+\delta)}{L} \log (n) +
		\frac{ m  (a+\delta) \log(n)}{L}
		(1-L^{1-(k-\ell)})\\
		&\hspace{1cm}
		= m \left(L-m -\frac{1}{L^{\ell}} + 1-L^{1-(k-\ell)}\right)  \frac{(a+\delta)}{L} \log (n).
	\end{align*}
As such, it suffices to show that for $k$ large enough and $\ell \in \{0,\ldots,k-1\}$ one has 
	\begin{equation}\label{eq:technical inequality v2}
		m \left(L-m -\frac{1}{L^{\ell}} + 1-L^{1-(k-\ell)}\right)  \frac{(a+\delta)}{L} \geq \frac{a}{2}.
	\end{equation}
	This is clear for $\ell$ large enough, as
 \begin{equation*}
     m \left(L-m -\frac{1}{L^{\ell}} + 1-L^{1-(k-\ell)}\right)  \frac{a+\delta}{L}
     \geq
     m \left(L-m -\frac{1}{L^{\ell}} \right)  \frac{a+\delta}{L}
     \geq m (L-m) \frac{a}{L} \geq \frac{a}{2},
 \end{equation*}
 where the second-to-last inequality holds for $\ell$ large enough.
 Next, let us consider small $\ell$, say $\ell \leq \frac{k}{2}$. In this situation, we have
 \begin{equation*}
     m \left(L-m -\frac{1}{L^{\ell}} + 1-L^{1-(k-\ell)}\right) \geq m \left(L- m -\frac{1}{L^{\ell}} + 1-L^{1-(k/2)}\right)
 \end{equation*}
 and thus, in order to show \eqref{eq:technical inequality v2} for $k$ large enough, it suffices to show that
	\begin{equation}\label{eq:inequality half v2}
		m \left(L-m -\frac{1}{L^{\ell}} + 1\right)  \geq \frac{L}{2}.
	\end{equation}
    In order to show \eqref{eq:inequality half v2}, it suffices to consider $\ell = 0$. For $\ell=0$, we have that
	\begin{equation*}
		m \left(L-m -\frac{1}{L^{\ell}} + 1\right) = m \left( L - m \right) \geq L-1 \geq \frac{L}{2} .
	\end{equation*}
	This shows \eqref{eq:inequality half v2} (and thus \eqref{eq:technical inequality v2}), which finishes the proof of Case 3.\\
	
	\medskip
	
\noindent	\textbf{The general case: }
	Finally, we consider general sets $A$. 
	Instead of directly proving the statement for general sets $A, \bar{A}$, we will reduce the problem to one of the three cases from before. So let $A \subsetneq \Lambda_{n,n+k}\setminus \{\Lambda_n(0)\}$ be a non-empty set, and let 
	\begin{equation*}
		\ell = \inf \left\{j \in \{1,\ldots, k\} : \Lambda_j(\Lambda_n(0)) \cap A \neq \emptyset, \Lambda_j(\Lambda_n(0)) \cap \bar{A} \neq \emptyset \right\} .
	\end{equation*}
    Since $A \subsetneq \Lambda_{n,n+k} \setminus \left\{ \Lambda_n(0) \right\}$, we have $\ell \in \{1,\ldots,k\}$. By symmetry, we can assume without loss of generality that $\Lambda_{\ell-1}\left( \Lambda_{n}(0) \right) \setminus \left\{ \Lambda_{n}(0) \right\} \subseteq A$, interchanging the roles of $A$ and $\bar{A}$ if this is not the case.
    We split the set $A$ into an {\sl inner part} and an {\sl outer part} by
	\begin{align*}
		A_{\ell}^i \coloneqq A \cap \Lambda_{\ell}(\Lambda_n(0)) \qquad \text{ and }  \qquad A_{\ell}^o \coloneqq A \setminus A_{\ell}^i
	\end{align*}
	and in the same way split $\bar{A}$ into the inner and the outer part
	\begin{align*}
		\bar{A}_{\ell}^i \coloneqq  \bar{A} \cap \Lambda_{\ell}(\Lambda_n(0)) \qquad \text{ and } \qquad \bar{A}_{\ell}^o \coloneqq \bar{A} \setminus \bar{A}_{\ell}^i.
	\end{align*}
	The superscripts $i$ and $o$ stand for the {\sl inner} and {\sl outer} part of $A$ or $\bar{A}$, as appropriate.
	Observe that
	\begin{align*}
		J(A,\bar{A}) &= J(A_{\ell}^i,\bar{A}_{\ell}^i) + J(A_{\ell}^i,\bar{A}_{\ell}^o) + J(A_{\ell}^o,\bar{A}_{\ell}^i) + J(A_{\ell}^o,\bar{A}_{\ell}^o) .
	\end{align*}
	First, let us assume that $|A_{\ell}^i| \leq |\bar{A}_{\ell}^i|$. Then $J (A_{\ell}^i , A_{\ell}^o) \leq J (\bar{A}_{\ell}^i, A_{\ell}^o)$, since for given $A_{\ell}^o$ and $B \subset \Lambda_{\ell}(\Lambda_n(0))$ the quantity $J (B , A_{\ell}^o)$ is proportional to $|B|$. Thus we get that
	\begin{align*}
		J(A_{\ell}^i, A_{\ell}^o \cup \bar{A}_{\ell}^i \cup \bar{A}_{\ell}^o ) & =
		J(A_{\ell}^i, A_{\ell}^o ) + J(A_{\ell}^i, \bar{A}_{\ell}^i ) + J(A_{\ell}^i,  \bar{A}_{\ell}^o )\\
		&
		\leq J(\bar{A}_{\ell}^i, A_{\ell}^o ) + J(A_{\ell}^i, \bar{A}_{\ell}^i ) + J(A_{\ell}^i,  \bar{A}_{\ell}^o )\\
		&
		\leq J(\bar{A}_{\ell}^i, A_{\ell}^o ) + J(A_{\ell}^i, \bar{A}_{\ell}^i ) + J(A_{\ell}^i,  \bar{A}_{\ell}^o ) + J(A_{\ell}^o , \bar{A}_{\ell}^o) = J (A, \bar{A}).
	\end{align*}
    If $|A_{\ell}^i| > |\Bar{A}_{\ell}^i|$, then $J (\Bar{A}_{\ell}^i , \Bar{A}_{\ell}^o) \leq J (A_{\ell}^i, \Bar{A}_{\ell}^o)$. This implies that
        \begin{align*}
		J(\bar{A}_{\ell}^i, \bar{A}_{\ell}^o \cup A_{\ell}^i \cup A_{\ell}^o ) & = J(\bar{A}_{\ell}^i, \bar{A}_{\ell}^o  ) + J(\bar{A}_{\ell}^i, A_{\ell}^i  ) + J(\bar{A}_{\ell}^i, A_{\ell}^o )
        \\
        &
        \leq
        J(A_{\ell}^i, \bar{A}_{\ell}^o  ) + J(\bar{A}_{\ell}^i, A_{\ell}^i  ) + J(\bar{A}_{\ell}^i, A_{\ell}^o )
        \\
        &
        \leq
        J(A_{\ell}^i, \bar{A}_{\ell}^o  ) + J(\bar{A}_{\ell}^i, A_{\ell}^i  ) + J(\bar{A}_{\ell}^i, A_{\ell}^o ) + J(A_{\ell}^o, \bar{A}_{\ell}^o)
        = J (A, \bar{A}).
	\end{align*}
    In any case, we see that
    \begin{equation*}
         J (A, \bar{A}) \geq \min \left( J(A_{\ell}^i, \bar{A}_{\ell}^o \cup \bar{A}_{\ell}^i \cup A_{\ell}^o ), J(\bar{A}_{\ell}^i, \bar{A}_{\ell}^o \cup A_{\ell}^i \cup A_{\ell}^o ) \right) .
    \end{equation*}
    Let us first assume that $J(A_{\ell}^i, \bar{A}_{\ell}^o \cup \bar{A}_{\ell}^i \cup A_{\ell}^o ) \leq J(\bar{A}_{\ell}^i, \bar{A}_{\ell}^o \cup A_{\ell}^i \cup A_{\ell}^o )$. If there exists $u \in A_{\ell}^i$ such that $\Lambda_n(0) \notin \Lambda_{\ell-1}(u)$, $\Lambda_{\ell-1}(u) \cap A \neq \emptyset$, and $\Lambda_{\ell-1}(u) \cap \bar{A} \neq \emptyset$, then we get that
	\begin{align*}
		J(A,\bar{A}) \geq J(A_{\ell}^i, A_{\ell}^o \cup \bar{A}_{\ell}^i \cup \bar{A}_{\ell}^o ) \geq \frac{a}{2} \log(n)
	\end{align*}
    by Case 1. If there does not exist such a $u \in A_{\ell}^{i}$, then we know that for all $u \in \Lambda_{\ell} (\Lambda_n(0)) \setminus \Lambda_{\ell-1}(\Lambda_n(0))$ either $\Lambda_{\ell-1}(u) \subset A$ or $\Lambda_{\ell-1}(u) \subset \bar{A}$ , and that $\Lambda_{\ell-1}(\Lambda_n(0)) \setminus \{\Lambda_n(0)\} \subset A$ (which we assumed above). Thus, set $A_{\ell}^i$ is of the form as described in Case 2, and thus we also get that
    \begin{align*}
		J(A,\bar{A}) \geq J(A_{\ell}^i, A_{\ell}^o \cup \bar{A}_{\ell}^i \cup \bar{A}_{\ell}^o ) \geq \frac{a}{2} \log(n)
	\end{align*}
as required. We are left to consider the case where $J(A_{\ell}^i, \bar{A}_{\ell}^o \cup \bar{A}_{\ell}^i \cup A_{\ell}^o ) > J(\bar{A}_{\ell}^i, \bar{A}_{\ell}^o \cup A_{\ell}^i \cup A_{\ell}^o )$. Again, if there exists $u \in A_{\ell}^i$ such that $\Lambda_n(0) \notin \Lambda_{\ell-1}(u)$, $\Lambda_{\ell-1}(u) \cap A \neq \emptyset$, and $\Lambda_{\ell-1}(u) \cap \bar{A} \neq \emptyset$, then we get that
	\begin{align*}
		J(A,\bar{A}) \geq J(A_{\ell}^i, A_{\ell}^o \cup \bar{A}_{\ell}^i \cup \bar{A}_{\ell}^o ) \geq \frac{a}{2} \log(n)
	\end{align*}
    by Case 1. If there does not exist such a $u \in A_{\ell}^{i}$, then we know that for all $u \in \Lambda_{\ell} (\Lambda_n(0)) \setminus \Lambda_{\ell-1}(\Lambda_n(0))$ either $\Lambda_{\ell-1}(u) \subset A$ or $\Lambda_{\ell-1}(u) \subset \bar{A}$ , and that $\Lambda_{\ell-1}(\Lambda_n(0)) \setminus \{\Lambda_n(0)\} \subset A$. Thus, the set $\bar{A}_{\ell}^i$ is of the form as described in Case 3, and thus we also get that
    \begin{align*}
		J(A,\bar{A}) \geq J(\bar{A}_{\ell}^i, \bar{A}_{\ell}^o \cup A_{\ell}^i \cup A_{\ell}^o ) \geq \frac{a}{2} \log(n)
	\end{align*}
as required.
\end{proof} 

Our next lemma bounds the probability of the existence of a cluster of small density in $\omega_{n,n+k}$ with $k$ constant and $n$ large and deduces bounds on the probability that $\omega_{n,n+k}$ fails to be connected in the same regime.

\begin{lemma}\label{lem:semi-good}
	Let $a>0$ and let $J : \H_L \to \R_{\geq 0}$ be a kernel so that $J\left(e\right) \geq \frac{a L \log (j)}{L^{2j}}$ for every edge $e$ with $|e|=L^j$. Then, for all $k\in \N$ there exists a constant $C=C(k,J)$ such that \begin{align}\label{eq:inside connection 1/2}
	&\p_{J} \left( \Lambda_{n,n+k} \textnormal{ is not connected} \right) \leq C n^{-\tfrac{a}{2}} \textnormal{ and}\\
    \label{eq:outside connection 1}
    & \p_{J} \left( \Lambda_{n,n+k} \textnormal{ has a cluster of density $\gamma$} \right) \leq C  n^{-a(1-\gamma)}
	\end{align}
    for every $n\geq 1$ and $\gamma \in \{L^{-k},2L^{-k},\ldots,1-L^{-k},1\}$.
\end{lemma}

Similarly to the statements of Lemma \ref{lem:allbutone} and Remark \ref{rem:allbutone}, respectively, we emphasize that it is important here to get the sharp power, with the factor $1/2$ in \eqref{eq:inside connection 1/2}, and that this power is essentially the best possible.

\begin{remark}
    Let $L=2, a>0$ and let $J : \H_2 \to \R_{\geq 0}$ be a kernel so that $J\left(e\right) \leq \frac{a 2 \log (j)}{2^{2j}}$ for every edge $e$ with $|e|=2^j$. Then, for all $k\in \N$ there exists a constant $c=c(k)>0$ such that \begin{align}\label{eq:inside connection 1/2 remark}
	&\p_{J} \left( \Lambda_{n,n+k} \textnormal{ is not connected} \right) \geq c n^{-\tfrac{a}{2}}
	\end{align}
    for every $n\geq 1$.
\end{remark}

\begin{proof}
    Define the sets $A=\Lambda_{n,n+k-1}$ and $B=\Lambda_{n,n+k}\setminus A$. Then 
    \begin{equation*}
        J(A,B) \leq 2^{n+k-1} 2^{n+k-1} \frac{a 2 \log(n+k)}{\left( 2^{n+k} \right)^2} = \frac{a \log(n+k)}{2}
    \end{equation*}
    so that
    \begin{equation*}
        \p_{J} \left( \Lambda_{n,n+k} \text{ is not connected} \right) \geq \p_{J} ( A \nsim B ) = \exp \left( -J(A,B) \right) \geq (n+k)^{-\tfrac{a}{2}} \geq c n^{-\tfrac{a}{2}}
    \end{equation*}
    for some constant $c=c(k)>0$.
\end{proof}

Next, we go to the proof of Lemma \ref{lem:semi-good}.

\begin{proof}[Proof of \Cref{lem:semi-good}]
We stress that all constants in this proof may depend on $k$, making a rather crude analysis possible. 
    First note that \eqref{eq:inside connection 1/2} directly follows from \eqref{eq:outside connection 1}, as whenever $\Lambda_{n,n+k}$ is not connected, it must contain a cluster of density at most $1/2$. Thus, we get that
    \begin{align*}
         \p_{J} \left( \Lambda_{n,n+k} \text{ is not connected} \right) &\leq \sum_{\substack{\gamma \in \{L^{-k},2L^{-k},\ldots\}: \\ \gamma \leq \frac{1}{2}}}
        \p_{J} \left( \Lambda_{n,n+k} \text{ has a cluster of density $\gamma$} \right) \\
        &
        \leq \sum_{\substack{\gamma \in \{L^{-k},2L^{-k},\ldots\}: \\ \gamma \leq \frac{1}{2}}} C n^{-a(1-\gamma)}
        \leq \widetilde{C} n^{-\frac{a}{2}}
    \end{align*}
    for some constant $\widetilde{C}$ depending on $C$ and $k$. In order to prove \eqref{eq:outside connection 1} it suffices to show that for each set $\emptyset \subsetneq A\subsetneq \Lambda_{n,n+k}$ with $|A|=\gamma L^k$
	\begin{equation}\label{eq:inside connection prob upperbound}
	\p_J \left( A \nsim \Lambda_{n,n+k}\setminus A \right) = \exp \left( - J \left( A , \Lambda_{n,n+k}\setminus A \right) \right)   \leq n^{-a (1-\gamma)},
	\end{equation}
    with the claim then following by a union bound over all sets $A \subseteq \Lambda_{n,n+k}$ with $|A|=\gamma L^k$.
    For a set $A \subseteq \Lambda_{n,n+k}$ define $A^c = \Lambda_{n,n+k}\setminus A$.
    In order to prove \eqref{eq:inside connection prob upperbound}, it suffices to show that for all non-empty sets $A \subset \Lambda_{n,n+k}$ with $|A|=\gamma L^k$, $\gamma \in \left\{ \frac{1}{L^k}, \frac{2}{L^k}, \ldots, \frac{L^k-1}{L^k} \right\}$ one has
	\begin{align}\label{eq:prob bound 1-eps}
	J \left(A , A^c\right) \geq a (1-\gamma) \log(n).
	\end{align}

    We first consider a special case and then reduce the general case to this special case.

    \medskip

    \noindent \textbf{The special case.}
    We start with the case where $A$ is of the form $A=\bigcup_{i=1}^{m} \Lambda_{\ell}(v_i)$ with $\ell \in \{0,\ldots,k-1\}, m \in \{1,\ldots,L-1\}$, and $v_1,\ldots,v_m \in \Lambda_{n,n+k}$ such that the sets $\Lambda_{\ell}(v_1), \ldots, \Lambda_{\ell}(v_m)$ are disjoint with $\Lambda_{\ell}(v_i) \subset \Lambda_{\ell+1}(v_1)$. Here $\gamma = \frac{|A|}{L^k}= m L^{\ell-k}$ and
    \begin{align}
        &\notag J(A,A^c)  = J (A, \Lambda_{\ell+1}(v_1)\setminus A) + \sum_{j=\ell+1}^{k-1}J (A, \Lambda_{j+1}(v_1) \setminus \Lambda_{j}(v_1) )\\
        \notag &\hspace{5mm} \geq m L^{n+\ell} (L-m) L^{n+\ell} \frac{a L \log(n+\ell+1)}{L^{2(n+\ell+1)}} + \sum_{j=\ell+1}^{k-1} m L^{n+\ell}(L-1) L^{n+j} \frac{a L \log(n+j+1)}{L^{2(n+j+1)}}\\
        \notag &\hspace{5mm} \geq (L-1)  \frac{a \log(n)}{L} + \sum_{j=\ell+1}^{k-1} (L-1)  \frac{a \log(n)}{L^{1+j-\ell}}
        = (L-1)  \frac{a \log(n)}{L}  \sum_{j=\ell}^{k-1}  \frac{1}{L^{j-\ell}}\\
        \notag &\hspace{5mm} = (L-1)  \frac{a \log(n)}{L}  \sum_{j=0}^{k-\ell-1}  \frac{1}{L^{j}} 
        = (L-1)  \frac{a \log(n)}{L} \frac{L-L^{1+\ell-k}}{L-1} = a \log(n) \left( 1- L^{\ell-k} \right)\\
        \label{eq:long calculation} &\hspace{5mm} \geq a \log(n) \left( 1- m L^{\ell-k} \right) = a \log(n) \left( 1- \gamma \right).
    \end{align}

    \noindent \textbf{The general case.}
    If $\emptyset \subsetneq A \subsetneq \Lambda_{n,n+k}$ is not of the form as described above, let $\ell \in \{1,\ldots,k-1\}$ and $u\in A$ be such that $\Lambda_{\ell}(u)\cap A \neq \emptyset$ and $\Lambda_{\ell}(u)\cap A^c \neq \emptyset$, where $\ell$ is the smallest such value. Let $B=A^c$. We divide the sets $A$ and $B$ into their {\sl inner} and {\sl outer} parts by
    \begin{align*}
        & A_i \coloneqq A \cap \Lambda_{\ell}(u), \qquad
        A_o \coloneqq A \cap \Lambda_{\ell}(u)^c = A \setminus A_i, \text{ and } \\
        & B_i \coloneqq B \cap \Lambda_{\ell}(u), \qquad
        B_o \coloneqq B \cap \Lambda_{\ell}(u)^c = B \setminus B_i,
    \end{align*}
    respectively. By the definition of $\ell$, this implies that we can write
    \begin{equation*}
        A_i = \bigcup_{j=1}^{m} \Lambda_{\ell-1}(v_j) \text{ and }
        B_i = \bigcup_{j=m+1}^{L} \Lambda_{\ell-1}(v_j)
    \end{equation*}
    for some $v_1,\ldots,v_L \in \Lambda_{n,n+k}$ with $v_1,\ldots,v_L \in \Lambda_{\ell}(u)$
    for some $m\in \{1,\ldots,L-1\}$. 
    Assume that $|B_i|\geq |A_i|$. Then one has $J(A_i,A_o)\geq J(B_i, A_o)$, as for a a set $S\subset \Lambda_{\ell}(u)$ the quantity $J(S,A_o)$ is proportional to the size of the input $|S|$, since $A_o \cap \Lambda_{\ell}(u) = \emptyset$. In particular, we get that $J \left( A_o, B_i  \right) = J \left( B_i , A_o \right) \geq J \left( A_i , A_o \right)$, which implies that
    \begin{align*}
        J(A,A^c) & = J \left( A_i \cup A_o, B_i \cup B_o \right)
        \\
        &
        =
        J \left( A_i, B_i  \right)
        +
        J \left( A_i , B_o \right)
        +
        J \left( A_o, B_i  \right)
        +
        J \left( A_o,  B_o \right)\\
        &
        \geq 
        J \left( A_i, B_i  \right)
        +
        J \left( A_i , B_o \right)
        +
        J \left( A_o, B_i  \right)\\
        &
        \geq 
        J \left( A_i, B_i  \right)
        +
        J \left( A_i , B_o \right)
        +
        J \left( A_i, A_o  \right) \\
        &
        = J \left( A_i, B_i \cup B_o \cup A_o\right) = J(A_i, A_i^c) \\
        & \overset{\eqref{eq:long calculation}}{\geq} a \log(n) \left( 1- \frac{|A_i|}{L^k} \right)
        \geq a \log(n) \left( 1- \frac{|A|}{L^k} \right) .
    \end{align*}
    For the second-to-last inequality, we used a comparison with \eqref{eq:long calculation}. This is possible as the set $A_i$ is of the form $A_i=\bigcup_{j=1}^{m} \Lambda_{\ell-1}(v_j) \subset \Lambda_{\ell}(u)$ as described in the special case above, and $B_i \cup B_o \cup A_o = \Lambda_{n,n+k}\setminus A_i$. So inequality \eqref{eq:long calculation} holds for the set $A_i$. This shows inequality \eqref{eq:prob bound 1-eps} for the case where $|B_i|\geq |A_i|$. The case where $|B_i|\leq |A_i|$ follows by symmetry in $A$ and $A^c$.
\end{proof}

The final lemma of this section estimates the probability that a specific large set of vertices in $\Lambda_{n,n+k}$ fails to connect to the annulus  $\Lambda_{n,n+2k}\setminus \Lambda_{n,n+k}$ by a single edge.

\begin{lemma}\label{lem:semi-good2}
	Let $a,\delta>0, n \in \N$, and let $J : \H_L \to \R_{\geq 0}$ be a kernel such that $J\left(L^j\right) \geq \frac{(a+\delta)L \log (j)}{L^{2j}}$ for every $j\geq n$. Then, for all $k\in \N$ for which $(a+\delta)( 1 - L^{-k} ) \geq a$ and for all subsets $A\subset \Lambda_{n,n+k}$ of size $|A|=\gamma L^k$ one has
    \begin{align}\label{eq:semi good 2}
	&\p_{J} \left( A \nsim \Lambda_{n,n+2k} \setminus \Lambda_{n,n+k} \right) \leq n^{-\gamma a}.
	\end{align}
\end{lemma}

\begin{proof}[Proof of \Cref{lem:semi-good2}]
    We can compute that if $k$ is sufficiently large that $(a+\delta)( 1 - L^{-k} ) \geq a$ then the total weight of edges between $A$ and $\Lambda_{n,n+2k}\setminus\Lambda_{n,n+k}$ satisfies
    \begin{align*}
        &J\!\textbf{}\left( A, \Lambda_{n,n+2k}\setminus\Lambda_{n,n+k}  \right)
        =
        \sum_{j=k}^{2k-1} J\left( A, \Lambda_{n,n+j+1}\setminus\Lambda_{n,n+j}  \right)
        \\
        &\hspace{25mm}
        \geq
        \sum_{j=k}^{2k-1} \gamma L^k L^n (L-1) L^j L^n \frac{(a+\delta) L\log(n)}{L^{2(n+j+1)}}
        =
        \sum_{j=k}^{2k-1} \gamma (L-1)  \frac{(a+\delta) \log(n)}{L^{j-k+1}}\\
        &\hspace{25mm}
        =
        \frac{\gamma (L-1) (a+\delta) \log(n)}{L}
        \sum_{j=0}^{k-1}   \frac{1}{L^{j}}
        = 
        \frac{\gamma (L-1) (a+\delta) \log(n)}{L}
        \frac{L-L^{1-k}}{L-1} \\
        &\hspace{25mm}
        =
        \gamma (a+\delta) \log(n)
        (1-L^{-k}) \geq 
        \gamma a \log(n).
    \end{align*}
    This implies \eqref{eq:semi good 2} in light of \eqref{eq:Xi_large_deviation}.
\end{proof}

\subsection{The renormalization step}
\label{sec:renormalization_step}

 For $a>1$  and a kernel $J$ satisfying $J(L^j) \geq aL^{1-2j} \log(j)$ for all sufficiently large $j$, the results from Section \ref{sec:finiteboxes} imply that there exists $\eps =\eps(J)> 0$ and $k_0=k_0(J)<\infty$ such that for all $k\geq k_0$  there exists $n_0=n_0(k) < \infty$ such that for the following estimates hold for all $n\geq n_0$ and all $A\subset \Lambda_{n,n+k}$:
\begin{align}
     \label{eq:block1 removed 1/2} \p_{J} \left( \Lambda_{n,n+k}\setminus\{\Lambda_n(0)\} \text{ is not connected off $\Lambda_n(0)$ in $\omega_{n,n+k}$} \right)
		&\leq n^{-\frac{1}{2}-\eps}\\
     \label{eq:block2 not connected} \p_{J} \left( \Lambda_{n,n+k} \text{ is not connected in $\omega_{n,n+k}$} \right) &\leq n^{-\frac{1}{2}-\eps}, \\
    \label{eq:block3 small cluster}
     \p_{J} \left( \Lambda_{n,n+k} \text{ has a cluster of density at most $\gamma$ in $\omega_{n,n+k}$} \right) &\leq  n^{-(1-\gamma)-2\eps},\\ \text{and} \qquad
     \label{eq:block4 to nextblock} \p_{J} \left( A \nsim \Lambda_{n,n+2k} \setminus \Lambda_{n,n+k} \right) &\leq  n^{- L^{-k}|A|}.
\end{align}
We now use these estimates to complete the proof of \Cref{theo:existence mixed}. We begin by making some relevant choices of parameters.

\medskip
\noindent
\textbf{Choice of parameters.}
Fix $a>1$ and let $a^\star = 1 + \frac{a-1}{2}$ so that $1<a^\star < a$. Define the isometry-invariant kernel $\bar{J}$ by
	\begin{equation*}
	\bar{J}(e) = a^\star L |e|^{-2}\log\log(|e|).
	\end{equation*}
	Let $\eps = \eps(\bar{J}) \in (0,1)$ and $k_0=k_0(\bar{J})$ be as above, so that for each $k\geq k_0$ there exists $n_0=n_0(k)<\infty$ such that the estimates \eqref{eq:block1 removed 1/2} through \eqref{eq:block4 to nextblock} hold for this kernel whenever $n\geq n_0$. We fix  $N \in \N$, $\zeta \in (0,0.01)$, $\theta \in (0.99,1)$, and $r \geq k_0$ such that
	\begin{align*}
		N \eps > 10, \ (1-\zeta)^4 \theta^2 a > a^\star, \ 4NL^{-\frac{r}{2}} < 0.5 , \ 3 N L^{-r} < \zeta .
	\end{align*}
    Then there exists $n_0 = n_0(r)$ such that the estimates \eqref{eq:block1 removed 1/2} through \eqref{eq:block4 to nextblock} hold for all $n \geq n_0$.
    
	For each $g \geq 2$, let $\bar{r}_g$ be the unique integer multiple of $r$ for which $\log(g)^{10} g \leq L^{\bar{r}_g} < L^{r} \log(g)^{10} g$ and define $M_g \coloneqq r+\sum_{m=2}^{g} \bar{r}_m$. (The power 10 in the definition of $\bar{r}_g$ is arbitrary, and every large enough integer would work here.) Given an integer $g_0$ to be chosen, we also define a sequence  $\left(\theta_g\right)_{g\geq g_0}$ recursively by
	\begin{align*}
    & \theta_{g_0}=1 , \\
	& \theta_{g+1} = \theta_g - 2N \left( \log(g)^3 L^{-\bar{r}_{g+1} } + \log(g)^5 g^{-1-\eps} \right) \text{ for } g \geq g_0.
	\end{align*}
	We take $g_0$ large enough that $\bar{r}_{g+1} \leq \log(g)^2$ for all $g\geq g_0$ and such that $\theta_g \geq \theta$ for all $g \in \{g_0,g_0+1,\ldots\}$, which is possible since
    \begin{equation*}
        \inf_{g} \theta_{g} \geq 1 - \sum_{g=g_0}^{\infty} 2N \left( \log(g)^3 L^{-\bar{r}_{g+1} } + \log(g)^5 g^{-1-\eps} \right)
    \end{equation*}
    and the definition of the sequence $\bar{r}_g$ makes the series
    \begin{equation*}
        \sum_{g=2}^{\infty} 2N \left( \log(g)^3 L^{-\bar{r}_{g+1} } + \log(g)^5 g^{-1-\eps} \right)
    \end{equation*}
    convergent. For an $n$-block $\varpi \in \Lambda_{n,\infty}$, with $n \geq r$, we define the set of $(n-r)$-blocks contained in $\varpi$ by $\chi_r(\varpi) = \left\{ \varphi \in \Lambda_{n-r,\infty} : \varphi \subseteq \varpi \right\}$.

\medskip
\noindent
\textbf{Contracting blocks and largest clusters.}
    For a subset $U\subseteq \chi_r(\varpi)$, we define the two graphs $G^{\mathrm{bl}}\left( U \right) = \left( U , E^{\mathrm{bl}} \right)$ and $G^{\mathrm{mx}}\left( U \right) =  \left( U , E^{\mathrm{mx}} \right)$ as the undirected graphs with vertex sets $U$ and edge sets
    \begin{align*}
        E^{\mathrm{bl}} & = \left\{ \{ \varphi , \psi \} \subset U : \varphi \neq \psi, \varphi \sim \psi \right\}, \textnormal{ and} \\
        E^{\mathrm{mx}} & = \left\{ \{ \varphi , \psi \} \subset U : \varphi \neq \psi, K_{\max}(\varphi) \sim K_{\max}(\psi) \right\},
    \end{align*}
    respectively. So in particular the graph $G^{\mathrm{bl}}\left( U \right)$ can be constructed from the edges between different blocks $\varphi, \psi \in U$, whereas the graph $G^{\mathrm{mx}}\left( U \right)$ can only be constructed after observing the edges between different blocks $\varphi, \psi \in U$, as well as the largest open components $K_{\max}(\varphi)$ and $K_{\max}(\psi)$ inside the blocks. We will only consider the case where $U=\chi_r(\varpi)$ or where $U=\chi_r(\varpi)\setminus \{\varphi \}$ for some $\varphi \in \chi_r(\varpi)$.\\

    For $\varphi \in \chi_r(\varpi)$ and $A \subseteq \chi_{r}(\varphi)$ we write $A \sim \chi_{r}(\varpi)\setminus \{\varphi\}$ if there exist $x,y \in \H_L$ with $x \in \sigma$ for some $\sigma \in A$ and $y \in \varphi^\prime$ for some $\varphi^\prime \in \chi_r(\varpi) \setminus \{\varphi\}$ such that $x \sim y$. Otherwise, we write $A \nsim \chi_{r}(\varpi)\setminus \{\varphi\}$.
    For $\varpi \in \Lambda_{n,\infty}$, the graph $G^{\mathrm{bl}}\left( \chi_r(\varpi) \right)$ has exactly the same distribution as $\Lambda_{n-r,n}$ and for $\varphi \in \chi_r(\varpi)$, the graph $G^{\mathrm{bl}}\left( \chi_r(\varpi) \setminus \{\varphi\}\right)$ has exactly the same distribution as the graph $\Lambda_{n-r,n}\setminus \left\{ \Lambda_{n-r}(0) \right\}$. So by the choice of parameters, we have for all $n-r\geq n_0$, all $\varpi \in \Lambda_{n,\infty}$, and all $\varphi \in \chi_r(\varpi)$ that
    \begin{align}
     \label{eq:block1 removed 1/2 v2} \p_{\bar{J}} \left( G^{\mathrm{bl}}\left( \chi_r(\varpi) \setminus \{\varphi\} \right) \text{ is not connected} \right)
		&\leq (n-r)^{-\frac{1}{2}-\eps}\\
     \label{eq:block2 not connected v2} \p_{\bar{J}} \left( G^{\mathrm{bl}}\left( \chi_r(\varpi) \right) \text{ is not connected} \right) &\leq (n-r)^{-\frac{1}{2}-\eps}, \\
    \label{eq:block3 small cluster v2}
     \p_{\bar{J}} \left( G^{\mathrm{bl}}\left( \chi_r(\varpi) \right) \text{ has a cluster of density at most $\gamma$} \right) &\leq  (n-r)^{-(1-\gamma)-2\eps},\\ \text{and, for all $A \subset \chi_r(\varphi)$,} \qquad
     \label{eq:block4 to nextblock v2} \p_{\bar{J}} \left( A \nsim \chi_{r}(\varpi)\setminus \{\varphi\} \right) &\leq  (n-r)^{- L^{-r}|A|}.
\end{align}

\medskip
\noindent
\textbf{Good, bad, and mediocre blocks.}
We now iteratively contract blocks of the form $\Lambda_{n}(u)$ into vertices as follows. For each such block, let $K_{\max}\left( \Lambda_{n}(u) \right)$ be the largest open cluster inside $\Lambda_{n}(u)$. If this is not uniquely defined, pick one of the largest clusters inside the block in a manner that depends only on the configuration inside the block. We say that two blocks $\varpi \in \Lambda_{n,\infty}$ and $\varphi \in \Lambda_{m,\infty}$ with $\varpi \cap \varphi = \emptyset$ are connected if there exists an edge between the largest clusters of these blocks.
We say that the block $\Lambda_{M_{g_0}}(u)$ is {\sl good} if $\left|K_{\max} \left( \Lambda_{M_{g_0}}(u) \right) \right|= \left| \Lambda_{M_{g_0}}(u)  \right|$ and {\sl bad} otherwise.

Let $g \geq g_0$ and consider the percolation configuration restricted to $\Lambda_{M_g+\bar{r}_{g+1}}(w)$. 
We now describe how to merge blocks of the form $\Lambda_{M_g}(u) \in \Lambda_{M_{g},\infty}$ within the set $\Lambda_{M_{g}+\bar{r}_{g+1}}(w) \in \Lambda_{M_{g}+\bar{r}_{g+1},\infty}$.
We first explore the configuration inside each of the $L^{\bar{r}_{g+1}}$ many blocks of the form $\Lambda_{M_{g}}(u) \subseteq \Lambda_{M_{g}+\bar{r}_{g+1}}(w)$. After this, we reveal the percolation configuration inside blocks of the form $\Lambda_{M_{g}+r}(u)$, and connect blocks of the form $\Lambda_{M_{g}}(v)$ inside $\Lambda_{M_g+r}(u)$. Next, we consider blocks of the form $\Lambda_{M_{g}+2r}(u)$ as base graph and connect blocks of the form $\Lambda_{M_g+r}(v)$ inside $\Lambda_{M_g+2r}(u)$, etc.

We already defined what it means for blocks of the form $\Lambda_{M_{g_0}}(u)$ to be good/bad. For blocks at level $M_{g_0}$, there are only these two options. For blocks at level $M_{g_0}+kr$, with $k \in \N$, there will be three options, which are  \emph{good}, \emph{mediocre}, and \emph{bad}. Furthermore, for $g \geq g_0$, $kr \in \{ r, \ldots, \bar{r}_{g+1}\}$ and $j \in \{M_{g}, M_{g}+r ,\ldots,kr\}$, a good or mediocre block at level $M_{g}+kr$ can have a defect at level $j$.

Let $\varpi \in \Lambda_{M_{g}+kr,\infty}$ be a good or mediocre block, where $g \geq g_0$ and $kr \in \left\{  r , \ldots , \bar{r}_{g+1} \right\}$. For $j \in \{ M_{g}, M_{g}+ r , \ldots, M_{g} + kr \}$, we define the number of defects of $\varpi$ at level $j$ by
    \begin{equation*}
        \# \operatorname{Def}_j (\varpi) \coloneqq \left| \left\{ \varphi \in \Lambda_{j, \infty} : \varphi \subseteq \varpi,  \varphi  \text{ is bad} \right\} \right|.
    \end{equation*}
    
Given that all of this has been defined for blocks of the form $\Lambda_{M_g}(u)$, we define it for blocks of the form $\Lambda_{M_{g+1}}(v)$ as follows.

\begin{itemize}
    \item[(1.)] For $j = M_{g}+r, M_{g}+2r, \ldots, M_{g}+\bar{r}_{g+1} = M_{g+1}$:
    \begin{itemize}
        \item[] For all $\varpi \in \Lambda_{j,\infty}$ with $\varpi \subseteq \Lambda_{M_{g+1}}(v)$:
        \item[(1.0)] If there exists $m \in \{M_{g},M_{g}+r,\ldots,j-r\}$ such that $\# \operatorname{Def}_{m}(\varpi) > N$, then we say that $\varpi$ is bad.
    \end{itemize}
        The next step, (1.1), is broken into four cases (1.1A) to (1.1D).     These cases are mutually exclusive and exhaustive: For each block $\varpi$, exactly one of the cases described in (1.1A) through (1.1D) holds.
    \begin{itemize}
        \item[(1.1A)] If all elements $\varphi \in \chi_r (\varpi)$ are good.
        \begin{itemize}
            \item[$\bullet$]  If $G^{\mathrm{mx}} \left(\chi_r(\varpi)\right)$ is connected, then we say that $\varpi$ is good. 
            \item[$\bullet$] If $G^{\mathrm{mx}} \left(\chi_r(\varpi)\right)$ is not connected, we say that $\varpi$ is mediocre. More specifically, if the smallest connected component of $G^{\mathrm{mx}} \left(\chi_r(\varpi)\right)$ is of size $\gamma L^r$, then we say that $\varpi$ is $\gamma$-mediocre.
        \end{itemize}
        \item[(1.1B)]  If at least two elements $\varphi, \varphi^\prime \in \chi_r (\varpi)$ are not good, we say that $\varpi$ is bad.
        \item[(1.1C)]  If exactly one element $\varphi \in \chi_r (\varpi)$ is bad and all other elements are good. 
        \begin{itemize}
            \item[$\bullet$] If $G^{\mathrm{mx}} \left(\chi_r(\varpi) \setminus \{\varphi \}\right)$ is connected then we say that $\varpi$ is good. 
            \item[$\bullet$] 
            Otherwise, we say that $\varpi$ is bad.
        \end{itemize}
        \item[(1.1D)] If exactly one element $\varphi \in \chi_r (\varpi)$ is mediocre and all other elements are good. 
        \begin{itemize}
            \item[$\bullet$] If $G^{\mathrm{mx}} \left(\chi_r(\varpi)\setminus \{\varphi \}\right)$ is connected and for all connected components $A$ of $G^{\mathrm{mx}} \left(\chi_r(\varphi)\right)$ there exists $\sigma \in A, \psi \in \chi_r(\varpi)\setminus \{\varphi \}$ such that $K_{\max}(\sigma) \sim K_{\max}(\psi)$, then we say that $\varpi$ is good.
            \item[$\bullet$]
            Otherwise, we say that $\varpi$ is bad.
        \end{itemize}
    \end{itemize}
    \item[(2.)] If there exists $j \in \left\{M_{g-1}, M_{g-1} + r,\ldots, M_{g} - r\right\}$ such that among the $L^{\bar{r}_{g+1}}$ many blocks $\Lambda_{M_{g}}(u) \subseteq \Lambda_{M_{g+1}}(v)$, there are at least  $\log(g)^3 \lceil L^{\bar{r}_{g+1}+M_{g}-j} g^{-1-\eps} \rceil$ good or mediocre blocks that have a defect at level $j$, then we say that $\Lambda_{M_{g+1}}(v)$ is bad.
    \item[(3.)] If $\Lambda_{M_{g+1}}(v)$ is good or mediocre and there exists $j \in \{ M_{g}, M_{g}+r, \ldots, M_{g+1} - r \}$ and $\varpi \in \Lambda_{j, \infty}$ with $\varpi \subseteq \Lambda_{M_{g+1}}(v)$ such that $\varpi$ is bad, and $j$ is the largest such value, then we say that $\Lambda_{M_{g+1}}(v)$ has a defect at level $j$ and write $\operatorname{Def}\left(\Lambda_{M_{g+1}}(v)\right) = j$. We write $\operatorname{Def}\left(\Lambda_{M_{g+1}}(v)\right) = -\infty$ in case there does not exist such $j \in \{ M_{g}, M_{g}+r, \ldots, M_{g+1} - r \}$, and say that $\Lambda_{M_{g+1}}(v)$ has no defect.
\end{itemize}

The basic idea behind these definitions is that we want to allow ``minor defects'' to be ``fixed'' the scale after they occur, while making sure that the resulting ``good'' blocks maintain a high maximum cluster density. The steps (1.0) and (2.) make sure that there are not too many minor defects whose contributions add up to a major defect. These steps guarantee that good blocks always have open clusters of density approximately $\theta$, as Lemma \ref{lem:density evolu} below shows.

It might happen that a block $\varpi$ is declared both bad and good, or both bad and mediocre, by different steps of the above algorithm. If this happens, we say that $\varpi$ is bad.
After running the above algorithm for all feasible $j$, we note that the final block $\Lambda_{M_{g+1}}(v)$ will be either good, mediocre, or bad. Further, the block $\Lambda_{M_{g+1}}(v)$ can have a defect at level $j$, for every $j \in \{M_{g}, M_{g}+r , \ldots, M_{g+1} - r \}$.
Also note that if the condition in (1.0) is triggered for some $\varpi \in \Lambda_{M_{g}+j,\infty}$, then also all blocks $\phi \in \Lambda_{M_g+ j + kr,\infty}$ with $\phi \supseteq \varpi$ will be bad for all $k \in \N$ for which $M_{g}+ j + kr \leq M_{g+1}$. Before going to the next lemma, we introduce the following notation.

\begin{definition}
    We say that a set $U \subset \Lambda_{M_g+jr,\infty}$ is non-bad if all its elements are either good or mediocre. For a good or mediocre block $\varphi \in \Lambda_{M_g+jr,\infty}$, we define
    \begin{equation*}
        S(\varphi) = \begin{cases}
            |K_{\max}(\varphi)| & \text{ if } \varphi \text{ is good}\\
            \sum_{\varphi^\prime \in \chi_r(\varphi)} |K_{\max}(\varphi^\prime)| & \text{ if } \varphi \text{ is mediocre}
        \end{cases}.
    \end{equation*}
\end{definition}

Working with the quantity $S(\varphi)$ allows a unified notation to treat both good and mediocre blocks $\varphi$. The key property of $S(\cdot)$ is that it behaves consistently with the notions of ``good” and ``mediocre”. 
More specifically, if $\varpi$ has been declared good or mediocre in cases (1.1A), (1.1C), or (1.1D) of step (1.1) in the algorithm above, then
\begin{equation}\label{eq:S not bad}
    S(\varpi) \geq \sum_{\varphi \in \chi_r(\varpi) : \varphi \text{ not bad}} S(\varphi).
\end{equation}

Indeed, if $\varpi$ has been declared good in step (1.1A) of the algorithm above, then all $\varphi \in \chi_r(\varpi)$ are good and the largest clusters $\left( K_{\max}(\varphi) \right)_{\varphi \in \chi_r(\varpi)}$ merge into one large cluster. Thus
\begin{equation*}
    S(\varpi) = |K_{\max}(\varpi)| \geq \sum_{\varphi \in \chi_r(\varpi)} |K_{\max}(\varphi)| = \sum_{\varphi \in \chi_r(\varpi)} S(\varphi).
\end{equation*}
Further, if $\varpi$ has been declared mediocre in step (1.1A) of the algorithm above, then all $\varphi \in \chi_r(\varpi)$ are good and thus
\begin{equation*}
    S(\varpi) =  \sum_{\varphi \in \chi_r(\varpi)} |K_{\max}(\varphi)| = \sum_{\varphi \in \chi_r(\varpi)} S(\varphi).
\end{equation*}
If $\varpi$ has been declared good in step (1.1C) of the algorithm above, then all but one $\varphi \in \chi_r(\varpi)$ are good, and exactly one $\varphi \in \chi_r(\varpi)$ is bad. Further, the largest clusters $\left( K_{\max}(\varphi) \right)_{\varphi \in \chi_r(\varpi), \varphi \text{ good}}$ merge into one large cluster so that
\begin{equation*}
    S(\varpi) = |K_{\max}(\varpi)| \geq \sum_{\varphi \in \chi_r(\varpi) : \varphi \text{ good}} |K_{\max}(\varphi)| = \sum_{\varphi \in \chi_r(\varpi) : \varphi \text{ not bad}} S(\varphi).
\end{equation*}
Finally, if $\varpi$ has been declared good in step (1.1D) of the algorithm above, with $\varphi^\prime \in \chi_r(\varpi)$ being the unique mediocre element of $\chi_r(\varpi)$, then all the largest clusters $\left( K_{\max}(\varphi) \right)_{\varphi \in \chi_r(\varpi) \setminus \{ \varphi^\prime \}}$ merge into one large cluster and all the connected components of $G^{\mathrm{mx}}\left(\chi_r(\varphi^\prime)\right)$ connect to this large cluster. Thus
\begin{equation*}
    S(\varpi) = |K_{\max}(\varpi)| \geq \sum_{\varphi \in \chi_r(\varpi) \setminus \{\varphi^\prime \} } |K_{\max}(\varphi)| + \sum_{\sigma \in \chi_r(\varphi^\prime)} |K_{\max}(\sigma)| = \sum_{\varphi \in \chi_r(\varpi)} S(\varphi).
\end{equation*}

We proceed with Lemma \ref{lem:density evolu}, which gives lower bounds on $S(\varpi)$ for good or mediocre blocks $\varpi \in \Lambda_{M_g+jr,\infty}$.

\begin{lemma}\label{lem:density evolu}

    \emph{\textbf{Part (A):}}
    Let $\varpi \in \Lambda_{M_{g}+jr,\infty}$ for some $g \geq g_0$ and $jr \in \left\{ r,\ldots, \bar{r}_{g+1} \right\}$. 
    If $\varpi$ is good or mediocre, then there exists a non-bad set $U \subset \Lambda_{M_{g},\infty}$ with
    \begin{equation*}
        |U| \geq L^{jr} - \sum_{k=0}^{j-1} \# \operatorname{Def}_{M_g + kr}(\varpi) L^{kr}
    \end{equation*}
    such that $\sigma \subset \varpi$ for all $\sigma \in U$ and
    \begin{equation}\label{eq:S 2 inequalities Part A}
        S(\varpi) \geq \sum_{\sigma \in U} S(\sigma) \geq (1-\zeta)^2 \theta_g  L^{M_{g}+jr}.
    \end{equation}
    \emph{\textbf{Part (B):}}
    If $\varpi \in \Lambda_{M_g,\infty}$ is good, with a defect at level $n \in \{M_{g-1}, M_{g-1}+r,\ldots,M_{g} - r \}$, then
    \begin{equation}\label{part B ineq1}
        |K_{\max}(\varpi)| \geq \theta_g L^{M_{g}} - 2N L^n \geq (1-\zeta) \theta_g L^{M_g}  .
    \end{equation}
    If $\varpi \in \Lambda_{M_g,\infty}$ is mediocre, with a defect at level $n \in \{M_{g-1}, M_{g-1}+r,\ldots,M_{g} - 2r \}$, then
    \begin{equation}\label{part B ineq2}
        \sum_{\varphi \in \chi_{r}(\varpi)}  |K_{\max}(\varphi)| \geq \theta_g L^{M_{g}} - 2N L^n \geq (1-\zeta) \theta_g L^{M_g} .
    \end{equation}
    If $\varpi \in \Lambda_{M_g,\infty}$ is good or mediocre, without a defect, then 
    \begin{equation*}
        S(\varpi) \geq \theta_g L^{M_{g}}  .
    \end{equation*}
\end{lemma}

\begin{proof}
    We prove the statement via induction on $\ell \geq M_{g_0}$ with $r \mid \ell$. The statement is clear when $\ell=M_{g_0}$. For the induction step, assume that $\ell = M_{g}+jr$ with $jr \in \{r,2r,\ldots,\bar{r}_{g+1}\}$ and that the induction hypothesis holds for all $\ell^\prime \in \{M_{g_0},M_{g_0}+r,\ldots,\ell\}$. Let $\varpi \in \Lambda_{M_g+jr+r,\infty}$. 

    \medskip
    \noindent
    \textbf{Part (A):}
    For the inductive step from $M_{g} + jr$ to $M_{g}+(j+1)r$, we distinguish the cases $jr = \bar{r}_{g+1}$ and $jr \in \{  r, 2r, \ldots, \bar{r}_{g+1} - r \}$. 
    
    \medskip
    \noindent
    \textbf{Case 1:} $\left( jr= \bar{r}_{g+1} \right)$. First remember that $M_{g}+\bar{r}_{g+1}=M_{g+1}$. So by the induction assumption, all elements of Part (B) above hold for $\varphi \in \chi_r(\varpi) \subset \Lambda_{M_{g+1},\infty}$.
    If $\varpi \in \Lambda_{M_{g+1} + r,\infty}$ is good or mediocre, define the set $U \subseteq \Lambda_{M_{g+1},\infty}$ by \[U \coloneqq \left\{ \varphi \in \chi_r(\varpi) : \varphi \text{ is good or mediocre} \right\}.\] If there exists $\varphi \in \chi_r(\varpi)$ that is bad, then $|U|=L^r-1$. Otherwise, $|U|=L^r$. In any of the two cases, one easily checks that
    \begin{equation*}
        |U| = L^{jr} - \sum_{k=0}^{0} \# \operatorname{Def}_{M_{g+1} + kr}(\varpi) L^{kr} .
    \end{equation*}
    This implies the first inequality in \eqref{eq:S 2 inequalities Part A} in the light of \eqref{eq:S not bad}.
    By the induction assumption one has $S(\varphi)=|K_{\max}(\varphi)| \geq (1-\zeta) \theta_{g+1} L^{M_{g+1}}$ for all good or mediocre $\varphi \in U$. If $\varpi$ is good or mediocre, then $|U|\geq L^{r}-1$. Thus, we get that
    \begin{multline*}
        S(\varpi) \geq \sum_{\varphi \in U} S(\varphi) \geq \sum_{\varphi \in U} (1-\zeta) \theta_{g+1} L^{M_{g+1}}\\ \geq \left( L^r- 1 \right) (1-\zeta) \theta_{g+1} L^{M_{g+1}} 
        \geq
        (1-\zeta)^2 \theta_{g+1} L^{M_{g+1} + r} ,
    \end{multline*}
    where we used that $L^{-r}< \zeta$. This shows the second inequality in inequality \eqref{eq:S 2 inequalities Part A}. 

    \medskip
    \noindent
    \textbf{Case 2(a):} $\left( jr \in \{0,r,\ldots,\bar{r}_{g+1} -  r\} , \varpi \text{ good} \right)$. 
    We start with the case where all $\varphi \in \chi_{r}(\varpi)$ are good or where exactly one is mediocre. By the induction assumption for each $\varphi \in \chi_{r}(\varpi)$ there exists a non-bad set $U_{\varphi} \subseteq \Lambda_{M_g,\infty}$ such that
    \begin{equation*}
        \bigcup_{\sigma \in U_\varphi } \sigma \subseteq \varphi, \ S(\varphi) \geq \sum_{\sigma \in U_\varphi} S(\sigma) , \text{ and } \ |U_\varphi|  \geq L^{jr} - \sum_{k=0}^{j-1} \# \operatorname{Def}_{M_g + kr}(\varphi) L^{kr} .
    \end{equation*}
    Define the set $U \coloneqq \bigcup_{\varphi \in \chi_r(\varpi)} U_{\varphi}$. Since $\varpi$ is good, and all $\varphi \in \chi_r(\varpi)$ are not bad, one gets
    \begin{align*}
        & \#\operatorname{Def}_{M_g + jr}(\varpi) = \sum_{\varphi \in \chi_r(\varpi)} \# \operatorname{Def}_{M_g + jr}(\varphi) = 0 , \text{ and }
    \\
        &\# \operatorname{Def}_{M_g + kr}(\varpi) = \sum_{\varphi \in \chi_r(\varpi)} \# \operatorname{Def}_{M_g + kr}(\varphi) \leq N \ \text{ for all } k \in \{0,\ldots,j\} .
    \end{align*}
    Thus, for the set $U$, we have that
    \begin{align*}
        & |U| = \sum_{\varphi \in \chi_{r}(\varpi)} |U_\varphi|  = \sum_{\varphi \in \chi_r(\varpi)} \left( L^{jr} - \sum_{k=0}^{j-1} \# \operatorname{Def}_{M_g + kr}(\varphi) L^{kr} \right)  \\
        &
        \hspace{5cm}
        = \sum_{\varphi \in \chi_r(\varpi)} \left( L^{jr} - \sum_{k=0}^{j} \# \operatorname{Def}_{M_g + kr}(\varphi) L^{kr} \right)
        \\
        &
        \hspace{5cm}
        = L^r L^{jr} - \sum_{\varphi \in \chi_r(\varpi)} \sum_{k=0}^{j} \# \operatorname{Def}_{M_g + kr}(\varphi) L^{kr}
        \\
        &
        \hspace{5cm}
        = L^{(j+1)r} - \sum_{k=0}^{j} \# \operatorname{Def}_{M_g + kr}(\varpi) L^{kr}
        \\
        &
        \hspace{5cm}
        \geq L^{(j+1)r} - \sum_{k=0}^{j} N L^{kr} = L^{(j+1)r} \left( 1 - \sum_{k=0}^{j} N L^{kr-(j+1)r} \right) .
    \end{align*}
    Further, if $\varpi$ is good and all $\varphi \in \chi_r(\varpi)$ are not bad, then
    \begin{align*}
        & S(\varpi) = |K_{\max}(\varpi)| \geq \sum_{\varphi \in \chi_r(\varpi)} S(\varphi) \geq \sum_{\varphi \in \chi_r(\varpi)} \sum_{ \sigma \in U_\varphi} S(\sigma)
        =
        \sum_{ \sigma \in U} S(\sigma) \geq |U| (1-\zeta) \theta_g L^{M_g}
        \\
        &
        \hspace{3cm}
        \geq
        (1-\zeta) \theta_g L^{M_{g}+(j+1)r} \left(1 - \sum_{k=0}^{j} N L^{kr-(j+1)r} \right) 
        \\
        &
        \hspace{3cm}
        \geq
        (1-\zeta) \theta_g L^{M_{g}+(j+1)r} \left(1 - 2N L^{-r} \right)
        \geq
        (1-\zeta)^2 \theta_g L^{M_{g}+(j+1)r},
    \end{align*}
    where we used the assumption $3 N L^{-r} < \zeta$ in the last inequality.
    
    We are left to show the induction step in the case where all but one $\varphi \in \chi_{r}(\varpi)$ are good and exactly one is bad. Say that $\tilde{\varphi} \in \chi_r(\varpi)$ is the unique element of $\chi_r(\varpi)$ that is bad. Then
    \begin{align*}
        & \# \operatorname{Def}_{M_g + kr}(\varpi) = \sum_{\varphi \in \chi_r(\varpi)} \# \operatorname{Def}_{M_g + kr}(\varphi) \geq  \sum_{\varphi \in \chi_r(\varpi) \setminus \{\tilde{\varphi} \}} \# \operatorname{Def}_{M_g + kr}(\varphi) \\
        & \text{ for all } k \in \{0,\ldots,j-1\}, \text{ and } 
        \# \operatorname{Def}_{M_g + jr}(\varpi) = 1.
    \end{align*}
    By the induction assumption, for each good $\varphi \in \chi_{r}(\varpi)$ there exists a non-bad set $U_{\varphi} \subseteq \Lambda_{M_g,\infty}$ such that
    \begin{equation*}
        \bigcup_{\sigma \in U_\varphi } \sigma \subseteq \varphi, \ S(\varphi) \geq \sum_{\sigma \in U_\varphi} S(\sigma) , \text{ and } \ |U_\varphi|  \geq L^{jr} - \sum_{k=0}^{j-1} \# \operatorname{Def}_{M_g + kr}(\varphi) L^{kr}  .
    \end{equation*}
    Now define the set $U \coloneqq \bigcup_{\varphi \in \chi_r(\varpi) \setminus \{\tilde{\varphi}\}}$. Since $\varpi$ is good, for each $k \in \{M_{g}, M_{g}+r,\ldots,M_{g}+(j-1)r\}$, there can be at most $N$ many vertices $\varphi \in \chi_r(\varpi)$ such that $\varphi$ has a defect at level $k$.
    Thus, we can bound the size of $U$ from below by
    \begin{align}
        & \notag |U| = \sum_{\varphi \in \chi_{r}(\varpi) \setminus \{\tilde{\varphi}\}} |U_\varphi| 
        \geq \sum_{\varphi \in \chi_{r}(\varpi) \setminus \{\tilde{\varphi} \} } \left( L^{jr} - \sum_{k=0}^{j-1} \# \operatorname{Def}_{M_g + kr}(\varphi) L^{kr} \right)
        \\
        & \notag
        \hspace{5cm}
        \geq \sum_{\varphi \in \chi_{r}(\varpi) \setminus \{\tilde{\varphi} \} } L^{jr}  - \sum_{\varphi \in \chi_{r}(\varpi) \setminus \{\tilde{\varphi} \} } \sum_{k=0}^{j-1}  \# \operatorname{Def}_{M_g + kr}(\varphi) L^{kr}
        \\
        & \notag
        \hspace{5cm}
        \geq
        L^{(j+1)r} - L^{jr}  - \sum_{\varphi \in \chi_{r}(\varpi)  } \sum_{k=0}^{j-1}  \# \operatorname{Def}_{M_g + kr}(\varphi) L^{kr}
        \\
        & \notag
        \hspace{5cm}
        =
        L^{(j+1)r} - \sum_{k=0}^{j}  \# \operatorname{Def}_{M_g + kr}(\varpi) L^{kr}
        \\
        & \label{eq:U set lower bound}
        \hspace{5cm}
        \geq
        L^{(j+1)r} - \sum_{k=0}^{j}  N L^{kr} .
    \end{align}
    Further, as $|K_{\max}(\sigma)| \geq (1-\zeta)\theta_g L^{M_{g}}$ for all good $\sigma \in U$ one also readily checks that
    \begin{align*}
        &S(\varpi) = |K_{\max}(\varpi)| \geq \sum_{\varphi \in \chi_r(\varpi) \setminus \{\tilde{\varphi}\} } S(\varphi) \geq \sum_{\varphi \in \chi_r(\varpi) \setminus \{\tilde{\varphi}\}} \sum_{\sigma \in U_\varphi} S(\sigma)
        =
        \sum_{\sigma \in U} S(\sigma)
        \\
        &
        \hspace{3cm}
        \overset{\eqref{eq:U set lower bound}}{\geq}
        \left( L^{(j+1)r} - \sum_{k=0}^{j}  N L^{kr} \right) (1-\zeta) \theta_g L^{M_{g}}
        \\
        &
        \hspace{3cm}
        \geq
        (1-\zeta) \theta_g L^{M_{g} + (j+1)r} \left( 1 - \sum_{k=0}^{j} N L^{kr-(j+1)r} \right) 
        \\
        &
        \hspace{3cm}
        \geq
        (1-\zeta) \theta_g L^{M_{g} + (j+1)r} \left( 1 - 2 N L^{-r} \right) 
        \geq
        (1-\zeta)^2 \theta_g L^{M_{g}+(j+1)r},
    \end{align*}
    where we again used the assumption $3 N L^{-r} < \zeta$ in the last inequality.
\medskip

\noindent
    \textbf{Case 2(b):} $\left( jr \in \{0,r,\ldots,\bar{r}_{g+1} -  r\} , \varpi \text{ mediocre} \right)$. 
    If $\varpi \in \Lambda_{M_g + (j+1)r}$ is mediocre, then all $\varphi \in \chi_{r}(\varpi)$ are good. Thus
    \begin{align*}
        & \#\operatorname{Def}_{M_g + jr}(\varpi) = \sum_{\varphi \in \chi_r(\varpi)} \# \operatorname{Def}_{M_g + jr}(\varphi) = 0 , \text{ and }
    \\
        &\# \operatorname{Def}_{M_g + kr}(\varpi) = \sum_{\varphi \in \chi_r(\varpi)} \# \operatorname{Def}_{M_g + kr}(\varphi) \leq N \ \text{ for all } k \in \{0,\ldots,j\} .
    \end{align*}
    By the induction assumption, for each $\varphi \in \chi_{r}(\varpi)$ there exists a non-bad set $U_{\varphi} \subseteq \Lambda_{M_g,\infty}$ such that
    \begin{equation*}
        \bigcup_{\sigma \in U_\varphi } \sigma \subseteq \varphi, S(\varphi) \geq \sum_{\sigma \in U_\varphi} S(\sigma) , \text{ and } |U_\varphi|  \geq L^{jr} - \sum_{k=0}^{j-1} \# \operatorname{Def}_{M_g + kr}(\varphi) L^{kr} .
    \end{equation*}
    Now define the set $U \coloneqq \bigcup_{\varphi \in \chi_r(\varpi)}$. The same proof as in Case 2(a) above shows that $|U| \geq L^{(j+1)r} - \sum_{k=0}^{j} \# \operatorname{Def}_{M_g + kr}(\varpi) L^{kr} \geq  L^{(j+1)r} - \sum_{k=0}^{j} N L^{kr}$.
    By the definition of $S(\varpi)$ one also readily checks that
    \begin{align*}
        & S(\varpi) = \sum_{\varphi \in \chi_r(\varpi)} S(\varphi) \geq \sum_{\varphi \in \chi_r(\varpi)} \sum_{\sigma \in U_\varphi} S(\sigma)
        =
        \sum_{\sigma \in U} S(\sigma)
        \\
        & \hspace{30mm}
        \geq \left( L^{(j+1) r} - \sum_{k=0}^{j} N L^{kr} \right) (1-\zeta) \theta_g L^{M_g}
        \\
        & \hspace{30mm}
        \geq \left( L^{(j+1) r} - 2N L^{jr} \right) (1-\zeta) \theta_g L^{M_g}
        \geq
        (1-\zeta)^2 \theta_g L^{M_{g}+(j+1)r}.
    \end{align*}

\medskip
\noindent
    \textbf{Part (B):} 
        By the results of Part (A) we know that if $\varpi \in \Lambda_{M_{g+1},\infty}$ is good or mediocre, then there exists a non-bad set $U \subset \Lambda_{M_{g},\infty}$ with
        \begin{equation*}
            |U| \geq L^{\bar{r}_{g+1}} - \sum_{kr \in \{0,r,\ldots,\bar{r}_{g+1}-r \} } \# \operatorname{Def}_{M_g + kr}(\varpi) L^{kr} \hspace{10mm} \text{ such that } \hspace{10mm}
            S(\varpi) \geq \sum_{\sigma \in U} S(\sigma)  .
        \end{equation*}
        What is left to show is that if $\varpi$ has a defect at level $n \in \{M_g, M_{g}+r,\ldots, M_{g+1}-r \} \cup \{-\infty\}$, then  $S(\varpi) \geq \theta_{g+1} L^{M_{g+1}} - 2N L^n \geq (1-\zeta)\theta_{g+1}L^{M_{g+1}}$. (Remember that we say that $\operatorname{Def}(\varpi)=-\infty$ if $\varpi$ does not have a defect.) We will only prove this when $\operatorname{Def}(\varpi)\neq -\infty$.  When $\operatorname{Def}(\varpi)=-\infty$, i.e., when $\varpi$ has no defect and thus $\#\operatorname{Def}_{M_g+kr}(\varpi) = 0$ for all $kr \in \{ 0 , r , \ldots , \bar{r}_{g+1} - r\}$, the exact same proof also works and shows that $S(\varpi) \geq \theta_{g+1} L^{M_{g+1}}$. \\
        
        First, we decompose the elements $\sigma \in U$ depending on at which level in $\{M_{g-1},M_{g-1}+r,\ldots,M_{g}-r\} \cup \{-\infty\}$ they have a defect. With this, we can lower bound $S(\varpi)$ by
        \begin{align}\label{eq:inequalities 1}
            \notag & S(\varpi) \geq \sum_{\sigma \in U} S(\sigma) 
            =
            \sum_{ir \in \{ M_{g-1} , M_{g-1} + r ,\ldots, M_{g}- r\}} \ \sum_{ \substack{ \sigma \in U : \\ \operatorname{Def}(\sigma)=ir } } S(\sigma) + \sum_{ \substack{ \sigma \in U : \\ \operatorname{Def}(\sigma) = - \infty }} S(\sigma)
            \\
            \notag &
            \hspace{25mm}
            \geq
            \sum_{ir \in \{ M_{g-1} , M_{g-1} + r ,\ldots, M_{g}- r\}} \sum_{ \substack{ \sigma \in U : \\ \operatorname{Def}(\sigma)=ir } } \left( \theta_g L^{M_g} - 2N L^{ir} \right) + \sum_{ \substack{ \sigma \in U : \\ \operatorname{Def}(\sigma) = - \infty }} \theta_g L^{M_g}
            \\
            & \hspace{25mm}
            = |U| \theta_g L^{M_g} - 2N
            \sum_{ir \in \{ M_{g-1} , M_{g-1} + r ,\ldots, M_{g}- r\}} \sum_{ \substack{ \sigma \in U : \\ \operatorname{Def}(\sigma)=ir } } L^{ir} .
        \end{align}
        If $\varpi$ is not bad, then, for all $ir \in \{ M_{g-1} , M_{g-1} + r ,\ldots, M_{g}- r\}$ one has
        \begin{equation*}
            \left| \left\{ \sigma \in U : \operatorname{Def}(\sigma) = ir \right\} \right| \leq \log(g)^3 \lceil L^{\bar{r}_{g+1} +M_g - ir} g^{-1-\eps} \rceil ,
        \end{equation*}
        by step (2.) of the algorithm above. Using this upper bound on $\left| \left\{ \sigma \in U : \operatorname{Def}(\sigma) = ir \right\} \right|$ and the trivial inequality $\lceil x \rceil \leq 1+x$, we get that
        \begin{align}\label{eq:defects sum}
            &\notag \sum_{ir \in \{ M_{g-1} , M_{g-1} + r ,\ldots, M_{g}- r\}} \sum_{ \substack{ \sigma \in U : \\ \operatorname{Def}(\sigma)=ir } } L^{ir}
            \\
            & \notag
            \hspace{8mm}
            \leq
            \sum_{ir \in \{ M_{g-1} , M_{g-1} + r ,\ldots, M_{g}- r\}} \hspace{-12mm} \log(g)^3 \lceil L^{\bar{r}_{g+1} +M_g - ir} g^{-1-\eps} \rceil  L^{ ir } 
            \\
            & \notag
            \hspace{8mm}
            \leq
            \sum_{ir \in \{ M_{g-1} , M_{g-1} + r ,\ldots, M_{g}- r\}} \hspace{-12mm} \log(g)^3   L^{ ir } 
            +
            \sum_{ir \in \{ M_{g-1} , M_{g-1} + r ,\ldots, M_{g}- r\}} \hspace{-1.5cm} \log(g)^3 L^{\bar{r}_{g+1} +M_g - ir} g^{-1-\eps}  L^{ ir } 
            \\
            & \notag
            \hspace{8mm}
            \leq \log(g)^3  L^{M_{g}}
            +
            \log(g)^3 \left( M_{g}-M_{g-1} \right) L^{\bar{r}_{g+1} + M_{g}} g^{-1-\eps} 
            \\
            &
            \hspace{8mm}
            \leq \log(g)^3  L^{M_{g}}
            +
            \log(g)^5 L^{ M_{g+1} } g^{-1-\eps} 
            =
            L^{M_{g+1}} \left( \log(g)^3 L^{-\bar{r}_{g+1} } + \log(g)^5 g^{-1-\eps} \right),
        \end{align}
        where we used the assumptions $M_{g}-M_{g-1}=\bar{r}_g \leq \log(g)^2$ in the last inequality. If $\varpi$ is good or mediocre, but has a defect at level $n \in \{M_g, M_{g}+r,\ldots, M_{g+1}-r \}$, then $\# \operatorname{Def}_{M_g + kr} (\varpi) \leq N$ when $M_g + kr \in \{M_{g}, M_{g}+r,\ldots, n\}$ and $\# \operatorname{Def}_{M_g + kr} (\varpi) = 0$ otherwise, so that
        \begin{multline}\label{eq:U lower bound}
            |U|  \geq L^{\bar{r}_{g+1}} - \hspace{-3mm} \sum_{kr \in \{0,r,\ldots,\bar{r}_{g+1}-r \} } \hspace{-3mm} \# \operatorname{Def}_{M_g + kr}(\varpi) L^{kr} 
            \\
            \geq
            L^{\bar{r}_{g+1}} - \hspace{-3mm} \sum_{ \substack{ kr \in \{0,r,\ldots,\bar{r}_{g+1}-r \} : \\ M_{g}+kr \leq n} } \hspace{-3mm} N L^{kr} 
            \geq
            L^{\bar{r}_{g+1}} - 2 N L^{n-M_{g}}
        \end{multline}
        Combining inequalities \eqref{eq:inequalities 1}, \eqref{eq:defects sum}, and \eqref{eq:U lower bound}, we get that
        \begin{align*}
            & S(\varpi) \overset{\eqref{eq:inequalities 1}}{\geq}
            |U| \theta_g L^{M_g} - 2N
            \sum_{ir \in \{ M_{g-1} , M_{g-1} + r ,\ldots, M_{g}- r\}} \sum_{ \substack{ \sigma \in U : \\ \operatorname{Def}(\sigma)=ir } } L^{ir}
            \\
            & \hspace{22mm}
            \overset{\eqref{eq:defects sum}}{\geq} |U| \theta_g L^{M_g} - 2N L^{M_{g+1}} \left( \log(g)^3 L^{-\bar{r}_{g+1} } + \log(g)^5 g^{-1-\eps} \right)
            \\
            & \hspace{22mm}
            \overset{\eqref{eq:U lower bound}}{\geq} \left( L^{\bar{r}_{g+1}} - 2N L^{n - M_g} \right) \theta_g L^{M_g} - 2N L^{M_{g+1}} \left( \log(g)^3 L^{-\bar{r}_{g+1} } + \log(g)^5 g^{-1-\eps} \right)
            \\
            & \hspace{22mm}
            \geq \theta_g L^{M_{g+1}} - 2N L^{n}  - 2N L^{M_{g+1}} \left( \log(g)^3 L^{-\bar{r}_{g+1} } + \log(g)^5 g^{-1-\eps} \right)
            \\
            & \hspace{22mm}
            = \left( \theta_g - 2N \left( \log(g)^3 L^{-\bar{r}_{g+1} } + \log(g)^5 g^{-1-\eps} \right) \right) L^{M_{g+1}} - 2N L^{n}
            \\
            & \hspace{22mm}
            =
            \theta_{g+1} L^{M_{g+1}} - 2N L^{n} ,
        \end{align*}
        which shows the first inequality in \eqref{part B ineq1} and \eqref{part B ineq2}, respectively. For the second inequality, we use that $\theta_{g+1} \geq \theta \geq 0.99$ and $3NL^{-r} < \zeta$ by assumption, and directly get that
        \begin{align*}
            & S(\varpi) \geq  \theta_{g+1} L^{M_{g+1}} - 2N L^{n} \geq \theta_{g+1} L^{M_{g+1}} - 2N L^{M_{g+1}-r}
            \\
            &
            \hspace{3cm}
            = \theta_{g+1} L^{M_{g+1}} \left( 1 - \frac{2N}{\theta_{g+1}} L^{-r} \right)
            \geq
            \theta_{g+1} L^{M_{g+1}} \left( 1 - 3N L^{-r} \right)
            \\
            &
            \hspace{3cm}
            \geq
            (1-\zeta)\theta_{g+1} L^{M_{g+1}},
        \end{align*}
        completing the proof.
\end{proof}

Next, we introduce the following proposition, which gives quantitative decay estimates on the probability that blocks $\varpi \in \Lambda_{M_{g},\infty}$ are bad or mediocre.

\begin{proposition}\label{prop:bad mediocre prop}
    Assume that
	\begin{align*}
		& \p_{q,J} \left(\Lambda_{M_{g}} \textnormal{ is bad}\right) \leq g^{-1-\eps},\\
        & \p_{q,J} \left(\Lambda_{M_{g}} \textnormal{ is $\gamma$-mediocre}\right) \leq g^{-(1-\gamma)-2\eps} , \ \textnormal{ for all $\gamma \in \{L^{-r},\ldots,1-L^{-r}\}$, and} \\
		& \p_{q,J} \left(\operatorname{Def}(\Lambda_{M_{g}}) = m\right) \leq L^{M_{g}-m} (g-1)^{-1-\eps} \text{ for all } m \in \left\{M_{g-1}, \ldots, M_{g}-r \right\}.
	\end{align*}
	Then, for all large enough $g \in \N$, all $\gamma \in \{L^{-r},\ldots,1-L^{-r}\}$, and all $k \in \{r,2r,\ldots,\bar{r}_{g+1} \}$ one has
	\begin{align}\label{eq:bad}
		& \p_{q,J} \left(\Lambda_{M_{g}+k} \textnormal{ is bad}\right) \leq (g+1)^{-1-\eps},\\
		\label{eq:mediocre} & \p_{q,J} \left(\Lambda_{M_{g}+k} \textnormal{ is $\gamma$-mediocre}\right) \leq (g+1)^{-(1-\gamma)-2\eps},
	\end{align}
    and, for the block $\Lambda_{M_{g+1}}(0)$ one has
    \begin{equation}\label{eq:defect}
         \p_{q,J} \left(\operatorname{Def}(\Lambda_{M_{g+1}}) = m\right) \leq L^{M_{g+1}-m} g^{-1-\eps} \text{ for all } m \in \left\{M_{g}, \ldots, M_{g+1}-r \right\},
    \end{equation}
\end{proposition}

Before we go to the proof of Proposition \ref{prop:bad mediocre prop}, we introduce the following lemma, which consists of four inequalities. The proof of the proposition then follows from a straightforward application of these inequalities.

\begin{lemma}\label{lem:longinequali}
    Let $\varpi \in \Lambda_{M_g+jr,\infty}$ for some $jr \in \{r,\ldots, \bar{r}_{g+1} - r\}$. Then
    \begin{align}\label{eq:prop57 ineq1}
        &\notag \p_{q,J} \left( \varpi \textnormal{ bad} \right)  
        \\
        & \notag 
        \leq
        \sum_{\varphi, \varphi^\prime \in \chi_{r}(\varpi)} \p_{q,J} \left( \varphi \textnormal{ not good} \right) \p_{q,J} \left( \varphi^\prime \textnormal{ not good} \right) 
        \\
        & \notag \hspace{7mm}
        + \sum_{\varphi \in \chi_{r}(\varpi)} \p_{q,J} \left( \varphi \textnormal{ not good} \right) \p_{\bar{J}} \left( G^{\mathrm{bl}} \left( \chi_{r}(\varpi) \setminus \{\varphi \} \right) \textnormal{ not connected} \right)
        \\
        & \notag \hspace{7mm}
        + \sum_{\gamma \in \{L^{-r}, \ldots, 1- L^{-r} \} }  \sum_{\varphi \in \chi_{r}(\varpi)} \p_{q,J} \left( \varphi \textnormal{ is $\gamma$-mediocre} \right) \sum_{A \subset \chi_{r}(\varphi) : |A|\geq \gamma L^r } \p_{\bar{J}} \left( A \nsim \chi_{r}(\varpi) \setminus \{ \varphi \} \right) 
        \\
        & \hspace{7mm}
        +
        \sum_{m \in \{0,r,\ldots, jr - r \} } \left( L^{jr - m} \p_{q,J} \left( \Lambda_{M_{g}+m}(0) \textnormal{ bad} \right) \right)^N
    \end{align}
    and, for every $\gamma \in \{L^{-r},\ldots,1-L^{-r} \}$,
    \begin{align}\label{eq:prop57 ineq2}
        \p_{q,J} \left( \varpi \textnormal{ is $\gamma$-mediocre} \right)
        \leq
        \p_{\bar{J}} \left( G^{\mathrm{bl}} \left( \chi_{r}(\varpi) \right) \textnormal{ has a cluster of density at most $\gamma$} \right).
    \end{align}
    If $\varpi \in \Lambda_{M_{g+1},\infty}$, then, for every $\gamma \in \{ L^{-r},\ldots, 1-L^{-r} \}$,
    \begin{equation}\label{eq:prop57 ineq3}
        \p_{q,J} \left( \varpi \textnormal{ is $\gamma$-mediocre} \right)
        \leq
        \p_{\bar{J}} \left( G^{\mathrm{bl}} \left( \chi_{r}(\varpi) \right) \textnormal{ has a cluster of density at most $\gamma$} \right),
    \end{equation}
    and
    \begin{align}
        \label{eq:prop57 ineq5}
        &\p_{q,J} \left( \varpi \textnormal{ bad} \right)  
        \\
        & \notag \leq
        \sum_{\varphi, \varphi^\prime \in \chi_{r}(\varpi)} \p_{q,J} \left( \varphi \textnormal{ not good} \right) \p_{q,J} \left( \varphi^\prime \textnormal{ not good} \right) 
        \\
        & \notag \hspace{4mm}
        + \sum_{\varphi \in \chi_{r}(\varpi)} \p_{q,J} \left( \varphi \textnormal{ not good} \right) \p_{\bar{J}} \left( G^{\mathrm{bl}} \left( \chi_{r}(\varpi) \setminus \{\varphi \} \right) \textnormal{ not connected} \right)
        \\
        & \notag \hspace{4mm}
        + \sum_{\gamma \in \{L^{-r}, \ldots, 1- L^{-r} \} }  \sum_{\varphi \in \chi_{r}(\varpi)} \p_{q,J} \left( \varphi \textnormal{ is $\gamma$-mediocre} \right) \sum_{A \subset \chi_{r}(\varphi) : |A|\geq \gamma L^r } \p_{\bar{J}} \left( A \nsim \chi_{r}(\varpi) \setminus \{ \varphi \} \right) 
        \\
        & \notag \hspace{4mm}
        +
        \sum_{m \in \{M_{g}, M_{g}+r,\ldots, M_{g+1} - r \} }  \hspace{-12mm} \left( L^{M_{g+1} - m} \p_{q,J} \left( \Lambda_{m}(0) \textnormal{ bad} \right) \right)^N
        \\
        & \notag \hspace{4mm}
        +
        \sum_{m \in \{M_{g-1}, M_{g-1}+r,\ldots, M_{g} - r \} } \hspace{-12mm} \exp \left( e L^{\bar{r}_{g+1}} L^{M_{g}-m} \p_{q,J} \left( \Lambda_{m}(0) \textnormal{ bad} \right) - \log(g)^3 \lceil L^{\bar{r}_{g+1} + M_{g} - m} g^{-1-\eps} \rceil \right) .
    \end{align}

\end{lemma}

\begin{proof}
    The proof follows by taking a union bound over the different situations in which a block $\varpi \in \Lambda_{M_g + jr,\infty}$ can be bad or $\gamma$-mediocre, respectively. Formally, for a block $\varpi \in \Lambda_{M_{g}+jr,\infty}$ with $M_{g}+jr \in \{M_g+r,\ldots,M_{g}+\bar{r}_{g+1}-r\}$ we have that
    \begin{align}\label{1.0 bad}
        & \p_{q,J} (\varpi \text{ bad}) \leq \sum_{\textnormal{X} \in \{\textnormal{B,C,D}\}} \p_{q,J} \left(\varpi \text{ declared bad in step (1.1X)} \right)
        \\
        & \label{bcd bad}
        \hspace{6cm} +  \p_{q,J} (\varpi \text{ declared bad in step (1.0)})
    \end{align}
    We start with an upper bound on the probability $\p_{q,J} (\varpi \text{ declared bad in step (1.0)})$. First note that for disjoint blocks $\varphi_1,\ldots,\varphi_N$ the events $\left(\{\varphi_i \text{ bad}\} \right)_{i \in [N]}$ are independent, since the event $\{\varphi_i \text{ bad}\}$ only depends on edges with both endpoints in $\varphi_i$.
    Thus, we get that
    \begin{align*}
        &\p_{q,J} (\varpi \text{ declared bad in step (1.0)})  \leq \sum_{m \in \{0,r,\ldots,jr-r\}}  \p_{q,J} \left( \#\operatorname{Def}_{M_g+m}(\varpi) > N \right)
        \\
        & \hspace{40mm}
        \leq \sum_{m \in \{0,r,\ldots,jr-r\}} \p_{q,J} \left( \bigcup_{ \substack{\varphi_1,\ldots,\varphi_N \in \Lambda_{M_g+m,\infty} : \\ \varphi_1,\ldots,\varphi_N \subset \varpi , |\{\varphi_1,\ldots,\varphi_N\}| = N} } \{ \varphi_i \text{ bad for all } i \} \right)
        \\
        & \hspace{40mm}
        \leq \sum_{m \in \{0,r,\ldots,jr-r\}}  \sum_{ \substack{\varphi_1,\ldots,\varphi_N \in \Lambda_{M_g+m,\infty} : \\ \varphi_1,\ldots,\varphi_N \subset \varpi, |\{\varphi_1,\ldots,\varphi_N\}| = N} }  \p_{q,J} \left( \Lambda_{M_g+m}(0) \text{ bad} \right)^N
        \\
        & \hspace{40mm}
        \leq 
        \sum_{m \in \{0,r,\ldots,jr-r\}} \left(L^{jr-m}  \p_{q,J} \left( \Lambda_{M_g+m}(0) \text{ bad} \right) \right)^N .
    \end{align*}
    Next, we upper bound the term in \eqref{bcd bad} when X $=$ B. Using again that the events $\left( \{ \varphi \text{ not good}\} \right)_{\varphi \in \chi_r(\varpi)}$ are independent, we get that
    \begin{align*}
        & \p_{q,J} \left(\varpi \text{ declared bad in step (1.1B)} \right) 
        \leq
        \sum_{\varphi, \varphi^\prime \in \chi_r(\varpi) : \varphi \neq \varphi^\prime} \p_{q,J} \left( \varphi, \varphi^\prime \text{ not good} \right)
        \\
        &
        \leq
        \sum_{\varphi, \varphi^\prime \in \chi_r(\varpi) : \varphi \neq \varphi^\prime} \p_{q,J} \left( \varphi \text{ not good} \right) \p_{q,J} \left(  \varphi^\prime \text{ not good} \right)
    \end{align*}
    Next, we study the term in \eqref{bcd bad} for X $=$ C. If there exists exactly one $\varphi \in \chi_r(\varpi)$ that is bad, then $|K_{\max}(\varphi^\prime)| \geq (1-\zeta)^2 \theta L^{M_g+jr-r}$ for all $\varphi^\prime \in \chi_{r}(\varpi)\setminus \{\varphi\}$. In particular, one has for all good blocks $\varphi^\prime, \varphi^\star \in \chi_{r}(\varpi) \setminus \{\varphi \}$, say with $\varphi^\prime = \Lambda_{M_g+jr-r}(u), \varphi^\star = \Lambda_{M_g+jr-r}(v)$, that
    \begin{align}\label{eq:mx bl domination}
        & \notag \p_{q,J} \left( K_{\max}(\varphi^\prime) \sim K_{\max}(\varphi^\star) \mid \varphi^\prime, \varphi^\star \text{ good} \right) 
        \\
        & \notag
        \hspace{40mm}
        \geq 
        1 - \exp \left(  - \left( (1-\zeta)^2 \theta L^{M_g+jr-r} \right)^2 J(\{u,v\}) \right)
        \\
        &
        \hspace{40mm}
        \geq 
        1 - \exp \left(  - \left( L^{M_g+jr-r} \right)^2 \bar{J}(\{u,v\}) \right) 
        =
        \p_{\bar{J}} \left( \varphi^\prime \sim \varphi^\star \right) .
    \end{align}
    So conditioned on the event that all $\varphi^\prime \in \chi_r(\varpi)\setminus \{\varphi\}$ are good, the graph $G^{\mathrm{mx}} \left( \chi_r(\varpi) \setminus \{\varphi \} \right)$ sampled from $\p_{q,J}$ stochastically dominates the graph $G^{\mathrm{bl}} \left( \chi_r(\varpi) \setminus \{\varphi \} \right)$ sampled from $\p_{\bar{J}}$, which implies that
    \begin{multline}\label{mx bl ineq}
         \p_{q,J} \left( G^{\mathrm{mx}} \left( \chi_r(\varpi) \setminus \{\varphi \} \right) \text{ not connected} \;\big|  \;\text{all } \varphi^\prime \in \chi_r(\varpi) \setminus \{\varphi\} \text{ are good} \right) \\
         \leq
         \p_{\bar{J}} \left( G^{\mathrm{bl}} \left( \chi_r(\varpi) \setminus \{\varphi \} \right) \text{ not connected} \right).
    \end{multline}
    Since the events $\{\varphi \text{ bad}\}$ and $\{G^{\mathrm{mx}}\left(\chi_r(\varpi) \setminus \{\varphi\} \right) \text{ not connected}\}$ are independent, we get that
    \begin{align*}
        & \p_{q,J} \left(\varpi \text{ declared bad in step (1.1C)} \right)
        \\
        &\hspace{10mm}
        =
        \p_{q,J} \Bigg( \bigcup_{\varphi \in \chi_r(\varpi)} \Big\{ \varphi \text{ bad}, G^{\mathrm{mx}} \left( \chi_r(\varpi) \setminus \{\varphi \} \right) \text{ not connected},
        \\
        &
        \hspace{75mm}
          \text{ all } \varphi^\prime \in \chi_r(\varpi) \setminus \{\varphi\} \text{ are good}\Big\} \Bigg)
        \\
        &
        \hspace{10mm}
        \overset{\eqref{mx bl ineq}}{\leq} \sum_{\varphi \in \chi_r(\varpi)}
        \p_{q,J} \left( \varphi \text{ bad} \right) \cdot
        \\
        &
        \hspace{30mm} \p_{q,J} \left( G^{\mathrm{mx}} \left( \chi_r(\varpi) \setminus \{\varphi \} \right) \text{ not connected} \big|  \text{all } \varphi^\prime \in \chi_r(\varpi) \setminus \{\varphi\} \text{ are good} \right) 
        \\
        &
        \hspace{10mm}
        \leq \sum_{\varphi \in \chi_r(\varpi)}
        \p_{q,J} \left( \varphi \text{ bad} \right) \p_{\bar{J}} \left( G^{\mathrm{bl}} \left( \chi_r(\varpi) \setminus \{\varphi \} \right) \text{ not connected} \right) .
    \end{align*}
    Finally, we need to consider the term in \eqref{bcd bad} when X $=$ D. Say that $\varphi$ is the unique mediocre element in $\chi_r(\varpi)$. If $\varpi$ is declared bad in step (1.1D), then either $G^{\mathrm{mx}}\left(\chi_r(\varpi)\setminus\{\varphi\}\right)$ is not connected or there exist $\gamma \in \{L^{-r},\ldots,1-L^{-r} \}$ and a connected component $A \subsetneq G^{\mathrm{mx}} \left( \chi_{r}(\varphi) \right)$ with $|A| \geq \gamma L^r$ such that $\varphi$ is $\gamma$-mediocre, $K_{\max}(\sigma) \nsim K_{\max}(\varphi^\prime)$ for all $\sigma \in A, \varphi^\prime \in \chi_{r}(\varpi)\setminus \{\varphi \}$, and all $\sigma \in A, \varphi^\prime \in \chi_r(\varpi)\setminus \{\varphi\}$ are good. Call the latter event $\mathcal{D}_{\varphi, A}$, i.e.,
    \begin{align*}
        &\mathcal{D}_{\varphi, A} \coloneqq \bigcap_{\sigma \in A} \left\{ \sigma \text{ good} \right\} \cap \left\{ \varphi \text{ is } \gamma\text{-mediocre} \right\} 
        \\
        &
        \hspace{3cm} 
        \cap
        \bigcap_{\sigma \in A, \varphi^\prime \in \chi_r(\varpi)\setminus \{\varphi\}} \left\{
        K_{\max}(\sigma) \nsim K_{\max}(\varphi^\prime) \right\}    \cap \bigcap_{\varphi^\prime \in \chi_r(\varpi)\setminus \{\varphi\} } \left\{ \varphi^\prime \text{ good} \right\} .
    \end{align*}
    We denote the four events in the definition of $\cD_{\varphi, A}$ by $\cA_1,\ldots,\cA_4$, i.e.,
    \begin{align*}
        \cA_1 & = \bigcap_{\sigma \in A} \left\{ \sigma \text{ good} \right\} ,
        \\
        \cA_2 & = \left\{ \varphi \text{ is } \gamma\text{-mediocre} \right\} ,
        \\
        \cA_3 & =
        \bigcap_{\sigma \in A, \varphi^\prime \in \chi_r(\varpi)\setminus \{\varphi\}} \left\{
        K_{\max}(\sigma) \nsim K_{\max}(\varphi^\prime) \right\}, \text{ and}  
        \\
        \cA_4 & =
        \bigcap_{\varphi^\prime \in \chi_r(\varpi)\setminus \{\varphi\} } \left\{ \varphi^\prime \text{ good} \right\} .
    \end{align*}
    Let $\cF$ be the $\sigma$-algebra generated by all the edges with both endpoints in $\psi$ for some $\psi \in \chi_r(\varpi)$, i.e.,
    \begin{equation*}
        \cF = \sigma \left( \omega(\{x,y\}) : x,y \in \psi , \psi \in \chi_r(\varpi)  \right).
    \end{equation*}
    The events $\cA_1, \cA_2$, and $\cA_4$ are measurable with respect to $\cF$, so that
    \begin{equation}\label{conditional}
        \p_{q,J} \left( \cD_{\varphi, A} \right) 
        = 
        \E_{q,J} \left[ \p_{q,J} \left( \cA_1 \cap \cA_2 \cap \cA_3 \cap \cA_4 \big| \cF \right) \right]
        =
        \E_{q,J} \left[ \mathbbm{1}_{ \cA_1 \cap \cA_2 \cap \cA_4} \p_{q,J} \left( \cA_3 \big| \cF \right) \right] .
    \end{equation}
    For two good blocks $\sigma = \Lambda_{M_g+(j-2)r}(u) \in A$ and $\varphi^\prime = \Lambda_{M_g+(j-1)r}(v) \in \chi_{r}(\varpi)\setminus \{\varphi\}$, we have that $|K_{\max}(\sigma)| \geq (1-\zeta)^2 \theta L^{M_g+(j-2)r}$ and $|K_{\max}(\varphi^\prime)| \geq (1-\zeta)^2 \theta L^{M_g+(j-1)r}$, so that for all $\omega \in \cA_1 \cap \cA_4$
    \begin{align*}
        & \p_{q,J} \left( K_{\max}(\sigma) \nsim K_{\max}(\varphi^\prime) | \cF \right) (\omega)
        \\
        & \hspace{4cm} \leq \exp \left( - (1-\zeta)^2 \theta L^{M_g+(j-2)r} (1-\zeta)^2 \theta L^{M_g+(j-1)r} J(u,v) \right)
        \\
        & \hspace{4cm}
        \leq
        \exp \left( - L^{M_g+(j-2)r} L^{M_g+(j-1)r} \bar{J}(u,v) \right)
        =
        \p_{\bar{J}} \left( \sigma \nsim \varphi^\prime \right).
    \end{align*}
    Since the connections between different block $\sigma \in A$ and $\varphi^\prime \in \chi_r(\varphi) \setminus \{\varphi \}$ are independent, this also implies that for all $\omega \in \cA_1 \cap \cA_4$
    \begin{align*}
        & \p_{q,J} \left( \bigcap_{\sigma \in A, \varphi^\prime \in \chi_r(\varpi)\setminus \{\varphi\}} \hspace{-4mm} \left\{
        K_{\max}(\sigma) \nsim K_{\max}(\varphi^\prime) \right\} \Big| \cF \right) (\omega)
        \\
        &
        \hspace{4cm}
        \leq 
        \p_{\bar{J}} \left( \bigcap_{\sigma \in A, \varphi^\prime \in \chi_r(\varpi)\setminus \{\varphi\}} \hspace{-4mm} \left\{
        \sigma \nsim \varphi^\prime \right\} \right)
        =
        \p_{\bar{J}} \left(  A \nsim \chi_r(\varpi)\setminus \{\varphi\} \right) .
    \end{align*}
    Inserting this inequality into \eqref{conditional}, we get that
    \begin{multline*}
        \p_{q,J} \left( \cD_{\varphi, A} \right) 
        = 
        \E_{q,J} \left[ \mathbbm{1}_{ \cA_1 \cap \cA_2 \cap \cA_4} \p_{q,J} \left( \cA_3 \big| \cF \right) \right]
        \leq 
        \p_{\bar{J}} \left(  A \nsim \chi_r(\varpi)\setminus \{\varphi\} \right)
        \E_{q,J} \left[ \mathbbm{1}_{ \cA_1 \cap \cA_2 \cap \cA_4} \right] \\
        \leq 
        \p_{\bar{J}} \left(  A \nsim \chi_r(\varpi)\setminus \{\varphi\} \right)
        \E_{q,J} \left[ \mathbbm{1}_{ \cA_2} \right] 
        =
        \p_{\bar{J}} \left(  A \nsim \chi_r(\varpi)\setminus \{\varphi\} \right)
        \p_{q,J} \left( \varphi \text{ is $\gamma$-mediocre} \right) 
        .
    \end{multline*}
    Remember that if $\varpi$ is declared bad in step (1.1D), then either there exists $\varphi \in \chi_r(\varpi)$ that is mediocre, all $\varphi^\prime \in \chi_r(\varpi)\setminus\{\varphi\}$ are good, and $G^{\mathrm{mx}}\left(\chi_r(\varpi)\setminus\{\varphi\}\right)$ is not connected, or there exists $\varphi,A$ such that the event $\cD_{\varphi,A}$ occurs. We bounded the probability of the event $\cD_{\varphi,A}$ above. For the first event, observe that for every $\varphi \in \chi_r(\varpi)$ we have that
    \begin{align*}
    	& \p_{q,J} \left( \varphi \text{ mediocre}, \bigcap_{\varphi^\prime \in \chi_r(\varpi)\setminus \{\varphi\}} \hspace{-3mm} \left\{\varphi^\prime \text{ good}\right\} , G^{\mathrm{mx}}\left(\chi_r(\varpi)\setminus\{\varphi\}\right) \text{ not connected} \right)
        \\
        &
        \hspace{7mm}
        \leq
        \p_{q,J} \left( \varphi \text{ mediocre} \right) \p_{q,J} \left( G^{\mathrm{mx}}\left(\chi_r(\varpi)\setminus\{\varphi\}\right) \text{ not connected} \Bigg| \bigcap_{\varphi^\prime \in \chi_r(\varpi)\setminus \{\varphi\}} \hspace{-3mm} \left\{\varphi^\prime \text{ good}\right\} \right)
        \\
        &
        \overset{\eqref{mx bl ineq}}{\leq} p_{q,J} \left( \varphi \text{ mediocre} \right) \p_{\bar{J}} \left( G^{\mathrm{bl}}\left(\chi_r(\varpi)\setminus\{\varphi\}\right) \text{ not connected} \right) .
    \end{align*}
    Thus, we can bound the probability that $\varpi$ gets declared bad in step (1.1D) via a union bound over all $\varphi \in \chi_r(\varphi)$, $\gamma \in \{L^{-r},\ldots,1-L^{-r}\}$, and sets $A \subset \chi_r(\varphi) $ with $ |A| \geq \gamma L^r$. We get that
    \begin{align*}
        & \p_{q,J} \left(\varpi \text{ declared bad in step (1.1D)} \right)\\
        &
        \leq
        \sum_{\varphi \in \chi_r(\varpi)} \p_{q,J} \left( \varphi \text{ mediocre}, \bigcap_{\varphi^\prime \in \chi_r(\varpi)\setminus \{\varphi\}} \hspace{-3mm} \left\{\varphi^\prime \text{ good}\right\} , G^{\mathrm{mx}}\left(\chi_r(\varpi)\setminus\{\varphi\}\right) \text{ not connected} \right)
        \\
        &
        \hspace{47mm} +
        \sum_{\gamma \in \{L^{-r},\ldots,1-L^{-r}\} } \sum_{\varphi \in \chi_r(\varpi)}  \sum_{A \subset \chi_r(\varphi) : |A| \geq \gamma L^r} \p_{q,J} \left( \mathcal{D}_{\varphi, A} \right)
        \\
        &
        \leq \sum_{\varphi \in \chi_r(\varpi)} \p_{q,J} \left( \varphi \text{ mediocre} \right) \p_{\bar{J}} \left( G^{\mathrm{bl}}\left(\chi_r(\varpi)\setminus\{\varphi\}\right) \text{ not connected} \right)
        \\
        & \hspace{5mm}
        +
        \sum_{\gamma \in \{L^{-r}, \ldots, 1- L^{-r} \} }  \sum_{\varphi \in \chi_{r}(\varpi)} \p_{q,J} \left( \varphi \text{ is $\gamma$-mediocre} \right) \sum_{A \subset \chi_{r}(\varphi) : |A|\geq \gamma L^r } \p_{\bar{J}} \left( A \nsim \chi_{r}(\varpi) \setminus \{ \varphi \} \right) .
    \end{align*}
    Combining the four upper bounds now established for the four terms in \eqref{1.0 bad} and \eqref{bcd bad} and using that $\p_{q,J} \left( \varphi \text{ not good} \right) = \p_{q,J} \left( \varphi \text{ bad} \right) + \p_{q,J} \left( \varphi \text{ mediocre} \right)$ implies inequality \eqref{eq:prop57 ineq1}, the first inequality in Proposition \ref{lem:longinequali}.
    If $\varpi \in \Lambda_{M_{g+1},\infty}$, there is one other way how $\varpi$ can be bad, namely if $\varpi$ is declared bad in step (2.) of the algorithm above. That is, if there exists $m \in \{M_{g-1},M_{g-1}+r,\ldots,M_{g}\}$ such that there are at least $\log(g)^3\lceil L^{\bar{r}_{g+1}+M_g-m} g^{-1-\eps} \rceil$ many blocks $\sigma \in \Lambda_{M_g,\infty}$ with $\sigma \subseteq \varphi$ for which $\operatorname{Def}(\sigma) = m$. First, note that for a block $\sigma \in \Lambda_{M_g,\infty}$ the probability that $\operatorname{Def}(\sigma)=m$ is bounded from above by
    \begin{equation}\label{def sigma m bound}
        \p_{q,J} \left( \operatorname{Def}(\sigma) = m \right) \leq \sum_{\psi \in \Lambda_{m,\infty} : \psi \subseteq \sigma} \p_{q,J} \left( \psi \text{ bad} \right) = L^{M_{g}-m}  \p_{q,J} \left( \Lambda_{m}(0) \text{ bad} \right) .
    \end{equation}
    The event $\{\operatorname{Def}(\sigma)=m\}$ depends only on edges with both ends in $\sigma$, so that for different blocks $\sigma \in \Lambda_{M_g,\infty}$, the events $\{\operatorname{Def}(\sigma)=m\}$ are independent. We write
    \begin{align*}
        \left| \left\{ \sigma \in \Lambda_{M_g,\infty} : \sigma \subseteq \varpi , \operatorname{Def}(\sigma)=m \right\} \right| = \sum_{\sigma \in \Lambda_{M_g,\infty} : \sigma \subseteq \varpi} \mathbbm{1}_{\operatorname{Def}(\sigma)=m} 
    \end{align*}
    and get from Markov's inequality that
    \begin{align*}
        & \p_{q,J} \left(  \left| \left\{ \sigma \in \Lambda_{M_g,\infty} : \sigma \subseteq \varpi , \operatorname{Def}(\sigma)=m \right\} \right| > \log(g)^3\lceil L^{\bar{r}_{g+1}+M_g-m} g^{-1-\eps} \rceil \right)
        \\
        & \hspace{9mm}
        \leq \p_{q,J} \left( \exp \left( \sum_{\sigma \in \Lambda_{M_g,\infty} : \sigma \subseteq \varpi} \mathbbm{1}_{\operatorname{Def}(\sigma)=m} \right) > \exp \left( \log(g)^3\lceil L^{\bar{r}_{g+1}+M_g-m} g^{-1-\eps} \rceil \right) \right)
        \\
        & \hspace{9mm}
        \leq \E_{q,J} \left[ \exp \left( \sum_{\sigma \in \Lambda_{M_g,\infty} : \sigma \subseteq \varpi} \mathbbm{1}_{\operatorname{Def}(\sigma)=m} \right) \right] \exp \left( - \log(g)^3\lceil L^{\bar{r}_{g+1}+M_g-m} g^{-1-\eps} \rceil \right)
        \\
        & \hspace{9mm}
        =
         \prod_{\sigma \in \Lambda_{M_g,\infty} : \sigma \subseteq \varpi} \E_{q,J} \left[ \exp \left( \mathbbm{1}_{\operatorname{Def}(\sigma)=m} \right) \right] \exp \left(- \log(g)^3\lceil L^{\bar{r}_{g+1}+M_g-m} g^{-1-\eps} \rceil \right)
        \\
        & \hspace{9mm}
        \overset{\eqref{eq:elementary markov}}{\leq}
        \prod_{\sigma \in \Lambda_{M_g,\infty} : \sigma \subseteq \varpi} \exp \left( e \p_{q,J} \left( \operatorname{Def}(\sigma)=m \right) \right) \exp \left( -\log(g)^3\lceil L^{\bar{r}_{g+1}+M_g-m} g^{-1-\eps} \rceil \right)
        \\
        & \hspace{9mm}
        \overset{\eqref{def sigma m bound}}{\leq}
        \prod_{\sigma \in \Lambda_{M_g,\infty} : \sigma \subseteq \varpi} \exp \left( e L^{M_g-m} \p_{q,J} \left( \Lambda_m(0) \text{ bad} \right) \right) \exp \left(- \log(g)^3\lceil L^{\bar{r}_{g+1}+M_g-m} g^{-1-\eps} \rceil \right)
        \\
        & \hspace{9mm}
        =
        \exp \left( e L^{\bar{r}_{g+1}} L^{M_g-m} \p_{q,J} \left( \Lambda_m(0) \text{ bad} \right) -  \log(g)^3\lceil L^{\bar{r}_{g+1}+M_g-m} g^{-1-\eps} \rceil \right) .
    \end{align*}
    In the above calculation, we also used the elementary inequality 
    \begin{multline}\label{eq:elementary markov}
        \E_{q,J} \left[ \exp \left( \mathbbm{1}_{\operatorname{Def}(\sigma)=m} \right) \right] = \p_{q,J} \left( \operatorname{Def}(\sigma)=m \right) e  + \left( 1- \p_{q,J} \left( \operatorname{Def}(\sigma)=m \right) \right) \\ \leq \exp \left( e \p_{q,J} \left( \operatorname{Def}(\sigma)=m \right) \right).
    \end{multline}
    A further union bound over all $m \in \{M_{g-1},M_{g-1}+r,\ldots,M_{g}-r\}$ shows that
    \begin{align*}
        &\p_{q,J} \left(\varpi \text{ declared bad in step (2.)} \right)
        \\
        &
        \leq \hspace{-3mm}
        \sum_{m \in \{M_{g-1},M_{g-1}+r,\ldots,M_{g}-r\} }  \hspace{-15mm} \p_{q,J} \left(  \left| \left\{ \sigma \in \Lambda_{M_g,\infty} : \sigma \subseteq \varpi , \operatorname{Def}(\sigma)=m \right\} \right| > \log(g)^3\lceil L^{\bar{r}_{g+1}+M_g-m} g^{-1-\eps} \rceil \right)
        \\
        & \leq \hspace{-3mm}
         \sum_{m \in \{M_{g-1},M_{g-1}+r,\ldots,M_{g}-r\} } \hspace{-15mm} \exp \left( e L^{\bar{r}_{g+1}} L^{M_g-m} \p_{q,J} \left( \Lambda_m(0) \text{ bad} \right) -  \log(g)^3\lceil L^{\bar{r}_{g+1}+M_g-m} g^{-1-\eps} \rceil \right) .
    \end{align*}
    A union bound over all the cases in which $\varpi$ can be declared ((1.0), (1.1B), (1.1C), (1.1D), and (2.)) shows inequality \eqref{eq:prop57 ineq5}. \\
    
    We are left to show inequalities \eqref{eq:prop57 ineq2} and \eqref{eq:prop57 ineq3}. Here, we do not make a distinction, whether $\varpi$ is of the form $\Lambda_{M_g}(u)$ or not. So say that $\varpi \in \Lambda_{M_g+jr}$ with $jr \in \{r,2r,\ldots,\bar{r}_{g+1}\}$. From \eqref{eq:mx bl domination} we get that 
    \begin{align*}
        \p_{q,J} \left( K_{\max}(\varphi^\prime) \sim K_{\max}(\varphi^\star) \mid \varphi^\prime, \varphi^\star \text{ good} \right) \geq
        \p_{\bar{J}} \left( \varphi^\prime \sim \varphi^\star \right) 
    \end{align*}
    for all $\varphi^\prime, \varphi^\star \in \chi_r(\varpi)$. So conditioned on the event that all $\varphi^\prime \in \chi_r(\varpi)$ are good, the graph $G^{\mathrm{mx}} \left( \chi_r(\varpi)\right)$ sampled from $\p_{q,J}$ stochastically dominates the graph $G^{\mathrm{bl}} \left( \chi_r(\varpi) \right)$ sampled from $\p_{\bar{J}}$.
    If $\varpi$ is $\gamma$-mediocre for some $\gamma \in \{L^{-r},\ldots,1-L^{-r}\}$, then all $\varphi \in \chi_r(\varpi)$ need to be good.
    Monotonicity thus implies that
    \begin{align*}
        &\p_{q,J} \left( \varpi \text{ is $\gamma$-mediocre} \right) \leq \p_{q,J} \left( \varpi \text{ is $\gamma$-mediocre} \big| \text{all $\varphi \in \chi_r(\varpi)$ are good} \right) 
        \\
        &
        \hspace{14mm}
        \leq
        \p_{q,J} \left( G^{\mathrm{mx}} \left( \chi_{r}(\varpi) \right) \text{ has a cluster of density at most $\gamma$} \big| \text{all $\varphi \in \chi_r(\varpi)$ are good}\right)
        \\
        & \hspace{14mm}
        \leq
        \p_{\bar{J}} \left( G^{\mathrm{bl}} \left( \chi_{r}(\varpi) \right) \text{ has a cluster of density at most $\gamma$} \right),
    \end{align*}    
    completing the proof.
\end{proof}

We continue with the proof of Proposition \ref{prop:bad mediocre prop}

\begin{proof}[Proof of Proposition \ref{prop:bad mediocre prop}]
    
    We prove Proposition \ref{prop:bad mediocre prop} via induction over $k \in \{ r,2r,\ldots, \bar{r}_{g+1} \}$.

    Let $\varpi = \Lambda_{M_{g}+jr}(0)$, with $jr \in \{r,2r,\ldots,\bar{r}_{g+1}-r\}$ and assume that for all $k \in \{0,r,\ldots,jr-r\}$ and for all $\varphi \in \Lambda_{M_g+k,\infty}$
    \begin{align*}
		& \p_{q,J} \left(\varphi \text{ is bad}\right) \leq g^{-1-\eps}, \text{ and } \p_{q,J} \left(\varphi \text{ is $\gamma$-mediocre}\right) \leq g^{-(1-\gamma)-2\eps} .
	\end{align*}
    By these induction assumptions one has for all $\varphi \in \Lambda_{M_{g}+k,\infty}$ that
    \begin{align*}
        &\p_{q,J} \left( \varphi \text{ not good} \right) \leq \p_{q,J} \left( \varphi \text{ bad} \right) + \sum_{\gamma \in \{L^{-r},\ldots, 1 \} : \gamma \leq \frac{1}{2} } \p_{q,J} \left( \varphi \text{ is $\gamma$-mediocre} \right)
        \\
        &
        \hspace{5cm}
        \leq g^{-1-\eps} + \sum_{\gamma \in \{L^{-r},\ldots, 1 \} : \gamma \leq \frac{1}{2} } g^{-(1-\gamma) -2 \eps }
        \leq L^{r} g^{-\frac{1}{2}-\eps}
    \end{align*}
    Remember that $\bar{J}$ and $\eps$ are such that for all $\varphi \in \chi_r(\varpi)$ and all non-empty subsets $A \subseteq \chi_r(\varphi)$
    \begin{align*}
        & \p_{\bar{J}} \left( G^{\mathrm{bl}} \left( \chi_{r}(\varpi) \setminus \{\varphi \} \right) \text{ not connected} \right) \leq \left( M_g + jr - r \right)^{-\frac{1}{2}-\eps} \leq g^{-\frac{1}{2}-\eps} \text{ and that}
        \\
        &
        \p_{\bar{J}} \left( A \nsim \chi_{r}(\varpi) \setminus \{ \varphi \} \right) \leq \left( M_g + jr - r \right)^{-L^{-r}|A|} \leq g^{-L^{-r}|A|} .
    \end{align*}
    Lemma \ref{lem:longinequali} then implies that
    \begin{align*}
        &\p_{q,J} \left( \varpi \text{ bad} \right) 
        \\
        &
        \leq
        \sum_{\varphi, \varphi^\prime \in \chi_{r}(\varpi)} \p_{q,J} \left( \varphi \text{ not good} \right) \p_{q,J} \left( \varphi^\prime \text{ not good} \right) 
        \\
        &
        \hspace{10mm}
        + \hspace{-3mm} \sum_{\varphi \in \chi_{r}(\varpi)} \p_{q,J} \left( \varphi \text{ not good} \right) \p_{\bar{J}} \left( G^{\mathrm{bl}} \left( \chi_{r}(\varpi) \setminus \{\varphi \} \right) \text{ not connected} \right)
        \\
        &
        \hspace{10mm}
        + \hspace{-3mm} \sum_{\gamma \in \{L^{-r}, \ldots, 1- L^{-r} \} }  \sum_{\varphi \in \chi_{r}(\varpi)} \p_{q,J} \left( \varphi \text{ is $\gamma$-mediocre} \right) \hspace{-2mm} \sum_{A \subset \chi_{r}(\varphi) : |A|\geq \gamma L^r } \hspace{-2mm}\p_{\bar{J}} \left( A \nsim \chi_{r}(\varpi) \setminus \{ \varphi \} \right) 
        \\
        &
        \hspace{10mm}
        + \hspace{-3mm}
        \sum_{m \in \{0,r,\ldots, jr - r \} } \left( L^{jr - m} \p_{q,J} \left( \Lambda_{M_g + m}(0) \text{ bad} \right) \right)^N
        \\
        &
        \leq 
        \sum_{\varphi, \varphi^\prime \in \chi_{r}(\varpi)} \left( L^r g^{-\frac{1}{2}-\eps} \right)^2 + \sum_{\varphi \in \chi_{r}(\varpi)} \left( L^r g^{-\frac{1}{2}-\eps} \right) g^{-\frac{1}{2}-\eps} 
        \\
        &
        \hspace{10mm}
        + \hspace{-3mm}
        \sum_{ \gamma \in \{L^{-r}, \ldots, 1- L^{-r} \} }  \sum_{\varphi \in \chi_{r}(\varpi)} g^{-(1-\gamma)-2\eps} \sum_{ \substack{A \subset \chi_{r}(\varphi) :  |A|\geq\gamma L^r } } g^{-L^{-r}|A|} 
        \\
        &
        \hspace{10mm}
        + \hspace{-3mm}
        \sum_{m \in \{0,r,\ldots, jr - r \} } \left( L^{jr - m} g^{-1-\eps} \right)^N
        \\
        &
        \leq
        L^{4r} g^{-1-2 \eps} + L^{2r} g^{-1-2 \eps} + \sum_{ \gamma \in \{L^{-r}, \ldots, 1- L^{-r} \} } L^r g^{-(1-\gamma)-2\eps} 2^{L^r} g^{-\gamma} + jr \left( L^{\bar{r}_{g+1}} g^{-1-\eps} \right)^N
        \\
        &
        \leq
        \left( L^{4r} + L^{2r} + L^{2r} 2^{L^r} \right) g^{-1-2\eps} + \bar{r}_{g+1} \left( L^r \log(g+1)^{10} (g+1) g^{-1-\eps} \right)^N
        \\
        &
        \leq
        \left( L^{4r} + L^{2r} + L^{2r} 2^{L^r} \right) g^{-1-2\eps} + \log(g)^2 L^{Nr} \log(g+1)^{10N} 2^{N} g^{-\eps N}
        \\
        &
        \leq
        \left( L^{4r} + L^{2r} + L^{2r} 2^{L^r} \right) g^{-1-2\eps} + L^{Nr} \log(g+1)^{10N+2} g^{-10}
        \leq
        (g+1)^{-1-\eps},
    \end{align*}
    where the last inequality holds for all $g$ large enough. In the above calculation, we also used the assumptions $\bar{r}_{g+1} \leq \log(g)^2$, $L^{\bar{r}_{g+1}} \leq L^{r} \log(g+1)^{10}(g+1)$, and $\eps N > 10$. This bounds the probability that $\varpi$ is bad from above. To control the probability that $\varpi$ is mediocre, observe that, again by Lemma \ref{lem:longinequali}, we have that
    \begin{multline*}
        \p_{q,J} \left( \varpi \text{ is $\gamma$-mediocre} \right) \leq
        \p_{\bar{J}} \left( G^{\mathrm{bl}} \left( \chi_{r}(\varpi) \right) \text{ has a cluster of density at most $\gamma$} \right) 
        \\
        \leq \left( M_{g}+jr-r \right)^{-(1-\gamma)-2\eps} \leq \left( g+1 \right)^{-(1-\gamma)-2\eps} .
    \end{multline*}
    Using this inductively, we see that
    \begin{align*}
        & \p_{q,J} \left( \varpi \text{ bad} \right) \leq (g+1)^{-1-\eps} \text{ and }
        \p_{q,J} \left( \varpi \text{ is $\gamma$-mediocre} \right) \leq (g+1)^{-(1-\gamma)-2\eps}
    \end{align*}
    for all $\varpi \in \Lambda_{M_g + k,\infty}$ with $k \in \{r,2r,\ldots,\bar{r}_{g+1}-r\}$. To study the probability that $\varpi$ is bad for $k= \bar{r}_{g+1}$, note that, by Lemma \ref{lem:longinequali}, we need to study the term
    \begin{equation}\label{eq:sum m}
        \sum_{m \in \{M_{g-1}, M_{g-1}+r,\ldots, M_{g} - r \} } \hspace{-9mm} \exp \left( e L^{\bar{r}_{g+1}} L^{M_{g}-m} \p_{q,J} \left( \Lambda_{m}(0) \text{ bad} \right) - \log(g)^3 \lceil L^{\bar{r}_{g+1} + M_{g} - m} g^{-1-\eps} \rceil \right) .
    \end{equation}
    For the term inside the exponent, we have for $g$ large enough that
    \begin{align*}
        & e L^{\bar{r}_{g+1}} L^{M_{g}-m} \p_{q,J} \left( \Lambda_{m}(0) \text{ bad} \right) - \log(g)^3 \lceil L^{\bar{r}_{g+1} + M_{g} - m} g^{-1-\eps} \rceil
        \\
        & \hspace{40mm}
        \leq
        e L^{\bar{r}_{g+1}} L^{M_{g}-m} (g-1)^{-1-\eps} - \log(g)^3 \lceil L^{\bar{r}_{g+1} + M_{g} - m} g^{-1-\eps} \rceil
        \\
        & \hspace{40mm}
        \leq
        2 e L^{\bar{r}_{g+1}} L^{M_{g}-m} g^{-1-\eps} - \log(g)^3 \lceil L^{\bar{r}_{g+1} + M_{g} - m} g^{-1-\eps} \rceil
        \\
        & \hspace{40mm}
        \leq
        \left( 2e - \log(g)^3 \right) \lceil L^{\bar{r}_{g+1} + M_{g} - m} g^{-1-\eps} \rceil
        \leq
         2e - \log(g)^3 .
    \end{align*}
    Inserting this into \eqref{eq:sum m}, we get that
    \begin{align*}
        & \sum_{m \in \{M_{g-1}, M_{g-1}+r,\ldots, M_{g} - r \} } \hspace{-9mm} \exp \left( e L^{\bar{r}_{g+1}} L^{M_{g}-m} \p_{q,J} \left( \Lambda_{m}(0) \text{ bad} \right) - \log(g)^3 \lceil L^{\bar{r}_{g+1} + M_{g} - m} g^{-1-\eps} \rceil \right)
        \\
        & \hspace{40mm}
        \leq
        \sum_{m \in \{M_{g-1}, M_{g-1}+r,\ldots, M_{g} - r \} } \exp \left(  2e - \log(g)^3 \right)
        \\
        & \hspace{40mm}
        \leq
        \bar{r}_{g} \exp \left(  2e - \log(g)^3 \right) \leq
        \log(g)^2 \exp \left(  2e - \log(g)^3 \right) \leq g^{-2},
    \end{align*}
    where the last inequality holds for all large enough $g$. This inequality can be used to upper bound the last expression in the sum \eqref{eq:prop57 ineq5}. The remaining four terms can be bounded in the same way as in the case $jr \in \{r,2r,\ldots,\bar{r}_{g+1} -r \}$ above. Thus we get for all $\varpi \in \Lambda_{M_g+\bar{r}_{g+1},\infty}$
    \begin{equation*}
        \p_{q,J}\left( \varpi \text{ bad} \right) \leq \left( L^{4r} + L^{2r} + L^{2r} 2^{L^r} \right) g^{-1-2\eps} + L^{Nr} \log(g+1)^{10N+2} g^{-10} + g^{-2} \leq (g+1)^{-1-\eps},
    \end{equation*}
    where the last inequality holds for all large enough $g$. The term $\p_{q,J}\left( \varpi \text{ is $\gamma$-mediocre} \right)$ can be estimated as in the case $\varpi \in \Lambda_{M_g+k,\infty}$ with $k \in \{r,2r,\ldots,\bar{r}_{g+1}-r\}$. For $\varpi \in \Lambda_{M_{g+1},\infty}$, one has by a union bound for all $m \in \{M_{g},M_{g}+r,\ldots,M_{g+1}-r\}$ that
    \begin{multline*}
        \p_{q,J} \left( \operatorname{Def}(\varpi) = m \right) \leq \sum_{\varphi \in \Lambda_{m,\infty}: \varphi \subseteq \varpi} \p_{q,J} \left( \varphi \text{ bad} \right)
        \\
        = L^{M_{g+1}-m} \p_{q,J} \left( \Lambda_m(0) \text{ bad} \right)
        \leq L^{M_{g+1}-m} g^{-1-\eps}
    \end{multline*}
    as claimed.
\end{proof}

Finally, we give the proof of Theorem \ref{theo:existence mixed}.

\begin{proof}[Proof of Theorem \ref{theo:existence mixed}]
	
    Take $h^\star \geq g_0$ large enough so that the results of the previous lemmas hold for $m \geq h^\star$. We will now show that for $q\in (0,1)$ and an $N_0, N_1 \in \N$ large enough, the kernel $J$ defined by 
	\begin{align*}
	J(e) = \begin{cases}
	\frac{a \log\log(|e|)}{|e|^{2}} & \text{ for } |e| \geq N_1\\
	N_0 & \text{ for } |e| < N_1
	\end{cases}
	\end{align*}
	satisfies $\p_{q,J} \left( |K_0| = \infty \right)> 0$. Take $N_1 = L^{M_{h^\star}}$ and take $q\in (0,1), N_0 \in \R_{>0}$ large enough so that for $u\in \mathbb{H}_L$
	\begin{align*}
	\p_{q,J} \left( \Lambda_{M_{h^\star}}(u) \text{ not connected} \right) \leq (h^\star)^{-1-\eps}.
	\end{align*}
    Lemma \ref{lem:density evolu} and Proposition \ref{prop:bad mediocre prop} show that for all large enough $g\geq h^\star$
    \begin{align*}
        &\p_{q,J} \left( \left| K_{\max} (\Lambda_{M_g}(0)) \right| \geq (1-\zeta)\theta \left| \Lambda_{M_g}(0) \right| \right)
        \geq
        \p_{q,J} \left( \Lambda_{M_g}(0) \text{ good} \right)
        \\
        &
        \hspace{15mm}
        =
        1-\p_{q,J} \left( \Lambda_{M_g}(0) \text{ is bad} \right) - \sum_{\gamma \in \{L^{-r},\ldots,1-L^{-r}\}} \p_{q,J} \left( \Lambda_{M_g}(0) \text{ is $\gamma$-mediocre} \right)
        \geq \frac{1}{2}.
    \end{align*}
    Using the isometry-invariance of $J$ and $(1-\zeta) \theta \geq 0.5$, we get that
	\begin{multline*}
		\p_{q,J} \left(|K_0| \geq  \frac{L^{M_{g}}}{2} \right)
		\geq
		\p_{q,J} \left( \left|K_0\left(\Lambda_{M_{g}}\right) \right| \geq \frac{L^{M_{g}}}{2} \right)
		\geq \frac{1}{2} \p_{q,J} \left( \left| K_{\max}\left(\Lambda_{M_{g}} (0) \right) \right| \geq \frac{L^{M_{g}}}{2} \right)
        \\
        \geq \frac{1}{2} \p_{q,J} \left( \left| K_{\max} (\Lambda_{M_g}(0)) \right| \geq (1-\zeta)\theta \left| \Lambda_{M_g}(0) \right| \right) \geq \frac{1}{4} > 0
	\end{multline*}
	and letting $g \to \infty$ finally shows that $\p_{q,J} \left(|K_0| = \infty\right) \geq \frac{1}{4} > 0$.
\end{proof}

\section{Proof of the hierarchical analogue of the Imbrie–Newman conjecture}\label{sec:imbrienewman}

It remains to prove that for a proper family of kernels $\cJ$ on $\H_L$ satisfying $\lambda_c(\cJ) \in (0,\infty)$ and $\cJ \approx n^{-2} \log\log (n)$, one has $\theta(\lambda_c)^2 \beta(\lambda_c) = 1$, under the assumption that the limit 
    \begin{equation}\label{eq:betalambdac}
		\beta(\lambda_c) = \lim_{|e|\to \infty} \cJ(\lambda_c,e) \frac{L^{-1} |e|^{2}}{\log\log(|e|)} .
	\end{equation}
    exists.

\begin{proof}[Proof of Theorem \ref{theo:ImbrieNewman}]
    As $\theta(\lambda)$ is continuous from the right (being an infimum of the continuous increasing functions $\p_\lambda (|K_0|\geq m)$), it follows from Proposition \ref{propo:discont} that $\theta(\lambda_c)^2 \beta(\lambda_c) \geq 1$. Thus, we are left to show that this inequality is indeed an equality. Suppose for the sake of contradiction that the limit defined in \eqref{eq:betalambdac} exists and that
    $\theta(\lambda_c)^2 \beta(\lambda_c) > 1$. Choose $\eps \in (0, \theta(\lambda_c)/10)$ and $\beta^\star < \beta(\lambda_c)$ so that $\left(\theta(\lambda_c)-2\eps\right)^2 \beta^\star > 1$. Let $N_0, N_1 \in [3,\infty)$, and $q \in (0,1)$ be large enough that the kernel $J$ defined by 
	\begin{align*}
	J(e) = \begin{cases}
	\frac{(\theta(\lambda_c)-2\eps)^2 \beta^\star L \log\log(|e|) }{|e|^{2}} & \text{ for } |e| \geq N_1 \\
	N_0 & \text{ for } |e| < N_1
	\end{cases}
	\end{align*}
	satisfies $\p_{q,J} \left( |K_0| = \infty \right)> 0$. (The requirement that $N_1 \geq 3$ is only necessary so that $\log\log(|e|) > 0$ for all edges $e$ with $|e|\geq N_1$).
    Such $N_0,N_1\in \N$ and $q\in (0,1)$ exist by \Cref{theo:existence mixed}.
	By Lemma \ref{lem:cluster size restricted}, the probability $\p_{\lambda_c} \left( |K_{\max}(\Lambda_n)| \geq (\theta(\lambda_c)-\eps) |\Lambda_n| \right)$ tends to $1$ as $n \to \infty$. 
    We now define a renormalized percolation configuration by ``contracting'' $n$-blocks $\Lambda_n(u)$ into single points. More precisely, 
    we say that an $n$-block $\Lambda_n(u)$ is  \textbf{good} if there exits an open cluster of density at least $(\theta(\lambda_c)-\eps)$ inside $\Lambda_n(u)$. For each good $n$-block, we pick the largest cluster inside this block, breaking ties by some deterministic rule that depends only on the configuration inside the block. 
    We say that two good $n$-blocks are \textbf{connected} if there exists an open edge between the largest clusters in the two blocks. By taking $n$ large enough, we can ensure that $\p_{\lambda_c} \left(\Lambda_n \text{ good}\right) > q$ and that 
	\begin{multline*}
		\p_{\lambda_c} \left(\Lambda_n(u) \text{ and } \Lambda_n(v) \text{ connected} \mid \Lambda_n(u), \Lambda_n(v) \text{ good}\right)
        \\
		\geq
		1-\exp\left(-(\theta(\lambda_c)-\eps)^2 L^{2n} \frac{ \beta^\star L \log\log(\|u-v\|)}{\|u-v\|^2}\right) 
	\end{multline*}
     for all good $n$-blocks $\Lambda_n(u)$ and $\Lambda_n(v)$.
    If we take $n$ large enough, then the expression inside the exponential can be bounded from below by
	\begin{multline}\label{two cases}
	(\theta(\lambda_c)-\eps)^2 L^{2n} \frac{ \beta^\star L \log\log(\|u-v\|)}{\|u-v\|^2} \\ \geq \begin{cases}
	(\theta(\lambda_c)-\eps)^2  \frac{ \beta^\star L \log\log(L^k)}{L^{2k}} & \text{ if } \|u-v\|> N_1 L^n\\
	2 N_0 & \text{ if } \|u-v\| \leq N_1 L^n
	\end{cases}
	\end{multline}
     for all vertices $u,v$ with $\|u-v\|=L^{k+n}$.
	As the family of kernels $\cJ$ is continuous (and $\eps \leq \theta(\lambda_c)/10)$, we can decrease $\lambda$ slightly from $\lambda=\lambda_c$ to $\lambda < \lambda_c$ while still maintaining that $\p_{\lambda} \left(\Lambda_n \text{ good}\right) > q$ and 
	\begin{multline*}
	\p_{\lambda} \left(\Lambda_n(u) \text{ and } \Lambda_n(v) \text{ connected } \mid \Lambda_n(u), \Lambda_n(v) \text{ good}\right)
    \\
	\geq
	1-\exp\left(-(\theta(\lambda_c)-2\eps)^2 L^{2n} \frac{ \beta^\star L \log\log(\|u-v\|)}{\|u-v\|^2}\right) 
	\end{multline*}
for all good $n$-blocks $\Lambda_n(u)$ and $\Lambda_n(v)$.
Since $\left( \theta(\lambda_c) -2\eps\right)^2 \geq \frac{1}{2} \left( \theta(\lambda_c) - \eps\right)^2$, this implies together with \eqref{two cases} that for all $u,v$ with $\|u-v\|= L^{k+n}$
	\begin{multline*}
	\p_{\lambda} \left(\Lambda_n(u) \text{ and } \Lambda_n(v) \text{ connected } \mid \Lambda_n(u), \Lambda_n(v) \text{ good}\right)
    \\ 
    \geq 
    \begin{cases}
	1-\exp \left( - (\theta(\lambda_c)-2\eps)^2  \frac{ \beta^\star L \log\log(L^k)}{L^{2k}} \right) & \text{ if } \|u-v\|> N_1 L^n\\
	1- \exp \left(-N_0 \right) & \text{ if } \|u-v\| \leq N_1 L^n
	\end{cases}.
	\end{multline*}
	So, in particular, under the measure $\p_{\lambda}$, the graph that results from contracting each $n$-block of the form $\Lambda_n(u)$ and connecting the largest clusters of the good $n$-blocks stochastically dominates the random graph sampled from $\p_{q,J}$. This implies that $\p_{\lambda}\left(|K_0|=\infty\right) > 0$, which contradicts the assumption that $\lambda< \lambda_c$.
\end{proof}

\subsection*{Acknowledgements}
TH is supported by NSF grant DMS-1928930 and a Packard Fellowship for Science and
Engineering.

\footnotesize{
\bibliographystyle{abbrv}
\bibliography{unimodularthesis}
}

\end{document}